\pdfoutput=1

\documentclass{article}

\newcommand{\notetoself}[1]{}

\newcommand{\samcomment}[1]{}
\renewcommand{\samcomment}[1]{{ \color{red}[#1]}} % Uncomment to show notes-to-self in-line in red % Intentional space before, none after

\newcommand{\dOne}{0}
\newcommand{\dTwo}{0}
\newcommand{\dThree}{0}
\newcommand{\s}{0}
\newcommand{\angleOne}{0}
\newcommand{\angleTwo}{0}
\newcommand{\angleThree}{0}

\usepackage{amssymb}
\usepackage{amsmath}

\usepackage{tabularx}

\usepackage{mathdots}

\usepackage[overload]{textcase}

\usepackage{lipsum}

\usepackage{scrextend} %Has the addmargin command
\usepackage{float} %Allows use of decimals for positions

\usepackage{tikz-cd}
\usetikzlibrary{arrows}
\usetikzlibrary{knots}
\usetikzlibrary{calc}
\usetikzlibrary{intersections}
\usetikzlibrary{tqft}
\usepackage{tikz-3dplot}
\tikzcdset{row sep/normal=1cm}
\usepackage{enumitem}
\setlist[enumerate]{leftmargin={1cm}, font={\bfseries}, label=\alph*., itemsep=.25cm, topsep=.25cm}

\usepackage[hidelinks, breaklinks]{hyperref}

\usepackage{amsthm}
\usepackage{thm-restate}

    \theoremstyle{definition}
\newtheorem{example}{Example}

\usepackage{letltxmacro}
\makeatletter
\let\oldr@@t\r@@t
\def\r@@t#1#2{%
\setbox0=\hbox{$\oldr@@t#1{#2\,}$}\dimen0=\ht0
\advance\dimen0-0.2\ht0
\setbox2=\hbox{\vrule height\ht0 depth -\dimen0}%
{\box0\lower0.4pt\box2}}
\LetLtxMacro{\oldsqrt}{\sqrt}
\renewcommand*{\sqrt}[2][\ ]{\oldsqrt[#1]{#2}}
\makeatother

\let\origfigure\figure
\let\endorigfigure\endfigure
\renewenvironment{figure}[1][tbph]{%
    \origfigure[#1]%
    \centering
}{%
    \endorigfigure
}

\usepackage{subfig}

\usepackage[mathscr]{euscript} %For our fancy letters

\newcommand{\C}{\mathbb{C}}
\newcommand{\R}{\mathbb{R}}

\newcommand{\bs}{\backslash}
\newcommand{\F}{\mathbb{F}} %Alternatively, mathcal{F} for Script F for family
\renewcommand{\O}{\mathcal{O}} %Script O for open sets
\newcommand{\im}{\text{Im }} %Image of map or imaginary part
\newcommand{\spn}[1]{\langle #1 \rangle}

\newcommand{\gl}[1]{\glWork#1\end}
    \def\glWork#1,#2\end{\text{GL}_{#1}(#2)}
\newcommand{\pgl}[1]{\pglWork#1\end}
    \def\pglWork#1,#2\end{\text{PGL}_{#1}(#2)}
\newcommand{\psl}[1]{\pslWork#1\end}
    \def\pslWork#1,#2\end{\text{PSL}_{#1}(#2)}

    \def\pspWork#1,#2\end{\text{PSp}_{#1}(#2)}
    
\let \tempHom\hom
\renewcommand{\hom}[1]{\tempHom \big(#1\big)}

\let\tempPhi\phi %Who wouldn't want this phi?
\let\phi\varphi
\let\varphi\tempPhi %Now varphi will print the "original" phi
\let\ForAll\forall 
    \renewcommand{\forall}{\ForAll \;} %Fixes stupid spacing issue
\let\Exists\exists
    \renewcommand{\exists}{\Exists \;} %Fixes stupid spacing issue
\let\LIM\lim
    \renewcommand{\lim}{\LIM\limits} %Why put a limit anywhere but under lim?

\newcommand{\set}[1]{\left\{ #1 \right\}}
\newcommand*\conj[1]{\mkern 1.5mu\overline{\mkern-1.5mu#1\mkern-1.5mu}\mkern 1.5mu}
\newcommand{\abs}[1]{\left|#1\right|}

\usepackage{pgfplots}
\usepgfplotslibrary{polar}
\pgfplotsset{compat=newest}

    \def\embedgraphwork#1,#2,#3\end{
\if\ifnum#3 > 3\fi
    \renewcommand{\dOne}{.5}
    \renewcommand{\dTwo}{.5}
    \renewcommand{\dThree}{.7}
    \renewcommand{\angleOne}{180/#1}
    \renewcommand{\angleTwo}{180/#2}
    \renewcommand{\angleThree}{180/#3}
\tikzset{rotate=-90}
\foreach \counter in{1, ..., #3}{
    \renewcommand{\s}{\counter*2*\angleThree}
    \coordinate (vertex1) at (0:0);
    \coordinate (vertex2) at (\s:1);
        \path[name path=A] (vertex2) -- ++ (\s + 180  - \angleOne:1.42); 
        \path[name path=B] (vertex1) --  (\s+\angleThree:1.42); 
        \path [name intersections={of=A and B,name=intersection}];
    \coordinate (vertex3) at (intersection-1);
    \coordinate (embedded1) at    ($.4*(vertex1)+.3*(vertex2)+.3*(vertex3)$);
    \coordinate (embedded2) at    ($.3*(vertex1)+.4*(vertex2)+.3*(vertex3)$);
    \coordinate (embedded3) at    ($.3*(vertex1)+.3*(vertex2)+.4*(vertex3)$);
    \draw (vertex1) to (vertex2);
    \draw (vertex1) to (vertex3);
    \draw (vertex2) to (vertex3) to (\s+2*\angleThree:1);
}
\tikzset{shift={(0:1)}}
\foreach \counter in{1, ..., #1}{
    \renewcommand{\s}{\counter*2*\angleOne}
        \path[name path=A] (\s:1) -- ++ (\s + 180  - \angleThree:1.42); 
        \path[name path=B] (0:0) --  (\s+\angleOne:1.42); 
        \path [name intersections={of=A and B,name=intersection}];
    \draw (0:0) to (\s:1);
    \draw (0:0) to (intersection-1);
    \draw (\s:1) to (intersection-1) to (\s+2*\angleOne:1);
}
\foreach \counter in{2, ..., #1}{
}
\tikzset{shift={(180:1)}}
\tikzset{shift={(intersection-1)}}
\tikzset{rotate=-\angleThree}
coordinate (temp) at (intersection-1)
\newdimen\Xtemp
\newdimen\Ytemp
\pgfgetlastxy{\Xtemp}{\Ytemp};
\foreach \counter in{1, ..., #2}{
    \renewcommand{\s}{\counter*2*\angleTwo}
        \path[name path=A] (\s:{veclen(\Xtemp, \Ytemp)}) -- ++ (\s + 180  - \angleThree:1.42); 
        \path[name path=B] (0:0) --  (\s+\angleTwo:1.42); 
        \path [name intersections={of=A and B,name=intersection}];
    \draw (0:0) to (\s:{veclen(\Xtemp, \Ytemp)});
    \draw (0:0) to (intersection-1);
    \draw (\s:{veclen(\Xtemp, \Ytemp)}) to (intersection-1) to (\s+2*\angleTwo:{veclen(\Xtemp, \Ytemp)});
}

\fi
}

\newcommand\pgfmathsinandcos[3]{%
  \pgfmathsetmacro#1{sin(#3)}%
  \pgfmathsetmacro#2{cos(#3)}%
}
\newcommand\LongitudePlane[3][current plane]{%
  \pgfmathsinandcos\sinEl\cosEl{#2} % elevation
  \pgfmathsinandcos\sint\cost{#3} % azimuth
  \tikzset{#1/.estyle={cm={\cost,\sint*\sinEl,0,\cosEl,(0,0)}}}
}
\newcommand\LatitudePlane[3][current plane]{%
  \pgfmathsinandcos\sinEl\cosEl{#2} % elevation
  \pgfmathsinandcos\sint\cost{#3} % latitude
  \pgfmathsetmacro\yshift{\cosEl*\sint}
  \tikzset{#1/.estyle={cm={\cost,0,0,\cost*\sinEl,(0,\yshift)}}} %
}
\newcommand\DrawLongitudeCircle[2][4]{
  \LongitudePlane{\angEl}{#2}
  \tikzset{current plane/.prefix style={scale=#1,dashed}}
  \pgfmathsetmacro\angVis{atan(sin(#2)*cos(\angEl)/sin(\angEl))} %
  \draw[current plane] (\angVis:1) arc (\angVis:\angVis+180:1);
}
\newcommand\DrawLatitudeCircle[2][5]{
  \LatitudePlane{\angEl}{#2}
  \tikzset{current plane/.prefix style={scale=#1,line width=.7pt}}
  \pgfmathsetmacro\sinVis{sin(#2)/cos(#2)*sin(\angEl)/cos(\angEl)}
  \pgfmathsetmacro\angVis{asin(min(1,max(\sinVis,-1)))}
  \draw[current plane] (\angVis:1) arc (\angVis:-\angVis-180:1);
}
\def\Rad{1} %sphere radius **LEAVE THIS AS 1**
\def\angEl{35} %angle of elevation
\tikzset{viewport/.style 2 args={
    x={({cos(-#1)*1cm},{sin(-#1)*sin(#2)*1cm})},
    y={({-sin(-#1)*1cm},{cos(-#1)*sin(#2)*1cm})},
    z={(0,{cos(#2)*1cm})}
}}

\pgfplotsset{only foreground/.style={
    restrict expr to domain={rawx*\CameraX + rawy*\CameraY + rawz*\CameraZ}{-0.05:100},
}}
\pgfplotsset{only background/.style={
    restrict expr to domain={rawx*\CameraX + rawy*\CameraY + rawz*\CameraZ}{-100:0.05}
}}

\def\addFGBGplot[#1]#2;{
    \addplot3[#1,only background, opacity=0.25] #2;
    \addplot3[#1,only foreground] #2;
}

\usepackage{xcomment} %Used below to put an appendix

\title{Conjugating Representations in \protect\\
    $\pgl{k, \C}$ into $\pgl{k, \R}$}
\date{October 2023}
\author{Jared Tristan Miller}

\begin{document}

\maketitle              %% Create the title page

\tableofcontents

\section*{Abstract}

The space of representations of a surface group into a given simple Lie group is a very active area of research and is particularly relevant to higher Teichm\"uller theory. For a closed surface, classical Teichm\"uller space is a connected component of the moduli space of representations into $\psl{2, \R}$ and \cite{Fock:2006} showed that the space of positive representations into $\psl{k, \R}$ coincides with the Hitchin component. In this paper we study representations of finitely generated groups into $\pgl{k, \C}$ and determine necessary and sufficient conditions for such a representation to be conjugate into $\pgl{k, \R}$. In this way, we identify representations in the larger representation variety which are conjugate in $\pgl{k, \C}$ to a representation in $\hom{\pi_1 (\Sigma), \pgl{k, \R}} / \pgl{k, \R}$.

\section{Introduction}

The space of representations of the fundamental group of a surface $\Sigma$ into $\pgl{k, \R}$ is of considerable interest. In the case where $\Sigma$ is oriented and closed, \cite{Labourie:2006} showed that the space of discrete, faithful representations $\pi_1(\Sigma) \rightarrow \psl{2, \R}$ up to conjugation by $\psl{2, \R}$ corresponds to classical Teichm\"uller space. 

Teichm\"uller space also generalizes to higher rank Teichm\"uller theory (for reference, see \cite{Wienhard:2018}) by considering representations into $\psl{k, \R}$ or an arbitrary simple Lie group. When the Lie group is split, real, and simple, these spaces of representations give rise to a Hitchin component, defined as the connected component of the moduli space of representations containing a principal Fuchsian representation. \cite{Fock:2006, Labourie:2006} showed independently that Hitchin components are higher Teichm\"uller spaces. It is therefore a natural question to ask when a representation into a larger group of linear transformations, such as into $\pgl{k, \C}$, is conjugate into $\pgl{k, \R}$. In this paper we describe necessary and sufficient conditions for certain representations in the larger representation variety $\hom{\pi_1(\Sigma), \pgl{k, \C}} / \pgl{k, \C}$ to be conjugate to representations into $\pgl{k, \R}$. 

If $\Sigma$ is some finite type surface and $\O \subseteq \hom{\pi_1(\Sigma), \pgl{k, \C}}$ is the open set of representations whose projective transformations have \emph{$\pgl{k, \R}$ compatible generic eigenvalues} (defined in section \ref{pglRcompatible}), then there exists functions $f_1, \cdots, f_n$ of the eigenvalues and eigendirections of the projective transformations such that some $\alpha \in \O$ is conjugate into $\pgl{k, \R}$ if and only if  $\alpha$ is in the vanishing set of $\set{f_1, \cdots, f_n}$. In the two dimensional version, given a representation $\alpha \in \hom{\pi_1(\Sigma), \pgl{2, \C}}$ with generators $M_1, \cdots, M_n$ with \emph{$\pgl{2, \R}$ compatible generic eigenvalues}, if the projective transformations $M_1$ and $M_2$ have \emph{hyperbolic eigenvalues} (defined in section \ref{hyperbolicEigenval}) then these functions can be reduced to the following functions of cross ratios (defined in section \ref{crossRatio}). 

\begin{itemize}
    \item $f_1 = [m_1^-, m_2^-, m_1^+, m_2^+] - \conj{ [m_1^-, m_2^-, m_1^+, m_2^+]}$ \\
\end{itemize}
and for each $j > 2$, if $M_j$ has \emph{hyperbolic eigenvalues} then
\begin{itemize}
    \item $f_{2j-4} = [m_1^-, m_j^-, m_1^+, m_2^+] - \conj{[m_1^-, m_j^-, m_1^+, m_2^+]}$ \\
    \item $f_{2j-3} = [m_1^-, m_j^+, m_1^+, m_2^+] - \conj{[m_1^-, m_j^+, m_1^+, m_2^+]}$ \\
\end{itemize}
and if $M_j$ has \emph{elliptic eigenvalues} (defined in section \ref{ellipticEigenval}) then
\begin{itemize}
    \item $f_{2j-4} = \left([m_1^-, m_j^-, m_1^+, m_2^+] - \conj{[m_1^-, m_j^-, m_1^+, m_2^+]} \right)$ \\
    \hspace*{3cm} $- \left( \conj{[m_1^-, m_j^+, m_1^+, m_2^+]} - [m_1^-, m_j^+, m_1^+, m_2^+ \right)$ \\
    \item $f_{2j-3} = \left([m_1^-, m_j^-, m_1^+, m_2^+] + \conj{[m_1^-, m_j^-, m_1^+, m_2^+]} \right)$ \\
    \hspace*{3cm} $- \left( \conj{[m_1^-, m_j^+, m_1^+, m_2^+]} + [m_1^-, m_j^+, m_1^+, m_2^+ \right)$ \\
\end{itemize}
where $m_j^-$ and $m_j^+$ denote the eigendirections of the generator $M_j$, taken as points in $\C P^1$. Throughout this paper we give similar conditions for more generic sets of generators and in arbitrary dimension. The number of real equations needed to determine if a representation of a surface $\Sigma$ is conjugate into $\pgl{2, \R}$ is $-3 \chi(\Sigma)$ where $\chi$ denotes the Euler characteristic, and this coincides with the complex dimension of the character variety of representations into $\pgl{2, \C}$. For representations into $\pgl{k, \C}$, the number of real equations needed to do the same is $(k^2 - 1) \chi(\Sigma)$, which also coincides with the complex dimension of the character variety in arbitrary dimension. 

%\samcomment{orbifolds and cone manifolds}

Projective linear transformations act on projective space and are induced by linear transformations on the vector space from which the projective space derives. Consider $\C P^{k-1} = \C^k / \sim$ with the standard quotient by complex scaling. The projective linear group $\pgl{k, \C}$ acts on $\C P^{k-1}$ by considering the action of a linear transformation on $\C^k$ and then taking the quotient; these projective transformations may or may not correspond to elements of $\pgl{k, \C}$ which are conjugate into $\pgl{k, \R}$. Since real transformations preserve $\R^k \subseteq \C^k$, a projective transformation may only be conjugate into $\pgl{k, \R}$ if it preserves some \emph{projective $\R$-form} (defined in section \ref{projRform}) which can be mapped to the quotient of $\R^k \subseteq \C^k$ by a change of basis or projective form. For example, consider the projective transformation 
    $$\begin{bmatrix} 2i & 0 \\ 0 & 1 \end{bmatrix}$$
which acts on the extended complex plane by rotating by $\pi/2$ and scaling by $2$. Since this transformation does not preserve any circles or lines in the extended complex plane, there is no change of basis that maps the extended real line, which is preserved by every real projective transformation, to a circle or line which is preserved by this transformation. Therefore, this projective transformation is not individually conjugate into $\pgl{2, \R}$. For more details, see lemma \ref{eigenvalLemma}.  

It is well known that matrix conjugation preserves eigenvalues and that complex eigenvalues of matrices over $\R$ occur in complex conjugate pairs. Further, a matrix with distrinct eigenvalues is diagonalizable over $\C$ and diagonalizable matrices are conjugate to diagonal matrices that have their eigenvalues along the diagonal. Collectively, this means that an element $M \in \pgl{k, \C}$ with distinct eigenvalues is individually conjugate into $\pgl{k, \R}$ if and only if there exists a real line $R$ through the origin in $\C$ such that all of the eigenvalues of $M$ are either on $R$ or are pairwise inverted by reflection about $R$ in $\C$. Throughout this paper, we assume that each projective transformation $M$ has distinct eigenvalues. This condition is not very restrictive, as we can approximate repeated eigenvalues by taking limits of projective transformations with distinct eigenvalues. The author is planning to address this restriction in future work. 

A representation into $\pgl{k, \C}$ being conjugate into $\pgl{k, \R}$ requires every generator of the fundamental group of the surface to map to an element in $\pgl{k, \C}$ which preserves a common \emph{projective $\R$-form} which may be transformed into the quotient of $\R^k \subseteq \C^k$. In this way, we are concerned about whether a finitely generated collection of transformations is simultaneously conjugate into $\pgl{k, \R}$, meaning there exists some projective transformation in $\pgl{k, \C}$ which conjugates every element in the collection into $\pgl{k, \R}$. We study this in three ways: first using the complex conjugation coming from the complexification of a real projective frame, second using Fock-Goncharov coordinates, namely cross ratios and triples ratios, coming from flags of eigendirections as projective linear transformation invariants, and lastly using only cross ratios coming from flags of eigendirections. We describe both the space of representations which are conjugate into $\pgl{k, \R}$ and how to determine if an individual representation is conjugate into $\pgl{k, \R}$.

\section{Using \texorpdfstring{$\R$}{R}-forms and Conjugation}

Let $\F$ be a field and $\F^k$ be the $k$-dimensional vector space over $\F$. The projective space $\F P^{k-1}$ is defined to be the quotient space of $F^k \bs \set{0}$ by $F^\times$, that is, the space of nonzero vectors quotient by scaling. Let $\gl{k, \F}$ be the group of invertible $k \times k$ matrices over $\F$ with the operation of ordinary matrix multiplication. Then $\pgl{k, \F}$ is defined as the quotient of $\gl{k, \F}$ by its center, which consists of all nonzero scalar transformations of $\F^k$. An element of $\pgl{k, \R}$, called a \emph{projective transformation}, acts naturally on $\R P^{k-1}$ by left multiplication and this action can be extended to acting on $\C P^{k-1}$. Let $\R P^{k-1} \subseteq \C P^{k-1}$ denote the quotient of $\R^k \subseteq \C^k$ by complex scaling. As products of points which can be represented with real vectors by transformations which can be represented with real matrices, the action of $\pgl{k, \R}$ on $\C P^{k-1}$ maps $\R P^{k-1} \subseteq \C P^{k-1}$ to itself, thereby preserving (not necessarily point-wise) the $\R P^{k-1} \subseteq \C P^{k-1}$.

Given the vector space $\C^k$, an \emph{$\R$-form} in $\C^k$ is the real span of any $\C$-basis for $\C^k$. The space of $\R$ forms can be identified with the homogeneous space $\gl{k, \C} / \gl{k, \R}$, because $\gl{k, \C}$ acts transitively on $\R$ forms and the stabilizer in $\gl{k, \C}$ of the $\R$ form coming from $\R^k \subseteq \C^k$ is $\gl{k, \R}$. A \emph{projective $\R$-form}\label{projRform} is the space of nonzero vectors in an $\R$-form quotient by scaling by $\C$. This can be interpreted as a copy of $\R P^{k-1} \subseteq \C P^{k-1}$. In particular, in dimension 2 we have that a projective $\R$ form is a circle or line in $\hat{\C}$, such as in the figures \ref{hyperbolicFigure} and \ref{ellipticFigure}. Since a projective $\R$-form is formed by the quotient of scaling by $\C$, the fiber of each projective $\R$-form is $S^1$, corresponding to multiplying each basis vector for an $\R$-form by a common scalar $e^{i \theta}$ for $\theta \in \R$. 

As points in $\C P^{k-1}$ are unchanged when multiplying by a scalar in $\C$, there would ambiguity of which representative of a given point in $\C P^{k-1}$ to consider. In particular, the $\R$-span in $\C^k$ of $k$ points in $\C P^{k-1}$ depends on the choice of representatives of each point. For that reason, we define a \emph{projective frame} for $\F P^{k-1}$ as a collection of $k + 1$ points such that no hyperplane in $\F^k$ contains $k$ of them. Denoting the projection of a point $v \in \F^k$ into $\F P^{k-1}$ by $p(v)$, a projective frame can be written as $\set{p(e_0), \cdots, p(e_k), p(e_0 + \cdots e_k)}$ where $\set{e_0, \cdots, e_k}$ is some basis of $\F^k$. In this case, we say that $\set{e_0, \cdots, e_k}$ is a basis associated to the given projective frame. A projective frame for $\C P^{k-1}$ does not lift to a unique basis of $\C^k$, but the $\R$-span in $\C^k$ of any basis associated to a particular projective frame projects to the same projective $\R$-form. 

We say that a projective transformation $M \in \pgl{k, \C}$ is \emph{conjugate into $\pgl{k, \R}$} if there exists some element $\Gamma \in \pgl{k, \C}$ such that $\Gamma^{-1} \; M \; \Gamma \in \pgl{k, \R}$. This is equivalent to there being a change of basis such that in the new basis $M \in \pgl{k, \R}$. We say that a collection of projective transformations $\set{M_j}$ is \emph{simultaneously conjugate into $\pgl{k, \R}$} if there exists a single element $\Gamma \in \pgl{k, \C}$ such that $\Gamma^{-1} \; M_j \; \Gamma \in \pgl{k, \R}$ for all $M_j$ in the collection. We start to study conjugation into $\pgl{k, \R}$ with a few lemmas. 

\begin{restatable}{lemma}{projectStabilizer}
The projection of the stabilizer of an $\R$ form is equal to the stabilizer of the projection of that $\R$ form. In other words, projections and stabilizers of $\R$ forms commute. 
\end{restatable}
\begin{proof}
The stabilizer of $\R^k$ is $\gl{k, \R}$. By definition, an $\R$ form is equivalent to $\R^k$ in some basis, so its stabilizer is isomorphic to $\gl{k, \R}$. The projection of $\gl{k, \R}$ is $\pgl{k, \R}$. On the other hand, the projection of an $\R$ form $\R^k \subseteq \C^k$ is $\R P^{k-1}$ and the stabilizer of $\R P^{k-1}$ is $\pgl{k, \R}$. 

\begin{tikzcd}[column sep = huge]
\R \text{ form in } \C^k \arrow[r, "stabilizer"] \arrow[d, "projection"] & \gl{k, \R} \arrow[d, "projection"] \\
\text{Projective $\R$ form in } \C P^{k-1} \arrow[r, "stabilizer"] & \pgl{k, \R}
\end{tikzcd}

\end{proof}

\begin{restatable}{lemma}{Rform}
An element $M \in \pgl{k, \C}$ is conjugate into $\pgl{k, \R}$ if and only if $M$ preserves some projective $\R$ form. 
\end{restatable}
\begin{proof}
If an element is conjugate into $\pgl{k, \R}$, then since conjugation amounts to a change of basis the element is in $\pgl{k, \R}$ with respect to some basis. In that basis the element preserves the projective $\R$ form coming from $\R^k \subseteq \C^k$. Matrix conjugation is necessarily invertible and therefore there exists some preserved projective $\R$ form which can be mapped to the projective $\R$ form coming from $\R^k \subseteq \C^k$ by that change of basis. 

On the other hand, suppose the element preserves some projective $\R$ form. All projective $\R$ forms in a given $\C P^{k-1}$ are equivalent up to change of basis, so without loss of generality assume that the projective $\R$ form which is preserved is the projection of $\R^k \subseteq \C^k$. Then in particular the projective transformation is in the stabilizer of the projection of $\R^k \subseteq \C^k$, so is in $\pgl{k, \R}$. Then the preserved projective frame, with respect to its original basis, has stabilizer which is conjugate to $\pgl{k, \R}$. Therefore, the element is conjugate into $\pgl{k, \R}$. 

\end{proof}

\begin{restatable}{corollary}{conjIFFrForm}\label{conjIFFrForm}
A finitely generated group of elements in $\pgl{k, \C}$ is simultaneously conjugate into $\pgl{k, \R}$ if and only if every generator in the collection preserves some common projective $\R$-form. 
\end{restatable}

Further, it is well known that the eigenvalues of a real matrix are either real or occur in complex conjugate pairs.

\begin{restatable}{lemma}{eigenvals} \label{eigenvalLemma}
A projective transformation $M$ in $\pgl{k, \C}$ may be conjugate into $\pgl{k, \R}$ only if for any representative there exists a real line $\ell$ through the origin in $\C$ such that all eigenvalues are either contained in $\ell$ or pairwise swapped by reflection about $\ell$. 
\end{restatable}
\begin{proof}
By elementary linear algebra, a linear transformation over $\R$ will necessarily have eigenvalues which are either real or occur with their complex conjugate also being an eigenvalue. This means that the eigenvalues of a transformation over $\R$ will be contained in or pairwise swapped by reflection about the real axis in $\C$. 

An element in $\pgl{k, \C}$ is projectively real if and only if it is a scalar multiple of an element in $\pgl{k, \R}$. This scalar multiplication results in a simultaneous scalar multiplication of each eigenvalue, which acts as a rotation about the origin in $\C$. Hence, the scalar multiplication rotates the real axis to some line $\ell$ through the origin in $\C$, such that the eigenvalues are either contained in or pairwise swapped by reflection about $\ell$. Since eigenvalues are conjugation invariant, the result follows. 

\end{proof}

If such a line $\ell$ exists, then we say that the eigenvalues of $M$ are \emph{$\pgl{k, \R}$ compatible}. \label{pglRcompatible}

Let $M \in \pgl{k, \C}$ have eigenvalues $\lambda_1, \lambda_2, \cdots, \lambda_k$. We say that $M$ is

\begin{itemize}
    \item (strictly) hyperbolic if there exists a real line $\ell$ through the origin in $\C$ such that all of the eigenvalues are contained in $\ell$. In this case we call the eigenvalues and associated eigendirections hyperbolic. \label{hyperbolicEigenval}
    \item (strictly) elliptic if there exists a real line $\ell$ through the origin in $\C$ such that the eigenvalues can be paired $(\lambda_{2j}, \lambda_{2j+1})$ and reflection about $\ell$ maps $\lambda_{2j} \leftrightarrow \lambda_{2j+1}$ for every pair. In this case we call the eigenvalues and associated eigendirections elliptic. \label{ellipticEigenval} 
    \item mixed (hyperbolic and elliptic) if it is $\pgl{k, \R}$ compatible but not strictly hyperbolic or strictly elliptic.
\end{itemize}

These definitions of eigenvalues are consistent with an element of $\pgl{k, \C}$ being conjugate into $\pgl{k, \R}$ implying the elliptic eigenvalues can be paired so that the ratio of each pair is in $S^1$ and the ratio of any two hyperbolic eigendirections is real. 

Transformations being hyperbolic and elliptic are not mutually exclusive. For example, consider an element in $\pgl{2, \C}$ with eigenvalues $i$ and $-i$. The eigenvalues are contained in a real line through the origin in $\hat{\C}$, namely the imaginary axis, but are also inverted by the conjugation corresponding to the projective $\R$ form coming from $\R^2 \subseteq \C^2$. In the case where a pair of eigenvalues are contained in a line through the origin in $\C$ and swapped by reflection about another line through the origin in $\C$, the transformation may preserve more projective $\R$ forms than it would otherwise. 

\begin{example}
In the case where different choices of $\ell$ would result in different collections of eigendirections being labelled as elliptic or hyperbolic, this coincides with there being two fundamentally different possible ways to conjugate a representative of the individual projective transformation into $\pgl{k, \R}$. For example, consider the matrix
    $$M = \begin{bmatrix}
        i & 0 & 0 & 0 \\
        0 & -i & 0 & 0 \\
        0 & 0 & 2 & 0 \\
        0 & 0 & 0 & -2
    \end{bmatrix}$$
The real line $\ell$ in $\C$ could be either the real axis or imaginary axis. In the case where the real axis is considered as $\ell$, the conjugation might use
    $$\Gamma = \begin{bmatrix}
        i & -i & 0 & 0 \\
        1 & 1 & 0 & 0 \\
        0 & 0 & 1 & 0 \\
        0 & 0 & 0 & 1
    \end{bmatrix}$$
to conjugate $M$ to 
\begin{align*}
\Gamma^{-1} M \Gamma &= 
    \begin{bmatrix}
        0 & -1 & 0 & 0 \\
        1 & 0 & 0 & 0 \\
        0 & 0 & 2 & 0 \\
        0 & 0 & 0 & -2
    \end{bmatrix} \in \pgl{k, \R}
\end{align*}

On the other hand, if the imaginary axis is considered as $\ell$ the conjugation might use
    $$\Gamma = \begin{bmatrix}
        1 & 0 & 0 & 0 \\
        0 & 1 & 0 & 0 \\
        0 & 0 & i & -i \\
        0 & 0 & 1 & 1
    \end{bmatrix}$$
to conjugate $M$ to
\begin{align*}
\Gamma^{-1} M \Gamma &=
    \begin{bmatrix}
        i & 0 & 0 & 0 \\
        0 & -i & 0 & 0 \\
        0 & 0 & 0 & 2i \\
        0 & 0 & -2i & 0 
    \end{bmatrix} 
    = 
    \begin{bmatrix}
        1 & 0 & 0 & 0 \\
        0 & -1 & 0 & 0 \\
        0 & 0 & 0 & 2 \\
        0 & 0 & -2 & 0
    \end{bmatrix}
    \in \pgl{k, \R}
\end{align*}
\end{example}
This is a subtle distinction, but it is determining which pair(s) of eigendirections act as a source/sink and which pair(s) of eigendirections act as points of rotation. We will study projective transformations which have representatives without this possibility. If a representative of an element in $\pgl{k, \C}$ has no repeated eigenvalues and no pair of eigenvalues are both contained in some line through the origin in $\C$ and swapped by reflection about some line through the origin in $\C$, then we say that the projective transformation has \emph{generic eigenvalues}. For a projective transformation with generic eigenvalues, an eigenvalue or eigendirection being hyperbolic and being elliptic are mutually exclusive.

A \emph{conjugation} on a complex vector space $\C^k$ is a conjugate linear involution. By \cite{Conrad:Complexification}, conjugations are in bijection with $\R$ forms. For an $\R$ form $R$ in $\C^k$, we can represent any $v \in \C^k$ in terms of the basis whose $\R$ span is $R$. A conjugation can therefore be thought of as componentwise complex conjugation of points in this basis.

\begin{example}
Consider the $\C$ basis $\beta = \set{[1+i, 0, 0], [0, i, 0], [0, 0, 1]}$ for $\C^3$ and let $R$ be the $\R$ form consisting of the real span of $\beta$. Let $a = [1-i, 3, 5]$. In terms of the basis for $R$, 
    $$a = i [1+i, 0, 0] + -3i [0, i, 0] + 5 [0, 0, 1] \text{ so } [a]_{\beta} = [i, -3i, 5]$$
Then, with respect to the conjugation corresponding to $R$ we have
    $$\conj{a} = \conj{[a]_{\beta}} = [-i, 3i, 5] = -i [1+i, 0, 0] + 3i [0, i, 0] + 5 [0, 0, 1] = [-1-i, -3, 5]$$
\end{example}

We say that a collection of eigendirections in $\C P^{k-1}$ from some elements in $\pgl{k, \C}$ \emph{respects a conjugation} if the hyperbolic eigendirections are fixed by the conjugation and elliptic eigendirection pairs are swapped by the conjugation. Equivalently, we say that the conjugation respects the collection of eigendirections.

\begin{restatable}{lemma}{conjugateRepresentative}
If $v \in \C P^{k-1}$, then the componentwise complex conjugate of $v$, $\conj{v}$, does not depend on the choice of representative of $v$. That is, if $v \sim v' \in \C P^{k-1}$ then $\conj{v} \sim \conj{v'}$. 
\end{restatable}
\begin{proof}
Let $v = [x_1, x_2, \cdots, x_k]$ so that $\conj{v} = [\conj{x_1}, \conj{x_2}, \cdots, \conj{x_k}]$. Then since $v' \sim v$, we have $v' = [s x_1, s x_2, \cdots, s x_k]$ for some $s \in \C$. Taking the componentwise complex conjugate we have $\conj{v'} = [\conj{s} \conj{x_1}, \conj{s} \conj{x_2}, \cdots, \conj{s} \conj{x_k}] = \conj{s} [\conj{x_1}, \conj{x_2}, \cdots, \conj{x_k}] = \conj{s} \conj{v} \sim \conj{v}$. 

\end{proof}

\begin{restatable}{theorem}{conjugationTheorem} \label{conjugationTheorem}
A collection of elements in $\pgl{k, \C}$ with $\pgl{k, \R}$ compatible generic eigenvalues is simultaneously conjugate into $\pgl{k, \R}$ if and only if there exists an $\R$ form $R$ in $\C^k$ such that the eigendirections of each element respect the conjugation corresponding to $R$.
\end{restatable}
\begin{proof}
Conjugate transformations represent the same linear transformation with respect to different bases. If a change of basis maps the $\R$ form $R$ to the $\R$ form $R'$, then by definition of conjugation that change of basis will map elements which are conjugate with respect to the conjugation corresponding to $R$ to elements which are conjugate with respect to the conjugation corresponding to $R'$.

$\Longrightarrow$ \\
Since matrices over $\R$ are well known to have non-real eigendirections in componentwise complex conjugate pairs, a collection of projective transformations simultaneously conjugate into $\pgl{k, \R}$ will necessarily have each element such that the eigendirections respect the conjugation corresponding to the $\R$ form which is transformed to $\R^k \subseteq \C^k$ by the change of basis corresponding to the conjugation action which maps the projective transformations into $\pgl{k, \R}$. 

$\Longleftarrow$ \\
If an element of $\pgl{k, \C}$ commutes with the conjugation corresponding to a projective $\R$ form $R$, then it is in the stabilizer of $R$. If the eigendirections of an element respect the conjugation corresponding to $R$, then since conjugation on a basis determines conjugation on the entire space, we have that the element commutes with the conjugation. Since the stabilizer of any projective $\R$ form is a conjugate of $\pgl{k, \R}$, the result follows. 

\end{proof}

\begin{restatable}{theorem}{uniqueness} \label{uniqueness}
If a collection of diagonalizable transformations in $\pgl{k, \C}$ has eigendirections that can be partitioned into subsets which lift into the $\C$ span of disjoint, nonempty subsets of some basis for $\C^k$, such that each elliptic eigendirection lives in the same subset as its elliptic eigendirection pair, then the collection preserves either 0 or infinitely many projective $\R$ forms. 
\end{restatable}
\begin{proof}
Suppose that the collection of transformations preserves some common projective $\R$ form, which implies that the eigenvalues are $\pgl{k, \R}$ compatible. Further, suppose that the collection of eigendirections can be partitioned into subsets which lift into the $\C$ span of disjoint, nonempty subsets $\set{v_1, \cdots, v_j}$ and $\set{v_{j+1}, \cdots, v_k}$ of some basis $\set{v_1, \cdots, v_k}$ for $\C^k$, such that each elliptic eigendirection lives in the same subset as its elliptic eigendirection pair.

%Let $R$ be some preserved projective $\R$ form and lift to an $\R$ form $\tilde{R}$ with corresponding conjugation $c$. Lift the eigendirections such that they respect the conjugation $c$. 

Since there exists a preserved projective $\R$ form, there exists a conjugation on $\C^k$ and lifts of the eigendirections such that they respect that conjugation. Because every elliptic eigendirection and its elliptic eigendirection pair lift to the span of the same set, there exists a conjugation on $\C^k$ which respects these lifts of the eigendirections and maps the $\C$ span of $\set{v_1, \cdots, v_j}$ to the $\C$ span of $\set{v_1, \cdots, v_j}$ and the $\C$ span of $\set{v_{j+1} \cdots, v_k}$ to the $\C$ span of $\set{v_{j+1} \cdots, v_k}$. This conjugation corresponds to an $\R$ form which projects to a preserved projective $\R$ form. This $\R$ form will necessarily come from a basis which contains disjoint subsets whose $\C$ spans are equal to the $\C$ spans of $\set{v_1, \cdots, v_j}$ and $\set{v_{j+1}, \cdots, v_k}$ respectively. 

We can modify the lifts and conjugation to produce infinitely many distinct projective $\R$ forms which are preserved. For every eigendirection $v$ that lifted to $\tilde{v}$ in the $\C$ span of $\set{v_1, \cdots, v_j}$, instead lift to $s \tilde{v}$ for some fixed $s \in S^1 / S^0$. Leave the lifts in the $\C$ span of $\set{v_{j+1}, \cdots, v_k}$ unchanged. The conjugation which respects these lifts will correspond to the $\R$ form that had its basis elements whose $\C$ span is equals the $\C$ span of $\set{v_1, \cdots, v_k}$ multiplied by the same $s$. This produces a distinct projective $\R$ form for every distinct $s \in S^1 / S^0$ because the the first $j > 0$ components of the $\R$ form are rotated by $s \notin \R$ and the remaining $(k - j) > 0$ components are unchanged. In this case a single preserved projective $\R$ form implies infinitely many preserved projective $\R$ forms. 

\end{proof}

\begin{restatable}{theorem}{zeroONEinfinite}
A collection of diagonalizable transformations in $\pgl{k, \C}$ will preserve either $0, 1,$ or infinitely many projective $\R$ forms. 
\end{restatable}
\begin{proof}
If there exists an eigenbasis for each projective transformation such that the eigendirections from these eigenbases can be partitioned into the spans of disjoint, nonempty subsets of some basis for $\C^k$, then by linearity and by the previous theorem there are either $0$ or infinitely many projective $\R$ forms preserved. Assume then that no eigenbases exist such that the eigendirections can be partitioned in this way. 

If there do not exist eigenbasis that can be lifted to the $\C$ span of disjoint, nonempty subsets of some basis for $\C^k$, then up to scaling by $S^1 / S^0$ there is a unique lift of the collection of eigendirections and therefore up to $S^1 / S^0$ a unique $\R$ form whose corresponding conjugation respects lifts of the eigenbases. Thus, there is at most a unique projective $\R$ form preserved. 

\end{proof}

\begin{restatable}{corollary}{uniqueCorollary} \label{uniqueCorollary}
If a collection of transformations in $\pgl{k, \C}$ with $\pgl{k, \R}$ compatible generic eigenvalues has eigendirections that can not be partitioned into subsets which lift into the $\C$ span of disjoint, nonempty subsets of some basis for $\C^k$, such that each elliptic eigendirection lives in the same subset as its elliptic eigendirection pair, then the collection preserves at most $1$ projective $\R$ form.
\end{restatable}
\begin{proof}
In the case where there are no repeated eigenvalues, there is a unique eigenbasis for each projective transformation up to scaling individual elements by $\C \bs \set{0}$. If the eigendirections cannot be partitioned into subsets which lift into the $\C$ span of some disjoint, nonempty subsets of some basis for $\C^k$ such that each elliptic eigendirection lives in the same subsets as its elliptic eigendirection pair, then by conjugate linearity of conjugation, there is at most one conjugation preserving lifts of the eigendirections, up to multiplying each element in the associated $\R$ form by $S^1 / S^0$. Hence, the lift of a single preserved projective $\R$ form determines all possible conjugations which are compatible with lifts of an eigenbasis for each projective transformation, so by theorem \ref{conjugationTheorem} there is at most one projective $\R$ form preserved.

\end{proof}

\begin{example}
Consider the collection $\set{[1, 0, 0], [0, 1, 0], [0, 0, 1], [0, 1, 1]}$ of hyperbolic eigendirections coming from transformations in $\pgl{3, \C}$. Because the eigendirections can be partititioned into subsets $\set{[0, 1, 0], [0, 0, 1], [0, 1, 1]}$ and $\set{[1, 0, 0]}$ which lift into the $\C$ span of disjoint subsets of basis vectors for $\C^3$, if any projective $\R$ form is preserved then infinitely many will be. 

We then begin considering projective frames which correspond to projective $\R$-forms which could be preserved by the transformations corresponding to the collection. A valid projective frame is $\set{[i, 0, 0], [0, 1, 0], [0, 0, 1], [i, 1, 1]}$ with associated basis $\set{[i, 0, 0], [0, 1, 0], [0, 0, 1]}$, since every eigendirection in the collection is the projection of a point in the real span of the associated basis. This projective $\R$-form clearly contains the point $[i, 1, 1]$.

Another valid projective frame is $\set{[1+i, 0, 0], [0, 1, 0], [0, 0, 1], [1+i, 1, 1]}$ with associated basis $\set{[1+i, 0, 0], [0, i, 0], [0, 0, i]}$, since every eigendirection in the collection is also the projection of a point in the real span of this basis. By simple computations, we see that $[i, 1, 1]$ is not in the projective $\R$-form coming from this projective frame, so this projective frame produces a distinct projective $\R$-form. 
\begin{align*}
    r_1 [1+i, 0, 0] + r_2 [0, 1, 0] + r_3 [0, 0, 1] &= z [i, 1, 1] && \text{ for } r_j \in \R, z \in \C \\
    [r_1 (1 + i), r_2, r_3] &= z [i, 1, 1] && \text{ no solutions}
\end{align*}

By the same reasoning, any pair of projective frames 
    $$\set{[1, 0, 0], [0, 1, 0], [0, 0, 1], [z_1, 1, 1]}$$ 
and 
    $$\set{[1, 0, 0], [0, 1, 0], [0, 0, 1], [z_2, 1, 1]}$$ 
will correspond to distinct projective $\R$-forms when $\arg(z_1) \neq \pm \arg(z_2)$. Hence, the collection of eigendirections $\set{[1, 0, 0], [0, 1, 0], [0, 0, 1], [0, 1, 1]}$ may correspond to transformations which preserve infinitely many projective $\R$-forms. 
\end{example}

Given a collection of transformations in $\pgl{k, \C}$, our goal is to determine if every element in the collection preserves a common projective $\R$-form. We first determine, using two distinct projective transformations from the collection, which projective $\R$-form might be preserved, if any. We then determine if the collection is simultaneously conjugate into $\pgl{k, \R}$ by studying the conjugation action on each eigendirection. In later chapters we create flags from the projective frame and the eigendirections of each transformation, and use those flags to compute cross ratios and triple ratios to determine if the collection is simultaneously conjugate into $\pgl{k, \R}$. 

\section{Determining the Preserved Projective \texorpdfstring{$\R$}{R}-form}
For the remainder of the paper, we assume that each projective transformation does not have repeat eigenvalues. This is a simplification so that there is projectively a unique eigenbasis for each transformation and therefore there are infinitely many preserved projective $\R$ forms only if the collection of eigendirections can be individually lifted to vectors that live in the span of disjoint, nonempty subsets of some basis for $\C^k$. We will use genericity conditions which prevent such disjointedness, thereby ensuring that the collections of projective transformations preserve at most a single projective $\R$ form. 

If a projective frame of hyperbolic eigendirections with distinct eigenvalues is present in the set of eigendirections coming from a collection of transformations, then that projective frame defines the only possible projective $\R$ form that might be preserved by the corresponding transformations in $\pgl{k, \C}$. 

If that is not the case, then naively we could solve a system of equations coming from cross ratios, which would determine points in $\C P^{k-1}$ that are preserved by the transformations with a given collection of eigendirections or simply map hyperbolic eigendirections to the projective $\R$ form coming from $\R^k \subseteq \C^k$ and elliptic eigendirection pairs to component-wise complex conjugate pairs, until our freedom to map is over, and then checking if the collection of projective transformations has been transformed into one in $\pgl{k, \R}$. 

Our new approach is to again notice that conjugations are in bijection with $\R$ forms. In particular, in theorem 4.11 Conrad \cite{Conrad:Complexification} describes explicitly the bijection. By thoerem 1 we know conjugation about a preserved $\R$ form would act by fixing hyperbolic eigendirections and swapping elliptic eigendirection pairs, so we can determine how conjugation acts on the whole space and can efficiently determine the corresponding $\R$ form. 

Once we have determined a projective $\R$ form which is preserved, if any, by the transformations corresponding to $k+1$ eigendirections (and perhaps their elliptic eigendirection pair, if not used), we want to choose $k+1$ points that form a projective frame for that projective $\R$ form. 

\begin{restatable}{theorem}{findRform}
Given a collection of $k+1$ hyperbolic and/or elliptic eigendirections from elements in $\pgl{k, \C}$, such that no $k$ of them are contained in any hyperplane, the only projective $\R$ form compatible with those eigendirections can be represented as
    $$\set{p = \sum\limits_{j=1}^{k} \lambda_j v_j \mid \lambda_j \in \C \text{ not all 0 and } p = \conj{p}} / \sim$$
where $\set{v_j}$ is a basis associated to the projective frame coming from the $k+1$ eigendirections and $\conj{p} = \sum\limits_{j=1}^{k} \conj{\lambda_j} v_j'$ with
    $$v_j' = 
    \begin{cases}
        v_j & \text{if hyperbolic} \\
        \text{its eigendirection pair} & \text{if elliptic}
    \end{cases}$$
written in an associated basis to the projective frame coming from the images of the $k+1$ eigendirections, and $\sim$ is scaling by $\C \bs \set{0}$. 
\end{restatable}
\begin{proof}
Let $\set{h_1, \cdots, h_a, e_1, \cdots, e_b}$ be such a collection of $k+1$ hyperbolic eigendirections $h_j$ and elliptic eigendirections $e_j$ coming from projective transformations in $\pgl{k, \R}$. Individually scale each of the first $k$ eigendirections to produce a projective frame $\set{v_1, \cdots, v_k, v_{k+1}}$ with $\sum\limits_{j=1}^{k} v_j = v_{k+1}$ and take $\set{v_1, \cdots, v_k}$ to be the associated basis. 

Since they are real, the corresponding projective transformations preserve the projection of $\R^k \subseteq \C^k$, and by corollary \ref{uniqueCorollary}, the only projective $\R$ form the associated projective transformations might preserve is this the quotient of $\R^k \subseteq \C^k$. Further, we know that a real matrix has elliptic eigendirections in component-wise complex conjugate pairs. Hence, this projective $\R$ form can be represented as
    $$\set{p = \sum\limits_{j=1}^{k} \lambda_j v_j \mid \lambda_j \in \C \text{ not all 0 and } p \in \R^k} / \sim$$
where $\sim$ is projective equivalence. But $p = \sum\limits_{j=1}^{k} \lambda_j v_j \in \R^k$ if and only if
    $$p = \conj{p} = \sum\limits_{j=1}^{k} \conj{\lambda_j v_j} = \sum\limits_{j=1}^{k} \conj{\lambda_j} v_j'$$
Since homographies exist between any two projective frames on a given projective space and elements of the homography group, $\pgl{k,\C}$, are linear transformations, the result follows. 

\end{proof}

\subsection{\texorpdfstring{$\pgl{2, \C}$}{PGL(2, C)} example containing Projective Frame}
Elements in $\pgl{2, \C}$ are simultaneously conjugate into $\pgl{2, \R}$ if and only if they preserve a common projective $\R$ form. As described above, we first use $3$ eigendirections (and their elliptic eigendirection pairs, if not already included) to determine the only possible preserved projective $\R$ form. 

\begin{example}
Consider a collection of transformations with elliptic eigendirection pairs $\set{[1, 0], [0, 1]}$ and $\set{[1, 1], [-1, 1]}$. Denote the eigendirections by $v_1 = [1, 0], v_2 = [0, 1], v_3 = [1, 1]$, and $v_4 = [-1, 1]$. Taking the first three eigendirections, we have a projective frame $\set{[1, 0], [0, 1], [1, 1]}$ and associated basis $\set{[1, 0], [0, 1]}$. The images of these three eigendirections after a swapping elliptic eigendirection pairs gives a projective frame $\set{[0, 1], [1, 0], [-1, 1]}$ and associated basis $\set{[0, 1], [-1, 0]}$.

The only projective $\R$ form which might be preserved is
    $$\set{p = \sum\limits_{j=1}^{2} \lambda_j v_j \mid \lambda_j \in \C \text{ not all 0 and } p = \conj{p}} / \sim$$
If $p = \lambda_1 [1, 0] + \lambda_2 [0, 1] = [\lambda_1, \lambda_2]$, then $\conj{p} = \conj{\lambda_1} [0, 1] + \conj{\lambda_2} [-1, 0] = [-\conj{\lambda_2}, \conj{\lambda_1}]$. But for $p = \conj{p}$, this means $\lambda_1 = 0 = \lambda_2$, contradicting the assumption that not all $\lambda_j$ are $0$. Thus, no projective $\R$ forms are preserved by each of the corresponding transformations. 

This can be visualized geometrically by considering the projective $\R$ forms which are preserved by each projective transformation. Figure 1 below shows the action of projective transformations with representatives
\begin{align*}
    A &= \begin{bmatrix}
        2+i & 0 \\
        0 & 2-i
    \end{bmatrix} \\
    B &= \begin{bmatrix}
        1 & -i \\
        -i & 1 \\
    \end{bmatrix}
\end{align*}
on $\C P^1$. $A$ has eigendirections $[1, 0]$ and $[0, 1]$ with corresponding eigenvalues $2+i$ and $2-i$ respectively. $B$ has eigendirections $[1, 1]$ and $[-1, 1]$ with corresponding eigenvalues $1-i$ and $1+i$ respectively. 

Note that both $A$ and $B$ act as rotations of the Riemann sphere about the axis through their pair of eigendirections, thereby preserving projective $\R$ forms which consist of (possibly infinite radius) circles in $\C$ about which circle inversion swaps their pair of eigendirections. Since circle inversion preserves rays from the circle center, no common circle is preserved. The figure below shows the preserved projective $\R$ forms of $A$ in blue and $B$ in red, both drawn on $\C P^1$. 

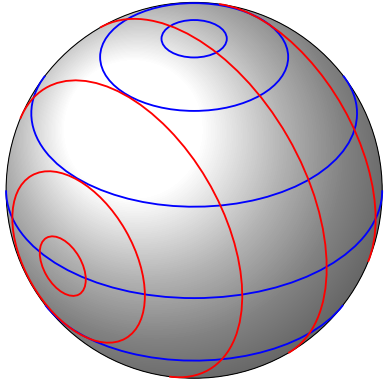
\begin{figure}
\begin{tikzpicture}[scale=2.5]

    %Define planes intersecting the sphere
    \LatitudePlane[equator]{\angEl}{0}
    \LatitudePlane[cancer]{\angEl}{40} %Not actually 23 degrees
    \LatitudePlane[arctic]{\angEl}{80} %Not actually 66 degrees
    \LatitudePlane[capricorn]{\angEl}{-40} %Not actually -23

    %Draw the ball
    \filldraw[ball color=white] (0, 0) circle (\Rad);
    %Draw the longitudes and latitudes
    \begin{scope}[draw=blue]
        \foreach \t in {-80, -60, -30, 0, 30, 60, 80} {\DrawLatitudeCircle[\Rad]{\t}}
    \end{scope}
    
    \begin{scope}[rotate=120, draw=red]
        \foreach \t in {-80, -60, -30, 0, 30, 60, 80} {\DrawLatitudeCircle[\Rad]{\t}}
    \end{scope}
\end{tikzpicture}
\caption{Preserved $\R$ forms that do not coincide}
\end{figure}
\end{example}

\begin{example}
Consider the subgroup of projective transformations generated by 
$$\begin{tabular}{lcc}
    Matrix representative & eigenvalues & eigendirections \\[0.5cm]
    {$a = \begin{bmatrix} 3i-1 & 3i-3 \\ -3i-3 & -3i-1 \end{bmatrix}$} & {$\begin{matrix} 1 \\ -2 \end{matrix}$}& {$\begin{matrix} [-1, 1] \\ [-i, 1] \end{matrix}$} \\[0.5cm]
    {$b = \begin{bmatrix} 3i+1 & -3i-3 \\ 3i-3 & -3i+1 \end{bmatrix}$} & {$\begin{matrix} 2 \\ -1 \end{matrix}$}& {$\begin{matrix} [-i, 1] \\ [1, 1] \end{matrix}$} \\[0.5cm]
    {$c = \begin{bmatrix} 1+i & 0 \\ 0 & 1-i \end{bmatrix}$} & {$\begin{matrix} 1+i \\ 1-i \end{matrix}$}& {$\begin{matrix} [1, 0] \\ [0, 1] \end{matrix}$} \\[0.5cm]
    {$d = \begin{bmatrix} -2+5i & -3 \\ -3 & -2-5i \end{bmatrix}$} & {$\begin{matrix} -1-2i \\ -1+2i \end{matrix}$}& {$\begin{matrix} [-i, 3] \\ [-3i, 1] \end{matrix}$} \\[0.5cm]
\end{tabular}$$

Based on the eigenvalues, projective transformations $a$ and $b$ are hyperbolic while $c$ and $d$ are elliptic. We demonstrate how to determine the projective $\R$ form which may be preserved starting with three hyperbolic eigendirections and then three elliptic eigendirections. The mixed case follows the same style on each respective type of eigendirection. 

\textit{Scenario 1.} If we were to start with three hyperbolic eigendirections, say $[-1, 1], [-i, 1], [1, 1]$, we would know that any projective $\R$ form preserved by the collection of projective transformations must pass through these three points. These three points are then used to construct a projective frame. Solving the equation
    $$\lambda_1 [-1, 1] + \lambda_2 [-i, 1] = [1, 1]$$
gives $\lambda_1 = -i$ and $\lambda_2 = 1+i$. We thus take the projective frame to be 
\begin{align*}
    F &= \set{\lambda_1 [-1, 1], \lambda_2 [-i, 1], [1, 1]} \\
    &= \set{[i, -i], [1-i, 1+i], [1, 1]}
\end{align*}
and the associated basis as $\set{[i, -i], [1-i, 1+i]}$. 

Then the only copy projective $\R$ form that might be preserved by the collection is the set of points
\begin{align*}
    \set{p = c_1 [i, -i] + c_2 [1-i, 1+i] \mid c_j \in \C \text{ not all 0 such that } p = \conj{p}} / \sim
\end{align*}
where the condition of $p = \conj{p}$ is complex conjugacy with respect to the $\R$-form and $\sim$ is complex projective equivalence. That is, with $p$ as above, $\conj{p} = \conj{c_1} [i, -i] + \conj{c_2} [1-i, 1+i]$ since the three hyperbolic eigendirections are preserved by conjugation about the $\R$-form. The condition of $p = \conj{p}$ is equivalent to
    $$[i c_1 + (1 - i) c_2, -i c_1 + (1 + i) c_2] = [i \conj{c_1} + (1 - i) \conj{c_2}, -i \conj{c_1} + (1 + i) \conj{c_2}]$$
which implies $\im(c_1) = \im(c_2)$. We can simplify the expression of $\R P^1$ as the set of points
    $$\set{p = [c_2 + (c_1 - c_2) i, c_2 + (c_2 - c_1) i] \mid c_j \in \C \text{ not all 0 such that } p = \conj{p}} / \sim$$
Then since $p = \conj{p} \Longrightarrow \im(c_1) = \im(c_2)$, the two coordinates have the same magnitude, so we have that the only projective $\R$ form that might be preserved by our collection is 
    $$\set{p = [z, 1] \mid z \in S^1}$$
    
\textit{Scenario 2.} If we were to start with three elliptic eigendirections, in this case $[1, 0], [0, 1], [-i, 3]$, we would know that the conjugation corresponding to the $\R$ form preserved by these projective transformations must map $[1, 0] \mapsto [0, 1], [0, 1] \mapsto [1, 0], [-i, 3] \mapsto [-3i, 1]$. These three points are then used to construct a projective frame. Solving the equation
    $$\lambda_1 [1, 0] + \lambda_2 [0, 1] = [-i, 3]$$
gives $\lambda_1 = -i$ and $\lambda_2 = 3$. We thus take the projective frame to be
\begin{align*}
    F &= \set{\lambda_1 [1, 0], \lambda_2 [0, 1], [-i, 3]} \\
    &= \set{[-i, 0], [0, 3], [-i, 3]}
\end{align*}
and the associated basis as $\set{[-i, 0], [0, 3]}$.

Then the only projective $\R$ form that might be preserved by the collection is the set of points
    $$p = \set{c_1 [-i, 0] + c_2 [0, 3] \mid c_j \in \C \text{ not all 0 such that } p = \conj{p}} / \sim$$
where the condition of $p = \conj{p}$ is complex conjugacy with respect to the $\R$-form and $\sim$ is complex projective equivalence. 

A projective frame for the images $\set{[0, 1], [1, 0], [-3i, 1]}$ of these three elliptic eigendirections is $\set{[0, 1], [-3i, 0], [-3i, 1]}$ with associated basis $\set{[0, 1], [-3i, 0]}$ so that with $p$ as above, $\conj{p} = \conj{c_1} [0, 1] + \conj{c_2} [-3i, 0]$. The condition of $p = \conj{p}$ is then
    $$[-i (c_1), 3 (c_2)] = [-3i (\conj{c_2}), \conj{c_1}]$$
which implies $c_1 = 3 \conj{c_2}$. Then the preserved projective $\R$ form is the set
\begin{align*}
    \set{p = [-i (c_1), 3 (c_2)] \mid c_j \in \C \text{ not all 0 such that } p = \conj{p}} / \sim \\
    = \set{p = \left[ c_1, \conj{c_1} \right] \mid c_1 \in \C \bs \set{0} } / \sim
\end{align*}
so the two coordinates of points $p$ in the projective $\R$ form have the same magnitude, which allow us to simplify the expression of the only projective $\R$ form which might be preserved by our collection as the set of points
    $$\set{p = [z, 1] \mid z \in S^1}$$
which coincides with the calculations when using hyperbolic eigendirections. 
\end{example}

\subsection{Set of Eigendirections not containing a Projective Frame}
For a collection of transformations whose set of eigendirections does not contain a projective frame, we solve a similar set of equations. By assumption, each element $\pgl{k, \R}$ compatible generic eigenvalues in $\pgl{k, \C}$, so has a unique set of eigendirections, up to individually scaling by $\C \bs \set{0}$, which form an eigenbasis. 

Then a basis associated to a projective $\R$ form which is preserved by the collection of transformations is some set $\set{v_1, \cdots, v_k}$ such that for every hyperbolic eigendirection $p$ we have
    $$p = \sum\limits_{j=1}^k \lambda_j v_j \text{ for } \lambda_j \in \R \text{ not all 0}$$
and for elliptic eigendirection pair $p, p'$ we have
    $$p = \sum\limits_{j=1}^k \lambda_j v_j \text{ for } \lambda_j \in \C \text{ not all 0}$$
with
    $$\conj{p} = \sum\limits_{i=1}^k \conj{\lambda_j v_j} = \sum\limits_{i=1}^k \conj{\lambda_j} v_j' = p'$$
where
    $$v_j' = 
    \begin{cases}
        v_j & \text{if hyperbolic} \\
        \text{its eigendirection pair} & \text{if elliptic}
    \end{cases}$$
\begin{example}
Consider a collection of elements in $\pgl{3, \C}$ whose set of eigendirections contains elliptic eigendirections $[-i, 1, 0], [i, 1, 0], [1+i, 1, 0], [1-i, 1, 0]$ and hyperbolic eigendirection $[0, 0, 1]$. By observation, this collection of eigendirections is compatible with the projective frame $\set{[1, 0, 0], [0, 1, 0], [0, 0, z], [1, 1, z]}$ and associated basis $\set{[1, 0, 0], [0, 1, 0], [0, 0, z]}$ for every $z \in \C$. Further, it is clear that that for each $z \in S^1 / S^0$ we obtain a distinct projective frame and therefore a different projective $\R$ form which is preserved.
\end{example}

\begin{example}
If we append the set of eigendirections from the previous example with the hyperbolic eigendirection $[0, 1, 1]$, then we are no longer able to scale the basis vector $[0, 0, 1]$ to become $[0, 0, z]$ for $z \notin \R$, as then the real span of the associated basis will not contain $[0, 1, 1]$. Hence, the associated basis must be projectively equivalent to $\set{[1, 0, 0], [0, 1, 0], [0, 0, 1]}$, so there is a unique projective $\R$ form preserved. 
\end{example}

\begin{example}
If we again append the set of eigendirections from the previous examples with an elliptic eigendirection pair $[1, 0, 1], [2, 0, 2]$, then since the only projective frame compatible with the previous eigendirections does not correspond to a conjugation which maps $[1, 0, 1] \leftrightarrow [2, 0, 2]$, we have that the projective $\R$ form is not preserved by the transformation corresponding to this elliptic eigendirection pair. Therefore, there are no projective $\R$ forms preserved by every transformation in the collection and hence the collection is not simultaneously conjugate into $\pgl{3, \R}$. 
\end{example}

\subsection{Determining Conjugacy into \texorpdfstring{$\pgl{, \R}$}{PGL(R)}}
A collection of elements in $\pgl{k, \C}$ is simultaneously conjugate into $\pgl{k, \R}$ if and only if they preserve a common projective $\R$ form. The collection of projective transformations preserves a common projective $\R$ form if and only if every element in the collection has eigenvalues which are compatible with being conjugate into $\pgl{k, \R}$ and for each element there exists an eigenbasis such that there exists a lift of all of the eigendirections from these eigenbasis that respects some conjugation coming from a common $\R$ form in $\C^k$. Therefore, we can determine if a collection of projective transformations is simultaneously conjugate into $\pgl{k, \R}$ by examining just the eigendirections and eigenvalues. In the same manner that we determined the only projective $\R$ form which might be simultaneously preserved, we can determine if that projective $\R$ form is in fact preserved by every transformation. 

In the coming chapters we detail how to use the so called Fock-Goncharov coordinates, namely cross ratios and triple ratios, to determine if a collection of elements in $\pgl{k, \C}$ is simultaneously conjugate into $\pgl{k, \R}$. 

\section{Introduction to Fock-Goncharov Coordinates}

\subsection{Cross Ratios in \texorpdfstring{$\C P^1$}{CP1}}
It is well known that M\"obius transformations, which form the group $\pgl{2, \C}$, act transitively on ordered triples of points in $\C P^1$ or equivalently points in $\hat{\C} := \C \cup \set{\infty}$. Given a quadruple $(A, B, C, D)$ of points each in $\C P^1$, we can use that transitivity to find the (unique) M\"obius transformation $\Gamma$ mapping $A \mapsto \infty$, $C \mapsto 0$, and $D \mapsto 1$\label{crossRatio}. We then assign to the ordered quadruple the image of $B$ after acting by $\Gamma$, and define the cross ratio to be that value. We denote a cross ratio by $[A, B, C, D]$. It is easy to compute that the cross ratio is equal to 
    $$[A, B, C, D] = \frac{(A-D)(C-B)}{(A-B)(C-D)}$$
    
This is a non-standard normalization, though several different normalizations are used by different authors. We choose this normalization so that the sign of a real cross ratio $[A, B, C, D]$ can be used to determine if $B$ and $D$ separate $A$ from $C$ along the common circle and so that in higher dimensions the quotient spaces are more convenient. A more standard normalization is that used by Fock and Goncharov \cite{Fock:2006} with $[[\infty, -1, 0, x]] = x$ defined by
    $$[[A, B, C, D]] = \frac{(A-B)(C-D)}{(A-D)(B-C)}$$
Our cross ratio can be seen as
\begin{align*}
    [A, B, C, D] &= -1/[[A, D, C, B]]
\end{align*}

For brevity, if we have transformations $G$ and $H$ with eigendirections $g_1, g_2$ and $h_1, h_2$ respectively, we introduce the notation $[G, H] := [g_1, h_1, g_2, h_2]$. 

Given two triangles in $\hat{\C}$ sharing a side, we can consider taking the cross ratio of their vertices. This value will tell us how the triangles lie next to each other on the surface and we will point out two results here. The reason for normalizing to $-1$ rather than $+1$ is so that positivity is associated to the triangles being disjoint. For our purposes, this will not be strictly necessary and it is more convenient to normalize to $+1$. 

\begin{restatable}{proposition}{realCR}
The cross ratio $[A, B, C, D]$ is real if and only if the four points $A, B, C, D$ lie on a circle or line in $\hat{C}$ and a real cross ratio will be positive precisely when the points $B$ and $C$ do not separate $A$ and $D$ along that circle or line.
\end{restatable}
\begin{proof}
The first half follows from the fact that M\"obius transformations map (possibly infinite radius) circles to (possibly infinite radius) circles, and in particular the circle through $A, B$, and $C$ will be mapped to the extended real line by the M\"obius transformation defining the cross ratio. The second half follows from simplification of the cross ratio
    $$[\infty, x, 0, 1] = \frac{(\infty + 1)(0 - x)}{(\infty - x)(0 - 1)} = x$$

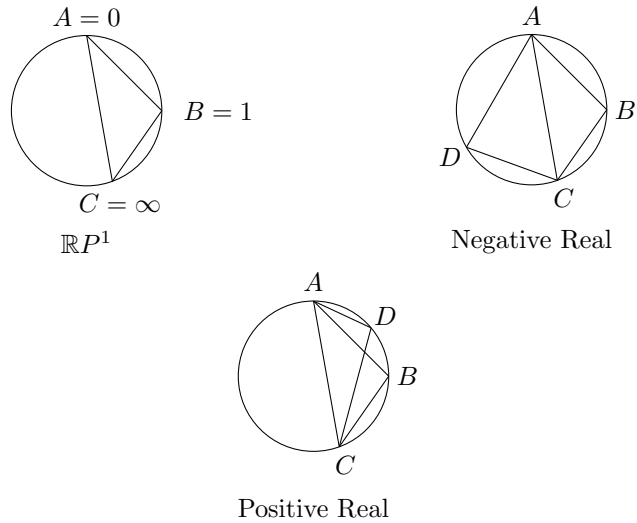
\begin{figure}[H]
    \begin{tikzpicture}[scale=0.5]
        \draw (0, 0) circle (2);
        %Edges
        \draw (90:2) -- (0:2);
        \draw (0:2) -- (-70:2);
        %AC
        \draw (90:2) -- (-70:2);
        %Labels
        \node at (90:2.5) {$A=0$};
        \node at (0:3.5) {$B=1$};
        \node at (-70:2.6) {$C=\infty$};
        \node at (270:3.5) {$\R P^1$};
        \hspace{2cm}
    \end{tikzpicture}
    \hspace{2cm}
    \begin{tikzpicture}[scale=0.5]
        \draw (0, 0) circle (2);
        %Edges
        \draw (90:2) -- (0:2);
        \draw (0:2) -- (-70:2);
        \draw (-70:2) -- (210:2);
        \draw (210:2) -- (90:2);
        %AC
        \draw (90:2) -- (-70:2);
        %Labels
        \node at (90:2.5) {$A$};
        \node at (0:2.5) {$B$};
        \node at (-70:2.5) {$C$};
        \node at (210:2.5) {$D$};
        \node at (270:3.5) {Negative Real};
    \end{tikzpicture}
    \hspace{2cm}
    \begin{tikzpicture}[scale=0.5]
        \draw (0, 0) circle (2);
        %Edges
        \draw (90:2) -- (0:2);
        \draw (0:2) -- (-70:2);
        \draw (-70:2) -- (40:2);
        \draw (40:2) -- (90:2);
        %AC
        \draw (90:2) -- (-70:2);
        %Labels
        \node at (90:2.5) {$A$};
        \node at (0:2.5) {$B$};
        \node at (-70:2.5) {$C$};
        \node at (40:2.5) {$D$};
        \node at (270:3.5) {Positive Real};
    \end{tikzpicture}
\caption{Geometric interpretations of real cross ratios}
\end{figure}
\end{proof}

\begin{restatable}{proposition}{modOneCR}
The cross ratio $[A, B, C, D]$ is of modulus 1 if and only if $B$ and $C$ are swapped by circle inversion about some (possibly infinite radius) circle through $A$ and $D$. 
\end{restatable}
\begin{proof}
Without loss of generality, suppose that $A = \infty$ and $C = 0$. Then $B$ and $D$ will be swapped by inversion about a circle through $A$ and $C$, equivalently a line through the origin, if and only if $\abs{B} = \abs{D}$. After simplifying
    $$[A, B, C, D] = \frac{B}{D} \in S^1$$
so the result follows. 

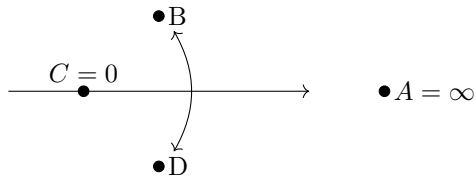
\begin{figure}[H]
    \begin{tikzpicture}
        \draw[->] (-1, 0) -- (3, 0);
        %Points
        \draw[fill=black] (0, 0) circle (2pt);
        \draw[fill=black] (4, 0) circle (2pt);
        \draw[fill=black] (1, 1) circle (2pt);
        \draw[fill=black] (1, -1) circle (2pt);
        %Labels
        \node[above] at (0,0) {$C = 0$};
        \node[right] at (4, 0) {$A = \infty$};
        \node[right] at (1, 1) {B};
        \node[right] at (1, -1) {D};
        %Reflection
        \draw[<->, bend right = 30] (1.2, -0.8) to (1.2, 0.8);
    \end{tikzpicture}
\caption{Geometric interpretations of modulus 1 cross ratios}
\end{figure}
\end{proof}

%\begin{corollary}
%Equivalently, the cross ratio $[A, B, C, D] \in S^1$ if and only if there exists a projective $\R$ form in $\C$ such that there exists lifts... 
%\end{corollary}

\subsection*{Cross Ratios in Higher Dimensions}
In spaces of complex dimension greater than 1, we are not able to take cross ratios directly since a collection of four points may not lie on a single complex line. In higher dimensions we take cross ratios coming from a quadruple of flags rather than a quadruple of points. A flag is a sequence of nested subspaces $\set{\vec{0}} \subsetneq F_1 \subsetneq F_2 \subsetneq \cdots \subsetneq F_k = \C^k$ where $\dim F_j = j$. For a flag $F$, let $F_j$ denote the subspace of dimension $j$. Given a transformation $G$ with eigendirections $g_1, g_2, \cdots, g_k$, we construct flags 
\begin{align*}
    F &= \set{\vec{0}} \subsetneq \spn{g_1} \subsetneq \spn{g_1} \oplus \spn{g_2} \subsetneq \cdots \subsetneq \spn{g_1} \oplus \spn{g_2} \oplus \cdots \oplus \spn{g_k} \\
    F' &= \set{\vec{0}} \subsetneq \spn{g_k} \subsetneq \spn{g_k} \oplus \spn{g_{k-1}} \subsetneq \cdots \subsetneq \spn{g_k} \oplus \spn{g_{k-1}} \oplus \cdots \oplus \spn{g_1}
\end{align*}
A flag $F$ is associated with some ordered basis $\set{v_1, \cdots, v_k}$ if for every $j$ we have $F_j = \text{Span } \set{v_1, \cdots, v_j}$. We say that $F$ and $F'$ above form a flag pair since they are associated to the reverse ordered basis of each other. 

A pair of flags $\set{A, B}$ in $\C^k$ is in generic position if $A_j \cap B_{k-j} = \set{\vec{0}}$ for every $j$. A set of flags $\set{A, B, \cdots, Z}$ in $\C^k$ is said to be in generic position if $A_{i_A} \oplus B_{i_B} \oplus \cdots Z_{i_Z} = \C^k$ for every $i_A + i_B + \cdots i_Z = k$. Note that generic position then implies $A_{i_A} \cap \left( B_{i_B} \oplus \cdots \oplus Z_{i_Z} \right) = \set{\vec{0}}$ and similar relations for every $i_A + i_B + \cdots i_Z = k$ and each flag.

Given two diagonalizable transformations $G$ and $H$ in $\pgl{k, \C}$, each without repeated eigenvalues, construct flags $A$ and $C$ from an eigenbasis for $G$ and $B$ and $D$ from and eigenbasis for $H$. If $A, B, C, D$ are in generic position, then we can take the cross ratios of projections coming from the quadruple of flags. For each $i+j = k-2$ with $i, j \geq 0$, let
\begin{align*}
    A^{i,j} &= \left(A_{i+1} \oplus C_j\right) / (A_i \oplus C_j) \\
    B^{i,j} &= \left(B_{1} \oplus (A_i \oplus C_j)\right) / (A_i \oplus C_j) \\
    C^{i,j} &= \left(C_{j+1} \oplus A_i\right) / (A_i \oplus C_j) \\
    D^{i,j} &= \left(D_1 \oplus (A_i \oplus C_j)\right) / (A_i \oplus C_j)
\end{align*}
and let $[A, B, C, D]$ denote the set of $k-1$ cross ratios $[A^{i,j}, B^{i,j}, C^{i,j}, D^{i,j}]$. These cross ratios are defined because the four points $A^{i,j}, B^{i,j}, C^{i,j}, D^{i,j}$ are contained in a single complex line. 

We will describe in more detail how to set up the flags coming from the eigenbases and how to take the cross ratios in the coming sections. 

\subsection{Cross Ratios Preserving Common Projective \texorpdfstring{$\R$}{R}-Form}
We know that hyperbolic eigendirections lie on the preserved projective $\R$ form and elliptic eigendirections pairs are swapped by inversion about the preserved projective $\R$ form. Conjugation action on a projective frame determines a unique $\R$-form up to collectively rotating the basis by $S^1 / S^0$, and therefore a unique projective $\R$ form preserved by the associated transformations. Thus, given flag pairs $A, C$ and $B, D$ with $A, B, C, D$ in generic position, we can order the eigenbases for the transformations so that we need at most the flag pair $A, C$ and the one dimensional parts of $B, D$ to determine a unique projective $\R$ form which might be preserved by the associated transformations. 

\begin{restatable}{lemma}{fgHH}
Given hyperbolic transformations with flag pairs $A, C$ and $B, D$, with $A, B, C, D$ in generic position, if the transformations preserve a common projective $\R$ form then every cross ratio coming from $[A, B, C, D]$ is real. 
\end{restatable}
\begin{proof}
Normalize so that $A, C$ are flags coming from the standard basis
\begin{align*}
    A &= \spn{e_1} \subseteq \spn{e_1} \oplus \spn{e_2} \subseteq \cdots \subseteq \spn{e_1} \oplus \cdots \spn{e_k} \\
    C &= \spn{e_k} \subseteq \spn{e_k} \oplus \spn{e_{k-1}} \subseteq \cdots \subseteq \spn{e_k} \oplus \cdots \spn{e_1}
\end{align*}
Further, normalize so that $D_1 = \spn{e_1 + e_2 + \cdots + e_k}$ is the sum of standard basis vectors. These $k+1$ eigendirections determine a unique candidate projective $\R$ form through the collection of eigendirections, that coming from $\R^k \subseteq \C^k$ in this basis. Since a preserved projective $\R$ form must contain all hyperbolic eigendirections, to preserve a common projective $\R$ form we must have that $B_1$ also lives in the projective $\R$ form through the $k+1$ eigendirections. The cross ratios coming from $[A, B, C, D]$ with this normalization are the ratios of consecutive coordinates of $B_1$, with none being $0$ by genericity. Thus, $B_1$ lives in the same projective $\R$ form as the eigendirections giving rise to $A, C,$ and $D_1$ if and only if these cross ratios are all real and hence the projective transformations cannot preserve a common projective $\R$ form if not all of $[A, B, C, D]$ are real. 

\end{proof}

\subsection{Triple Ratios in \texorpdfstring{$\C P^2$}{CP2}}
Triple ratios are analogous to cross ratios in one higher dimension, given a set of three flags $A, B, C$ in $\C P^2$. If $v_a, v_b, v_c$ are direction vectors for the lines $A_1, B_1, C_1$ and $f_a, f_b, f_c \in (\mathbb{C})^*$ are linear forms defining the planes $A_2, B_2, C_2$ then the triple ratio if $A, B, C$ is defined as
    $$r_3(A, B, C) = \frac{f_a (v_b) \; f_b (v_c) \; f_c (v_a)}{f_a (v_c) \; f_b (v_a) \; f_c (v_b)}$$

Together with cross ratios, the triple ratios are also used to determine if the corresponding transformations are simultaneously conjugate into $\pgl{3, \R}$.

\begin{figure}[H]
\begin{tikzpicture}[scale=1]
    %Dashed lines
    \draw[dashed] (-3, -2.3) -- (0.2, 3.3);
    \draw[dashed] (-3.2, -2) -- (3.4, -2);
    \draw[dashed] (-0.2, 3.3) -- (3, -2.3);
    %Solid lines
    \draw (-1.3, 0.6) -- (1.3, 0.6) -- (0, -2) -- (-1.3, 0.6);
    %Vertices
    \draw[fill=black] (-1.3, 0.6) circle (2pt);
    \draw[fill=black] (1.3, 0.6) circle (2pt);
    \draw[fill=black] (0, -2) circle (2pt);
    \node[right] at (1.3, 0.6) {$A_1$};
    \node[below] at (0.25, -2.25) {$B_1$};
    \node[left] at (-1.3, 0.6) {$C_1$};
    %Dashed line equations
    \node[right] at (0.75, 2) {$A_1 \oplus A_2 = A_2$};
    \node[left] at (-2.3, -0.9) {$C_1 \oplus C_2 = C_2$};
    \node[right] at (3.3, -2) {$B_1 \oplus B_2 = B_2$};
\end{tikzpicture}
\caption{Three Flags in $\pgl{3, \C}$}
\end{figure}
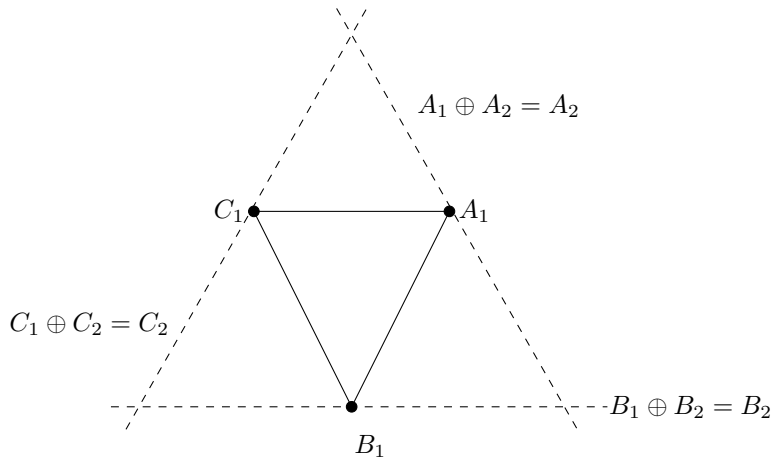

In the figure, the triple ratio is, in some sense, determining the intersections of the two dimensional portions of the flags. After we normalize as many degrees of freedom in the flags as possible and calculate the cross ratios, the triple ratio is then explicitly determining these intersections.  

Together cross ratios and triple ratios are used to construct the so called Fock-Goncharov coordinates on a surface. By simple computation, it is clear that each is a projective invariant and therefore the cross ratios and triple ratios coming from flags of eigendirections will not be changed by a simultaneous matrix conjugation which changes the collection of eigendirections. 

\section{Fock-Goncharov Coordinates for Conjugation into \texorpdfstring{$\pgl{2, \C}$}{PGL(2, C)}}
In the smallest case, we would like to know when a collection of elements in $\pgl{2, \C}$ is simultaneously conjugate into $\pgl{2, \R}$. Projective transformations in $\pgl{2, \R}$ will necessarily preserve the extended real line $\R \cup \set{\infty} = \hat{\R}$ in $\C P^1$, so it is necessary for each element in a subgroup of $\pgl{2, \C}$ to preserve a common projective $\R$ form for that subgroup to be conjugate into $\pgl{2, \R}$. On the other hand, only elements projectively in $\pgl{2, \R}$ preserve $\hat{\R}$, which can be seen by considering the action on $0 = [0, 1], 1 = [1, 1]$, and $\infty = [1, 0]$ in $\C P^1$. 

Thus, it is necessary and sufficient for each element to preserve a common projective $\R$ form for the subgroup to be conjugate into $\pgl{2, \R}$. We will use the eigenvalues and eigendirections of projective transformations to determine if each individual element preserves some projective $\R$ form and, if so, when each element in a collection preserves a common projective $\R$ form. 

\begin{restatable}{lemma}{pgltwoHH}
Two hyperbolic transformations in $\pgl{2, \C}$ without repeated eigenvalues are simultaneously conjugate into $\pgl{2, \R}$ if and only if the cross ratio of their eigendirections is in $\hat{\R}$.
\end{restatable}
\begin{proof}
The projective $\R$ forms preserved by a hyperbolic transformation in $\pgl{2, \C}$ are those through its pair of fixed points. Therefore, for two hyperbolic transformations to preserve a common projective $\R$ form, it is necessary and sufficient for the fixed points to all lie on a single projective $\R$ form. It is well known that the cross ratio of four points in the extended complex plane is real if and only if those four points lie on a circle or line, so we have that for a pair of hyperbolic transformations to be simultaneously conjugate into $\pgl{2, \R}$ it is necessary for the cross ratio of their pair of pairs of fixed points to be real. For hyperbolic transformations this is sufficient, as the circle or line the fixed points lie on will itself be preserved by both transformations and can be mapped by conjugation to the projective $\R$ form coming from the quotient of $\R^2 \subseteq \C^2$.

\end{proof}

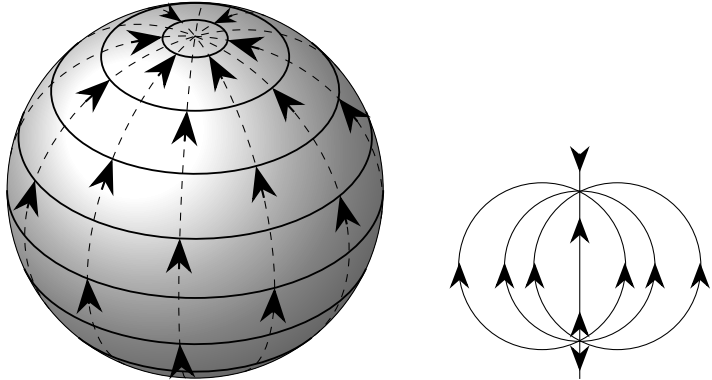
\begin{figure}[H]
\begin{tikzpicture}[scale=2.5]

    %Define planes intersecting the sphere
    \LongitudePlane[planeN2]{-\angEl}{185}
    \LongitudePlane[planeN1]{-\angEl}{155}
    \LongitudePlane[plane0]{-\angEl}{125}
    \LongitudePlane[plane1]{-\angEl}{95}
    \LongitudePlane[plane2]{-\angEl}{65}
    \LongitudePlane[plane3]{-\angEl}{35}
    \LongitudePlane[plane4]{-\angEl}{5}

    %Draw the ball
    \filldraw[ball color=white] (0, 0) circle (\Rad);
    %Draw the longitudes and latitudes
    \foreach \t in {-80, -60, ..., 80} {\DrawLatitudeCircle[\Rad]{\t}}
    \foreach \t in {-5, -35, ..., -175} {\DrawLongitudeCircle[\Rad]{\t}}

    %Draw latitude arrows
    \foreach \t in {5} {\draw[planeN2, -{Stealth[scale=2]}, thick] (-30+\t*20:\Rad) to (-20+\t*20:\Rad);}
    \foreach \t in {2, 4} {\draw[planeN1, -{Stealth[scale=2]}, thick] (-30+\t*20:\Rad) to (-20+\t*20:\Rad);}
    \foreach \t in {1, 3, 5} {\draw[plane0, -{Stealth[scale=2]}, thick] (-30+\t*20:\Rad) to (-20+\t*20:\Rad);}
    \foreach \t in {0, 2, 4} {\draw[plane1, -{Stealth[scale=2]}, thick] (-30+\t*20:\Rad) to (-20+\t*20:\Rad);}
    \foreach \t in {1, 3, 5} {\draw[plane2, -{Stealth[scale=2]}, thick] (-30+\t*20:\Rad) to (-20+\t*20:\Rad);}
    \foreach \t in {2, 4} {\draw[plane3, -{Stealth[scale=2]}, thick] (-30+\t*20:\Rad) to (-20+\t*20:\Rad);}
    \foreach \t in {3, 5} {\draw[plane4, -{Stealth[scale=2]}, thick] (-30+\t*20:\Rad) to (-20+\t*20:\Rad);}
    \foreach \t in {6} {\draw[plane0, -{Stealth[scale=1.25]}, thick] (-10+\t*20:\Rad) to (-20+\t*20:\Rad);}
    \foreach \t in {6} {\draw[plane2, -{Stealth[scale=1.25]}, thick] (-10+\t*20:\Rad) to (-20+\t*20:\Rad);}
\end{tikzpicture}
\hspace{0.5cm}
\begin{tikzpicture}[scale=0.5]
    \draw (0, -3) -- (0, 3);
    \draw (0, 0) circle (2);
    \draw (1, 0) circle (2.22);
    %\draw (2, 0) circle (2.82);
    \draw (-1, 0) circle (2.22);
    %\draw (-2, 0) circle (2.82);
    %Arrows
    \draw[thick, -{Stealth[scale=1.5]}] (0, -1.3) -- (0, -1.2);
    \draw[thick, -{Stealth[scale=1.5]}] (0, -2.7) -- (0, -2.8);
    \draw[thick, -{Stealth[scale=1.5]}] (-1.2, 0) -- (-1.2, 0.1);
    \draw[thick, -{Stealth[scale=1.5]}] (-2, 0) -- (-2, 0.1);
    \draw[thick, -{Stealth[scale=1.5]}] (-3.2, 0) -- (-3.2, 0.1);
    \draw[thick, -{Stealth[scale=1.5]}] (0, 1.2) -- (0, 1.3);
    \draw[thick, -{Stealth[scale=1.5]}] (0, 2.6) -- (0, 2.5);
    \draw[thick, -{Stealth[scale=1.5]}] (1.2, 0) -- (1.2, 0.1);
    \draw[thick, -{Stealth[scale=1.5]}] (2, 0) -- (2, 0.1);
    \draw[thick, -{Stealth[scale=1.5]}] (3.2, 0) -- (3.2, 0.1);
\end{tikzpicture}
\caption{Action of a Hyperbolic Transformation in $\pgl{2, \C}$ on $\C P^1$ and $\hat{\C}$}
\label{hyperbolicFigure}
\end{figure}

It is also well known that the projective $\R$ forms preserved by an elliptic transformation in $\pgl{2, \C}$ are those which are perpendicular in $\hat{\C}$ to the projective $\R$ forms through the two fixed points of the transformation. Hence, for two elliptic transformations to preserve a common projective $\R$ form it is necessary, but not sufficient, to have their fixed points all on a single projective $\R$ form. Equivalently, it is necessary for the cross ratio of the fixed points to be real. 

Circle inversion in $\hat{\C}$ about any projective $\R$ form preserved by an elliptic transformation will map its fixed points to each other, so having a real cross ratio is not sufficient for two elliptic transformations to preserve a common projective $\R$ form. An example can be seen by considering the elliptic transformation fixing $0$ and $\infty$ which preserves circles centered at the origin, and the elliptic transformation fixing $1$ and $-1$ which will not be inverted by any circle centered at the origin. 

\begin{restatable}{lemma}{pgltwoEE}
Two elliptic transformations $P$ and $Q$ in $\pgl{2, \C}$ without repeated eigenvalues and with eigendirections $p^-, p^+$ and $q^-, q^+$ are simultaneously conjugate into $\pgl{2, \R}$ if and only if $[p^-, q^-, p^+, q^+] \in \R^+$. 
\end{restatable}
\begin{proof}
Geometrically, two elliptic transformations in $\pgl{2, \C}$ are simultaneously conjugate into $\pgl{2, \R}$ if the fixed points lie on a single projective $\R$ form and the two fixed points of one transformation do not separate the fixed points of the other transformation along the projective $\R$ form through those fixed points. If we order the points of the cross ratio as $[p^-, q^-, p^+, q^+]$ for fixed point pairs $p^-, p^+$ of one transformation and $q^-, q^+$ of the other, this means the cross ratio is real and positive. This can be seen as both necessary and sufficient by using M\"obius transformations transitivity on triples of points in $\hat{\C}$ to normalize such that $p^- = \infty, p^+ = 0, q^+ = 1$ and considering the cross ratio 
    $$[p^-, q^-, p^+, q^+] = [\infty, q^-, 0, 1] = \frac{(\infty - 1)(0 - q^-)}{(\infty - q^-)(0 - 1)} = q^-$$ 
The circles preserved by an elliptic transformation fixing $0$ and $\infty$ are those centered at the origin. Since circle inversion preserves rays from the circle center, if $q^-$ not real, then no circle inversion about a circle centered at the origin will map $q^+$ to $q^-$. Similarly, if $q^-$ is real but negative then no circle inversion centered at the origin can swap the map $q^+$ to $q^-$. 

On the other hand, if $q^-$ is real and positive, then inversion about the circle of radius $\sqrt{q^-}$ will map $1 \leftrightarrow q^-$. This circle is the common preserved projective $\R$ form of the two elliptic transformation and can be mapped to the extended real line by a conjugation which carries $P$ and $Q$ into $\pgl{2, \R}$. 

\end{proof}

Note that in each of the previous lemmas, the eigenvalues are assumed to have $\pgl{2, \R}$ compatible eigenvalues because they are hyperbolic or elliptic transformations. This will continue throughout the paper. 

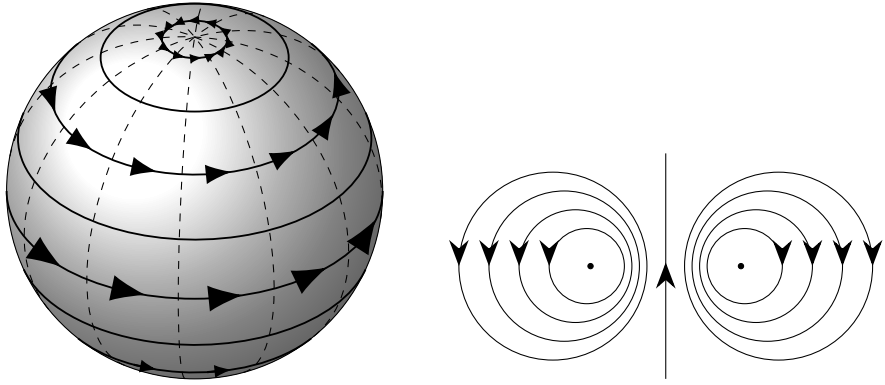
\begin{figure}[H] 
\begin{tikzpicture}[scale=2.5]

    %Define planes intersecting the sphere
    \LatitudePlane[equator]{\angEl}{0}
    \LatitudePlane[cancer]{\angEl}{40} %Not actually 23 degrees
    \LatitudePlane[arctic]{\angEl}{80} %Not actually 66 degrees
    \LatitudePlane[capricorn]{\angEl}{-40} %Not actually -23

    %Draw the ball
    \filldraw[ball color=white] (0, 0) circle (\Rad);
    %Draw the longitudes and latitudes
    \foreach \t in {-80, -60, ..., 80} {\DrawLatitudeCircle[\Rad]{\t}}
    \foreach \t in {-5, -35, ..., -175} {\DrawLongitudeCircle[\Rad]{\t}}

    %Draw latitude arrows
    \foreach \t in {0, 1, ..., 4} {\draw[equator, -{Latex[scale=2]}, thick] (-145+\t*30:\Rad) to (-135+\t*30:\Rad);}
    \foreach \t in {-1, 0, ..., 5} {\draw[cancer, -{Latex[scale=1.5]}, thick] (-145+\t*30:\Rad) to (-135+\t*30:\Rad);}
    \foreach \t in {0, 1, ..., 11} {\draw[arctic, -{Latex[scale=0.75]}, thick] (-145+\t*30:\Rad) to (-135+\t*30:\Rad);}
    \foreach \t in {1, 2} {\draw[capricorn, -{Latex[scale=1]}, thick] (-145+\t*30:\Rad) to (-135+\t*30:\Rad);}
\end{tikzpicture}
\hspace{.5cm}
\begin{tikzpicture}[scale=0.5]
    \draw (0, -3) -- (0, 3);
    \draw[fill=black] (-2, 0) circle (2pt);
    \draw[fill=black] (2, 0) circle (2pt);
    \draw (-2.1, 0) circle (1);
    \draw (-2.4, 0) circle (1.5);
    \draw (-2.7, 0) circle (2);
    \draw (-3, 0) circle (2.5);
    \draw (2.1, 0) circle (1);
    \draw (2.4, 0) circle (1.5);
    \draw (2.7, 0) circle (2);
    \draw (3, 0) circle (2.5);
    %Arrows
    \draw[thick, -{Stealth[scale=1.5]}] (0, 0) -- (0, 0.1);
    \draw[thick, -{Stealth[scale=1.5]}] (-3.1, 0.1) -- (-3.1, 0);
    \draw[thick, -{Stealth[scale=1.5]}] (-3.9, 0.1) -- (-3.9, 0);
    \draw[thick, -{Stealth[scale=1.5]}] (-4.7, 0.1) -- (-4.7, 0);
    \draw[thick, -{Stealth[scale=1.5]}] (-5.5, 0.1) -- (-5.5, 0);
    \draw[thick, -{Stealth[scale=1.5]}] (3.1, 0.1) -- (3.1, 0);
    \draw[thick, -{Stealth[scale=1.5]}] (3.9, 0.1) -- (3.9, 0);
    \draw[thick, -{Stealth[scale=1.5]}] (4.7, 0.1) -- (4.7, 0);
    \draw[thick, -{Stealth[scale=1.5]}] (5.5, 0.1) -- (5.5, 0);
\end{tikzpicture}
\caption{Action of an Elliptic Transformation in $\pgl{2, \C}$ on $\C P^1$ and $\hat{\C}$}
\label{ellipticFigure}
\end{figure}

\begin{restatable}{lemma}{pgltwoHHH}
Given a collection $\set{H_1, H_2, \cdots, H_k}$ of hyperbolic transformations in $\pgl{2, \C}$ without repeated eigenvalues and where $H_1$ and $H_2$ do not share the same pair of eigendirections, the collection is simultaneously conjugate into $\pgl{2, \R}$ if and only if
\begin{enumerate}
    \item $[H_1, H_2] \in \hat{\R}$
\end{enumerate}
and, denoting the eigendirections of $H_j$ by $h_j^-$ and $h_j^+$, for each $H_i$ with $i > 2$
\begin{enumerate}[resume]
    \item $[h_1^-, h_i^-, h_1^+, h_2^-] \in \hat{\R}$
    \item $[h_1^-, h_i^+, h_1^+, h_2^-] \in \hat{\R}$
\end{enumerate}
\end{restatable}
\begin{proof}
Since hyperbolic transformations preserve the projective $\R$ forms through their fixed points, a collection of transformations in $\pgl{2, \C}$ is simultaneously conjugate into $\pgl{2, \R}$ if and only if their fixed points all lie on a single projective $\R$ form. Without loss of generality, assume $H_1$ has eigendirections $\infty$ and $0$ and $H_2$ has an eigendirection at $1$. Then the only projective $\R$ form which might be preserved by both $H_1$ and $H_2$ is the extended real line in this basis, so as before $H_1$ and $H_2$ are simultaneously conjugate into $\pgl{2, \R}$ if and only if $[H_1, H_2] \in \hat{\R}$. 

If $H_1$ and $H_2$ preserve a common projective $\R$ form, normlize so that $h_1^- = \infty, h_1^+ = 0$, and $h_2^- = 1$. For any other transformation $H_i$, the transformation also preserves the same projective $\R$ form if and only if its eigendirections lie on the same projective $\R$ form as those of $H_1$ and $H_2$. Equivalently, $H_i$ also preserves a common projective $\R$ form if and only if its eigendirections are on the extended real line. We consider the pair of cross ratios
    $$[h_1^-, h_i^\pm, h_1^+, h_2^-] = \frac{(h_1^- - h_2^-)(h_1^+ - h_i^\pm)}{(h_1^- - h_i^\pm)(h_1^+ - h_2^-)}$$
which become
    $$[\infty, h_i^\pm, 0, 1] = h_i^\pm$$
so that the cross ratios are in the extended reals if and only if $h_i^-$ and $h_i^+$ are in $\hat{\R}$. Hence, $H_i$ preserves the same projective $\R$ form, and is therefore simultaneously conjugate into $\pgl{2, \R}$, if and only if $[h_1^-, h_i^-, h_1^+, h_2^-] \in \hat{\R}$ and $[h_1^-, h_i^+, h_1^+, h_2^-] \in \hat{\R}$.

\end{proof}

\begin{restatable}{lemma}{circleInversion}
If inversion about a circle $C$ in $\hat{\C}$ maps $q^- \mapsto q^+$ and $q^+ \mapsto q^-$, then any circle $\Delta$ through $q^-$ and $q^+$ intersects $C$ perpendicularly in $\hat{\C}$. 
\end{restatable}
\begin{proof}
Circle inversion maps circles to circles and fixes the circle being inverted about. In particular, this means that inversion about $C$ fixes $C \cap \Delta = \set{*_1, *_2}$. Thus, $q^-, q^+$, and $*_1$ are mapped to $\Delta$, so inversion about $C$ preserves $\Delta$ but swaps the part inside $C$ and the part outside $C$. Since circle inversion is anti-conformal, this implies that $C$ and $\Delta$ intersect perpendicularly in $\hat{\C}$. 

\end{proof}

\begin{restatable}{lemma}{pgltwoEEE}
Given a collection $\set{E_1, E_2, \cdots, E_n}$ of elliptic transformations in $\pgl{2, \C}$ without repeated eigenvalues and where $E_1$ and $E_2$ do not share the same pair of eigendirections, the collection is simultaneously conjugate into $\pgl{2, \R}$ if and only if
\begin{enumerate}
    \item $[E_1, E_2] \in \R^+$
\end{enumerate}
and for each $E_i$ with $i > 2$
\begin{enumerate}[resume]
    \item $[E_1, E_i] \in \R^+$
    \item $[E_2, E_i] \in \R^+$
\end{enumerate}
\end{restatable}
\begin{proof}
As in the case of $2$ elliptic transformations, for $n$ elliptic transformations to preserve a common projective $\R$ form it is necessary for the fixed points of each pair of transformations to lie on a single projective $\R$ form. Again note that circle inversion is anti-conformal, so maps circles to circles. Since inversion about any projective $\R$ form preserved by an elliptic transformation swaps its fixed points, for three elliptic transformations to preserve a common projective $\R$ form, inversion about that common projective $\R$ form must preserve the circles through each pair of pairs of fixed points. 

Consider three elliptic transformations with fixed points $(p^-, p^+), (q^-, q^+),$ and $(r^-, r^+)$ respectively. Without loss of generality, assume that $p^- = \infty$, $p^+ = 0$, and $q^- = 1$. We know that for a pair of elliptic transformations to preserve a common projective $\R$ form it is necessary and sufficient for the cross ratio of their pair of pairs of fixed points to be real and positive. This pairwise condition is also sufficient for three elliptic transformations to preserve a common projective $\R$ form.

Suppose $[p^-, q^-, p^+, q^+], [p^-, r^-, p^+, r+],$ and $[q^-, r^-, q^+, r^+]$ are real and positive, so that in particular if we normalize so that $p^- = \infty, p^+ = 0, q^- = 1$ as in the previous lemmas, we have $q^+ \in \R^+$. The circle centered at the origin about which inversion maps $1 \mapsto q^+$ is the circle of radius $\sqrt{q^+}$, which will by symmetry of intersections necessarily intersect all three circles through a pair of pairs of fixed points perpendicularly. Thus, inversion about this circle will map the circle through each pair of pairs of fixed points to itself while also preserving rays from the origin. Therefore, inversion about this circle will swap each pair of fixed points and be preserved by all three elliptic transformations. 

By the previous lemma, any transformation which has cross ratios of eigendirections as described being real and positive will preserve this same projective $\R$ form and is thus simultaneously conjugate into $\pgl{2, \R}$. On the other hand, any transformation which is simultaneously conjugate into $\pgl{2, \R}$ will preserve a common projective $\R$ form. If we normalize so that this projective $\R$ form is the extended real line, then inversion about that projective $\R$ form is standard complex conjugation, so the eigendirection pairs are complex conjugates. Given a pair of complex numbers $z_1, z_2$ and their complex conjugates $\conj{z_1}, \conj{z_2}$, the circle centered at on the real axis and passing through $z_1$ and $z_2$ will contain all four points $z_1, z_2, \conj{z_1}, \conj{z_2}$ so the cross ratio $[z_1, z_2, \conj{z_2}, \conj{z_1}]$ is real. Further, is is clear that $z_1$ and $\conj{z_1}$ do not separate $z_2$ and $\conj{z_2}$ along this circle, so the cross ratio is real and positive. This applies to every required pair of eigendirections, completing the proof. 

\end{proof}
\begin{figure}
\begin{tikzpicture}[scale=1.25]
    \draw[->] (0, 0) to (4, 0);
    \draw[->] (0, 0) to (3, 3);
    \draw (0, 0) circle (1.7321);
    \draw (2, 1) circle (1.41);
    %Tangents
    \draw[dashed] (0, 0) to (3, -0.68);
    \draw[dashed] (0, 0) to (1.35, 3);
    %Points
    \draw[fill=black] (0, 0) circle (2pt);
    \draw[fill=black] (1, 0) circle (2pt);
    \draw[fill=black] (3, 0) circle (2pt);
    \draw[fill=black] (2.35, 2.35) circle (2pt);
    \draw[fill=black] (0.65, 0.65) circle (2pt);
    %Labels
    \node[below=0.1cm] at (0, 0) {$p^-$};
    \node[above=0.1cm] at (1, 0) {$q^-$};
    \node[above=0.1cm] at (3, 0) {$q^+$};
    \node[right=0.1cm] at (0.65, 0.65) {$r^-$};
    \node[below=0.1cm] at (2.35, 2.35) {$r^+$};
    \node at (6, 0.2) {$p^+ = \infty$};
\end{tikzpicture}
\caption{Three elliptic transformations in $\pgl{2, \C}$ preserving a common circle}
\end{figure}
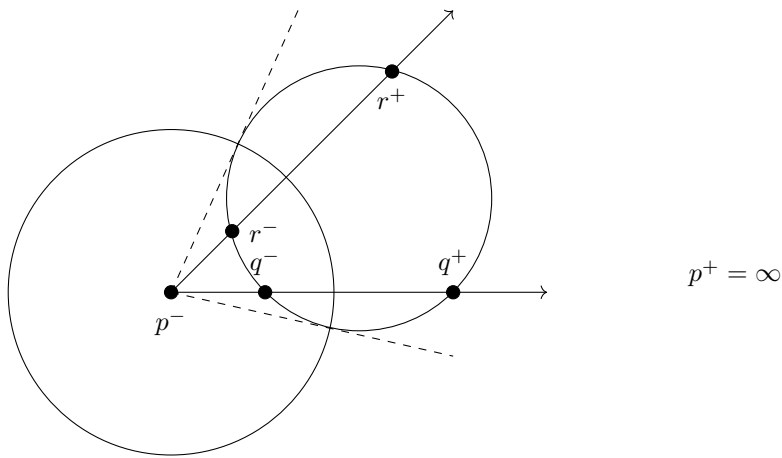

\begin{restatable}{lemma}{pgltwoHE}
A hyperbolic transformation in $\pgl{2, \C}$ with eigendirections $h^-,$ $h^+$ and an elliptic transformation in $\pgl{2, \C}$ with eigendirections $e^-, e^+$, each without repeated eigenvalues, are simultaneously conjugate into $\pgl{2, \R}$ if and only if the $[h^-, e^-, h^+, e^+] \in S^1$. 
\end{restatable}
\begin{proof}
Without loss of generality, assume the eigendirections of the hyperbolic transformation are at $\infty$ and $0$. Then the transformations preserve a common projective $\R$ form if and only if the eigendirections of the elliptic transformation can be inverted by reflection about a line through the origin. This is the case if and only if the magnitude of the eigendirections from the elliptic transformation are equal, and so the cross ratio $[h^-, e^-, h^+, e^+] = [\infty, e^-, 0, e^+] = \frac{e^-}{e^+}$ is in $S^1$. 

\end{proof}

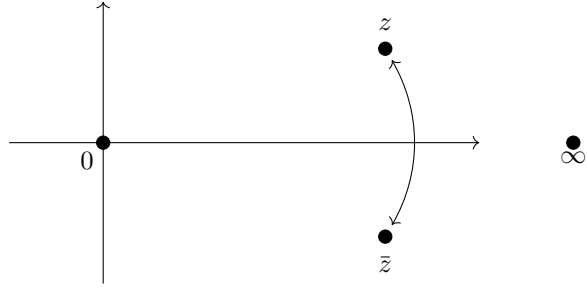
\begin{figure}
    \begin{tikzpicture}[scale=1.25]
        \draw[->] (-1, 0) -- (4, 0);
        \draw[->] (0, -1.5) -- (0, 1.5);
        \draw[fill=black] (0, 0) circle (2pt);
        \node[below left] at (0, 0) {0};
        \draw[fill=black] (5, 0) circle (2pt);
        \node[below] at (5, 0) {$\infty$};
        \draw[fill=black] (3, 1) circle (2pt);
        \node[above] at (3, 1.1) {$z$};
        \draw[fill=black] (3, -1) circle (2pt);
        \node[below] at (3, -1.1) {$\conj{z}$};
        \draw[bend left=30, <->, shorten <=5pt, shorten >= 5pt] (3, 1) to (3, -1);
    \end{tikzpicture}
    \caption{A hyperbolic and an elliptic transformation in $\pgl{2, \C}$}
\end{figure}

\begin{restatable}{theorem}{pgltwoOne}
Given a collection $\set{H_1, H_2, \cdots, H_m, E_1, \cdots, E_n}$ of $m \geq 2$ hyperbolic transformations and $n \geq 0$ elliptic transformations in $\pgl{2, \C}$, each without repeated eigenvalues and where $H_1$ and $H_2$ do not share the same pair of eigendirections, the collection is simultaneously conjugate into $\pgl{2, \R}$ if and only if
\begin{enumerate}
    \item $[H_1, H_2] \in \hat{\R}$
\end{enumerate}
for each $H_i$ with $i > 2$
\begin{enumerate}[resume]
    \item $[h_1^-, h_i^-, h_1^+, h_2^-] \in \hat{\R}$
    \item $[h_1^-, h_i^+, h_1^+, h_2^-] \in \hat{\R}$
\end{enumerate}
and for each $E_i$
\begin{enumerate}[resume]
    \item $[h_1^-, e_i^-, h_1^+, h_2^-] = \conj{[h_1^-, e_i^+, h_1^+, h_2^-]}$
\end{enumerate}
\end{restatable}
\begin{proof}
We know by a previous lemma that for $m$ hyperbolic transformations to preserve a common copy of $\R P^1$ it is necessary and sufficient for $[H_1, H_2] \in \hat{\R}$, $[h_1^-, h_2^-, h_2^+, h_i^-] \in \hat{\R} \; \forall i > 2$, and $[h_1^-, h_2^-, h_2^+, h_i^+] \in \hat{\R} \; \forall i > 2$. Thus, this theorem concerns conditions which are necessary and sufficient for elliptic transformations in $\pgl{2, \C}$ to also be simultaneously conjugate into $\pgl{2, \R}$. 

We know that for the $n$ elliptic transformations in $\pgl{2, \C}$ to preserve a common projective $\R$ form and be conjugate into $\pgl{2, \R}$ it is necessary and sufficient for $[E_1, E_2] \in \R^+$, $[E_1, E_i] \in \R \; \forall i > 2$, and $[E_2, E_i] \in \R \; \forall i > 2$. We do \textbf{not} use this condition, as we show that these relations follow from the cross ratios in condition \textbf{d}. 

Without loss of generality, assume that $h_1^- = \infty, h_1^+ = 0, h_2^- = 1$ so that the only projective $\R$ form which might be preserved by the collection is the extended real line. Then the cross ratios become $[h_1^-, e_i^\pm, h_1^+, h_2^-] = e_i^\pm$. The eigendirections of an elliptic transformation $E_i$ are swapped by inversion about the extended real line if and only if $e_i^- = \conj{e_i^+}$, so $[h_1^-, e_i^-, h_1^+, h_2^-] = \conj{[h_1^-, e_i^+, h_1^+, h_2^-]}$ if and only if $E_i$ is simultaneously conjugate into $\pgl{2, \R}$ with the hyperbolic transformations. The result and therefore previous conditions of $[E_1, E_2] \in \R^+$, $[E_1, E_i] \in \R \; \forall i > 2$, and $[E_2, E_i] \in \R \; \forall i > 2$ follow. 

\end{proof}

\begin{restatable}{theorem}{pgltwoTwo}
Given a collection $\set{E_1, E_2, \cdots, E_n, H_1, \cdots, H_m}$ of $n \geq 2$ elliptic transformations and $m \geq 0$ hyperbolic transformation in $\pgl{2, \C}$, each without repeated eigenvalues and where $E_1$ and $E_2$ do not share the same pair of eigenectors, the collection is simultaneously conjugate into $\pgl{2, \R}$ if and only if
\begin{enumerate}
    \item $[E_1, E_2] \in \R^+$
\end{enumerate}
for each $E_i$ with $i > 2$
\begin{enumerate}[resume]
    \item $[E_1, E_i] \in \R^+$
    \item $[E_2, E_i] \in \R^+$
\end{enumerate}
and denoting the eigendirections of $E_i$ by $e_i^-, e_i^+$ and of $H$ by $h^-, h^+$, for each $H_i$
\begin{enumerate}[resume]
    \item $[e_1^-, h_i^-, e_1^+, e_2^-] \cdot [e_1^-, h_i^-, e_1^+, e_2^+] \in S^1$
    \item $[e_1^-, h_i^+, e_1^+, e_2^-] \cdot [e_1^-, h_i^+, e_1^+, e_2^+] \in S^1$
\end{enumerate}
\end{restatable}
\begin{proof}
We know by a previous lemma that for $n$ elliptic transformations to preserve a common projective $\R$ form it is necessary and sufficient for $[E_1, E_2] \in \R^+$ and for each $i > 2$ $[E_1, E_i] \in \R^+$ and $[E_2, E_i] \in \R^+$. 

For the hyperbolic transformations to also be simultaneously conjugate into $\pgl{2, \R}$, the fixed points of each hyperbolic transformation must be on the projective $\R$ form preserved by the elliptic transformations. Without loss of generality, assume that the fixed points of the elliptic transformation $E_1$ are at $\infty$ and $0$ such that the only possible common preserved projective $\R$ forms are circles in the complex plane centered at the origin. Further, suppose that the remaining elliptic transformations imply the only possible preserved projective $\R$ form is the unit circle, which can be done by scaling from the origin, and that the eigendirections of the elliptic transformation $E_2$ are at $r$ and $1/r$, for $r \in \R$, which can be done by rotating about the origin. Then a hyperbolic transformation is simultaneously conjugate into $\pgl{2, \R}$ if and only if its fixed points lie on the unit circle. If the fixed points of some hyperbolic transformation are $h_i^-, h_i^+$, the cross ratios simplify to $[e_1^-, h_i^-, e_1^+, e_2^-] = [\infty, h_i^-, 0, r] = \frac{h_i^-}{r}$ and $[e_1^-, h_i^-, e_1^+, e_2^+] = [\infty, h_i^-, 0, 1/r] = r h_i^-$ for $h_i^-$. Similarly, the cross ratios simplify to $[e_1^-, h_i^+, e_1^+, e_2^-] = [\infty, h_i^+, 0, r] = \frac{h_i^+}{r}$ and $[e_1^-, h_i^+, e_1^+, e_2^+] = [\infty, h_i^+, 0, 1/r] = r h_i^+$ for $h_i^+$. Since the product of the two cross ratios for $h_i^-$ gives $(h_i^-)^2$ and the product of the two cross ratios for $h_i^+$ gives $(h_i^+)^2$, the hyperbolic transformation is also simultaneously conjugate into $\pgl{2, \R}$ if and only if these two products of cross ratios are in $S^1$. 

\end{proof}

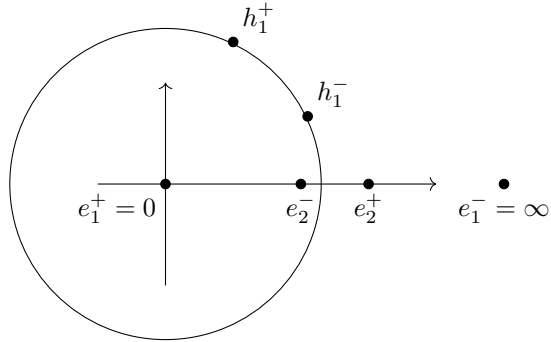
\begin{figure}
    \begin{tikzpicture}[scale=0.9]
        \draw[->] (-1, 0) -- (4, 0);
        \draw[->] (0, -1.5) -- (0, 1.5);
        \draw[fill=black] (0, 0) circle (2pt);
        \node[below left] at (0, 0) {$e_1^+ = 0$};
        \draw[fill=black] (5, 0) circle (2pt);
        \node[below] at (5, 0) {$e_1^- = \infty$};
        \draw[fill=black] (2, 0) circle (2pt);
        \node[below] at (2, 0) {$e_2^-$};
        \draw[fill=black] (3, 0) circle (2pt);
        \node[below] at (3, 0) {$e_2^+$};
        \draw (0, 0) circle (2.3);
        \draw[fill=black] (2.1, 1) circle (2pt);
        \node[above right] at (2.1, 1) {$h_1^-$};
        \draw[fill=black] (1, 2.1) circle (2pt);
        \node[above right] at (1, 2.1) {$h_1^+$};
    \end{tikzpicture}
    \caption{One hyperbolic and multiple elliptics transformations in $\pgl{2, \C}$}
\end{figure}

\section{Fock-Goncharov Coordinates for Conjugation into \texorpdfstring{$\pgl{3, \R}$}{PGL(3, R)}}

Transformations in $\pgl{3, \C}$ have eigendirections in $\C P^2$, so the fixed points do not necessarily lie on a single complex line and in general we are not able to take their cross ratios directly. Instead, we consider flags coming from the fixed points and take the cross ratios of quotients coming from these flags. 

Given flags $A, B, C, D$ in $\C^2$ in generic position, the four planes $A_2, A_1 \oplus B_1, A_1 \oplus C_1, A_1 \oplus D_1$ all share the common line $A_1$. We take the cross ratio of the four planes 
    $$[A_1 \oplus A_2, A_1 \oplus B_1, A_1 \oplus C_1, A_1 \oplus D_1]$$
by first projecting to $\C^3 / A_1$ to get four lines through the origin in $\C^2$, or equivalently four points in $\mathbb{C}P^1$. Similarly, the four planes $C_2, C_1 \oplus D_1, C_1 \oplus A_1, C_1 \oplus D_1$ share the common line $C_1$ so we take the cross ratio of the four planes
    $$[C_1 \oplus C_2, C_1 \oplus D_1, C_1 \oplus A_1, C_1 \oplus B_1]$$
by projecting to $\C^3 / C_1$ to get points in a copy of $\mathbb{C}P^1$.

However, we are able to take triple ratios directly, since the dimensionality is already correct. If $v_a, v_b, v_c, v_d$ are direction vectors for the lines $A_1, B_1, C_1, D_1$ and $f_a, f_b, f_c, f_d \in (\mathbb{C})^*$ are linear forms defining the planes $A_2, B_2, C_2, D_2$ then the triple ratios we take are
    $$r_3(A, B, C) = \frac{f_a (v_b) \; f_b (v_c) \; f_c (v_a)}{f_a (v_c) \; f_b (v_a) \; f_c (v_b)}$$
and
    $$r_3(A, C, D) = \frac{f_a (b_c) \; f_c (b_d) \; f_d (v_a)}{f_a (v_d) \; f_c (v_a) \; f_d (v_c)}$$

Given four flags $A, B, C, D$ in $\C^2$ coming from eigendirections, these two cross ratios and two triple ratios are used to determine if the corresponding transformations are simultaneously conjugate into $\pgl{3, \R}$.

\begin{figure}[H]
\begin{tikzpicture}[scale=1.3]
    %Dashed lines
    \draw[dashed] (-2, 2) -- (2.5, 2.5);
    \draw[dashed] (-1, 2.5) -- (-2, -2.5);
    \draw[dashed] (-2.5, -2) -- (3, -2.5);
    \draw[dashed] (2, 3) -- (2.5, -3);
    %Solid lines
    \draw (0.25, 2.25) -- (2.25, 0) -- (0.25, -2.25) -- (-1.5, 0) -- (0.25, 2.25);
    \draw (0.25, 2.25) -- (0.25, -2.25);
    %Vertices
    \draw[fill=black] (0.25, 2.25) circle (2pt);
    \draw[fill=black] (0.25, -2.25) circle (2pt);
    \draw[fill=black] (2.25, 0) circle (2pt);
    \draw[fill=black] (-1.5, 0) circle (2pt);
    \node[above] at (0.25, 2.25) {$A_1$};
    \node[below] at (0.25, -2.25) {$C_1$};
    \node[left] at (-1.5, 0) {$D_1$};
    \node[right] at (2.25, 0) {$B_1$};
    %Dashed line equations
    \node[right] at (2.5, 2.5) {$A_1 \oplus A_2 = A_2$};
    \node[right] at (3, -2.5) {$C_1 \oplus C_2 = C_2$};
    \node[right] at (2.5, -3) {$B_1 \oplus B_2 = B_2$};
    \node[left] at (-2, -2.5) {$D_1 \oplus D_2 = D_2$};
    %Coordinates
    \draw[fill=black] (0.25, 0.5) circle (2pt);
    \draw[fill=black] (0.25, -0.5) circle (2pt);
    \node[right] at (0.25, 0.5) {$a$};
    \node[right] at (0.25, -0.5) {$b$};
    \draw[fill=black] (-0.5, 0) circle (2pt);
    \node[above] at (-0.5, 0) {$t$};
    \draw[fill=black] (1, 0) circle (2pt);
    \node[above] at (1, 0) {$s$};
\end{tikzpicture}
\caption{Coordinates corresponding to two elements in $\pgl{3, \C}$}
\end{figure}
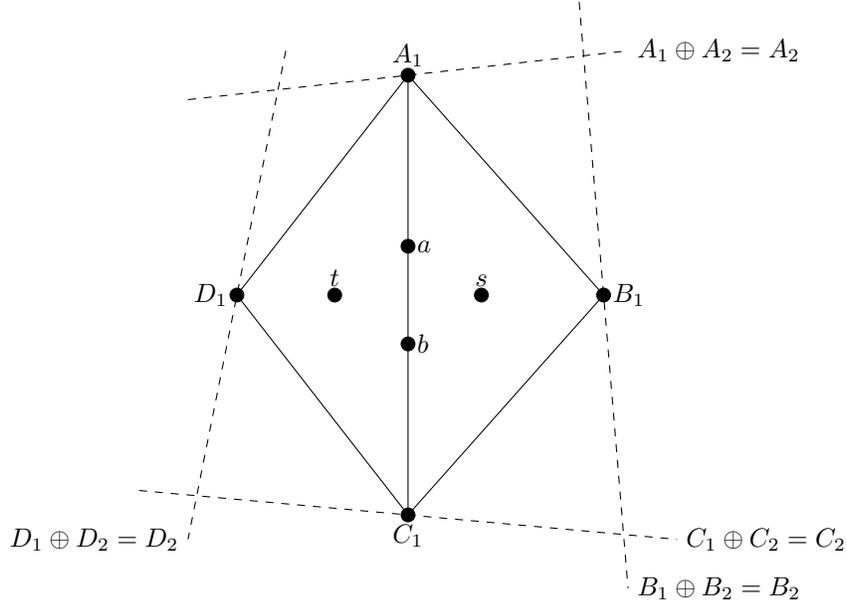

For the first two projective transformations $G, H$ in a collection, if $G$ has eigendirections $g_1, g_2, g_3$ and $H$ has eigendirections $h_1, h_2, h_3$ we create flags 
\begin{align*}
    A &= \spn{g_1} \subseteq \spn{g_1} \oplus \spn{g_2} \\
    C &= \spn{g_3} \subseteq \spn{g_3} \oplus \spn{g_2} \\
    B &= \spn{h_1} \subseteq \spn{h_1} \oplus \spn{h_2} \\
    D &= \spn{h_3} \subseteq \spn{h_3} \oplus \spn{h_2}
\end{align*}
and take the two cross ratios and two triple ratios of the quadruple of flags $A, B, C, D$ as above. 

For any subsequent projective transformation $M$ with eigendirections $m_1,$ $m_2, m_3$, make flags 
\begin{align*}
    \beta &= \spn{m_1} \subseteq \spn{m_1} \oplus \spn{m_2} \\
    \beta' &= \spn{m_3} \subseteq \spn{m_3} \oplus \spn{m_2} 
\end{align*}

Then we take the cross ratios and triple ratios as described above, using two different quadruples of flags depending on the types of projective transformations, each giving two new cross ratios and one new triple ratio. Depending on the type of $M$, these combined four cross ratios and two triple ratios will tell us is $M$ is also simultaneously conjugate into $\pgl{3, \R}$. 

\subsection{\texorpdfstring{$m \geq 2$}{m >= 2} Hyperbolic transformations}

\begin{restatable}{lemma}{pglthreeHH} \label{pglthreeHH}
Two hyperbolic elements $G$ and $H$ in $\pgl{3, \C}$ without repeated eigenvalues, with flags pairs $A, C$ and $B, D$ respectively, and where $A, B, C, D$ are in generic position are simultaneously conjugate into $\pgl{3, \R}$ if and only if
\begin{enumerate}
    \item their two cross ratios coming from $[A, B, C, D]$ are real
    \item their two triple ratios $r_3 (A, B, C)$ and $r_3 (A, C, D)$ are real
\end{enumerate}
\end{restatable}
\begin{proof}
$\Longrightarrow$ %(Projective transformations conjugate into $\pgl{3, \R}$ have all real coordinates as described)

Since cross ratios and triple ratios are invariant under conjguation, conjugate so that the collection of elements is in $\pgl{3, \R}$. Then the result is direct from the fact that all eigendirections and linear functionals will be strictly real. 

$\Longleftarrow$ %(Real coordinates implies conjugate into $\pgl{3, \R}$)

Normalize so that the eigendirections of $G$ are the standard basis vectors, $g_1 = e_1, g_2 = e_3, g_3 = e_2$, and one of the eigendirections of $H$ is the sum of the standard basis vectors, $h_3 = e_1 + e_2 + e_3$. Denote the remaining eigendirections of $H$ by $h_1 = [b_1, b_2, b_3], h_2 = [b_1', b_2', b_3']$. Then we have flags
\begin{align*}
    A &= \spn{[1, 0, 0]} \subseteq \spn{[1, 0, 0]} \oplus \spn{[0, 1, 0]} \\
    B &= \spn{[b_1, b_2, b_3]} \subseteq \spn{[b_1, b_2, b_3]} \oplus \spn{[b_1', b_2', b_3']} \\
    C &= \spn{[0, 0, 1]} \subseteq \spn{[0, 0, 1]} \oplus \spn{[0, 1, 0]} \\
    D &= \spn{[1, 1, 1]} \subseteq \spn{[1, 1, 1]} \oplus \spn{[b_1', b_2', b_3']}
\end{align*}

For cross ratios, we quotient by $A_1 = [1, 0, 0]$ to obtain
\begin{align*}
    [A_2, A_1 \oplus B_1, A_1 \oplus C_1, A_1 \oplus D_1] / A_1 &= \left[ \infty, \frac{b_2}{b_3}, 0, 1 \right] \\
    &= \frac{(\infty - 1)(0 - \frac{b_2}{b_3})}{(\infty - \frac{b_2}{b_3})(0 - 1)} = \frac{b_2}{b_3}
\end{align*}
and quotient by $C_1 = [0, 0, 1]$ to obtain
\begin{align*}
    [C_1 \oplus A_1, C_1 \oplus B_1, C_2, C_1 \oplus D_1] / A_1 &= \left[ \infty, \frac{b_1}{b_2}, 0, 1 \right] \\
    &= \frac{(\infty - 1)(0 - \frac{b_1}{b_2})}{(\infty - \frac{b_1}{b_2})(0 - 1)} = \frac{b_1}{b_2}
\end{align*}

Thus, these two cross ratios are real if and only if the one dimensional part of $B$ is projectively real and so lies in the projective $\R$ form determined by the eigendirections composing the flags $A, C$ and the one dimensional part of the flag $D$. 

Then using linear functionals 
\begin{align*}
    f_A &= e_3^* \\
    f_B &= [b_2 b_3' - b_3 b_2', b_3 b_1' - b_1 b_3', b_1 b_2' - b_2 b_1']^* \\
    f_C &= - e_1^* \text{ (written equivalently as } e_1^* \text{)}\\
    f_D &= [b_3' - b_2', b_1' - b_3', b_2' - b_1']^*
\end{align*}
defining the complex 2-dimensional parts of the flags and coming from the cross product in the standard manner, consider the triple ratios
\begin{align*}
    r_3(A, B, C) = \frac{f_A (v_B) \; f_B (v_C) \; f_C (v_A)}{f_A (v_C) \; f_B (v_A) \; f_C (v_B)} 
    &= \frac{e_3^* (v_B) \; f_B (e_3) \; e_1^* (e_1)}{e_3^* (e_3) \; f_B (e_1) \; e_1^* (v_B)} \\
    &= \frac{(b_3)(b_1 b_2' - b_2 b_1')(1)}{(1)(b_2 b_3' - b_3 b_2')(b_1)} \\
    &= \frac{b_1 b_2' b_3 - b_1' b_2 b_3}{b_1 b_2 b_3' - b_1 b_2' b_3}
\end{align*}
and
\begin{align*}
    r_3(A, C, D) = \frac{f_a (v_c) \; f_c (v_d) \; f_d (v_a)}{f_a (v_d) \; f_c (v_a) \; f_d (v_c)} 
    &= \frac{e_3^* (e_3) \; e_1^* (v_D) \; f_D (e_1)}{e_3^* (v_D) \; e_1^* (e_1) \; f_D (e_3)} \\
    &= \frac{(1)(1)(b_3' - b_2')}{(1)(1)(b_2' - b_1')} = \frac{b_3' - b_2'}{b_2' - b_1'}
\end{align*}

where $b_1, b_2,$ and $b_3$ are real by the assumption that the cross ratios are real. If $b_3 = 0$ then $B_1 \cap A_2 \neq \set{\vec{0}}$, so the flags are not in generic position. Similarly, if $b_1 = 0$ then $B_1 \cap C_2 \neq \set{\vec{0}}$, so the flags are not in generic position. Hence, assume that $b_1 \neq 0$ and $b_3 \neq 0$. Denoting the first triple ratio $s_1$ and the second $s_2$ then solving for $b_3'$ in the latter gives $b_3' = s_2(b_2' - b_1') + b_2'$.

If $b_2' = 0$ then the triple ratio $s_2$ being real means $b_3'$ is a real multiple of $b_1'$ and so $B_2$ lives in the same projective $\R$ form and the collection is simultaneously conjugate into $\pgl{3, \R}$. Otherwise, assume without loss of generality that $b_2' = 1$. 

In that case, solving for $b_1'$ in the first triple ratio then gives
\begin{align*}
    b_1' (s_1 b_2' b_3 - b_2 b_3) &= s_1 b_1 b_2 b_3' - b_1 b_2' b_3 \\
    b_1' (s_1 b_2' b_3 - b_2 b_3) &= s_1 b_1 b_2 (s_2(b_2' - b_1')) - b_1 b_2' b_3 \\
    b_1' (s_1 b_2' b_3 - b_2 b_3 + s_1 s_2 b_1 b_2) &= s_1 s_2 b_1 b_2 b_2' - b_1 b_2' b_3
\end{align*}
so that by assuming the triple ratios $s_1$ and $s_2$ are real and $b_2' = 1$ we have
\begin{align*}
    b_1' &= \frac{s_1 s_2 b_1 b_2 b_2' - b_1 b_2' b_3}{s_1 b_2' b_3 - b_2 b_3 + s_1 s_2 b_1 b_2} \\
    &= \frac{b_2' (s_1 s_2 b_1 b_2 - b_1 b_3)}{s_1 b_2' b_3 - b_2 b_3 + s_1 s_2 b_1 b_2} \\
    &= \frac{b_2' r_1}{b_2' r_2 + r_3} = \frac{r_1}{r_2 + r_3}
\end{align*}
for real numbers $r_1, r_2, r_3$. Hence, $b_1'$ is real and therefore $b_3'$ is also, meaning $B_2$ is in the same projective $\R$ form and the collection is simultaneously conjugate into $\pgl{3, \R}$. Since every eigendirection lies in the same projective $\R$ form, $G$ and $H$ are simultaneously conjugate into $\pgl{3, \C}$. 

\begin{figure}[H]
\begin{tikzpicture}[scale=0.9]
    %Dashed lines
    \draw[dashed] (-2, 2) -- (2.5, 2.5);
    \draw[dashed] (-1, 2.5) -- (-2, -2.5);
    \draw[dashed] (-2.5, -2) -- (3, -2.5);
    \draw[dashed] (2, 3) -- (2.5, -3);
    %Solid lines
    \draw (0.25, 2.25) -- (2.25, 0) -- (0.25, -2.25) -- (-1.5, 0) -- (0.25, 2.25);
    \draw (0.25, 2.25) -- (0.25, -2.25);
    %Vertices
    \draw[fill=black] (0.25, 2.25) circle (2pt);
    \draw[fill=black] (0.25, -2.25) circle (2pt);
    \draw[fill=black] (2.25, 0) circle (2pt);
    \draw[fill=black] (-1.5, 0) circle (2pt);
    \node[above] at (0.25, 2.25) {$\spn{g_1} = \spn{e_1}$};
    \node[below] at (0.25, -2.25) {$\spn{g_3} = \spn{e_2}$};
    \node[left] at (-1.5, 0) {$\spn{e_1 + e_2 + e_3} = \spn{h_3}$};
    \node[right] at (2.25, 0) {$\spn{h_1} = \spn{ [b_1, b_2, b_3] }$};
    %Dashed line equations
    \node[right] at (2.5, 2.5) {$\spn{e_1} \oplus \spn{e_3} = \spn{g_1} \oplus \spn{g_2}$};
    \node[right] at (3, -2.5) {$\spn{e_2} \oplus \spn{e_3} = \spn{g_3} \oplus \spn{g_2}$};
    \node[right] at (2.5, -3) {$\spn{h_1} \oplus \spn{h_2}$};
    \node[left] at (-2, -2.5) {$\spn{h_3} \oplus \spn{h_2}$};
    %Coordinates
    \draw[fill=black] (0.25, 0.5) circle (2pt);
    \draw[fill=black] (0.25, -0.5) circle (2pt);
    \node[right] at (0.25, 0.5) {$a$};
    \node[right] at (0.25, -0.5) {$b$};
    \draw[fill=black] (-0.5, 0) circle (2pt);
    \node[above] at (-0.5, 0) {$s_1$};
    \draw[fill=black] (1, 0) circle (2pt);
    \node[above] at (1, 0) {$s_2$};
\end{tikzpicture}
\caption{Normalization for two hyperbolic elements in $\pgl{3, \C}$}
\end{figure}
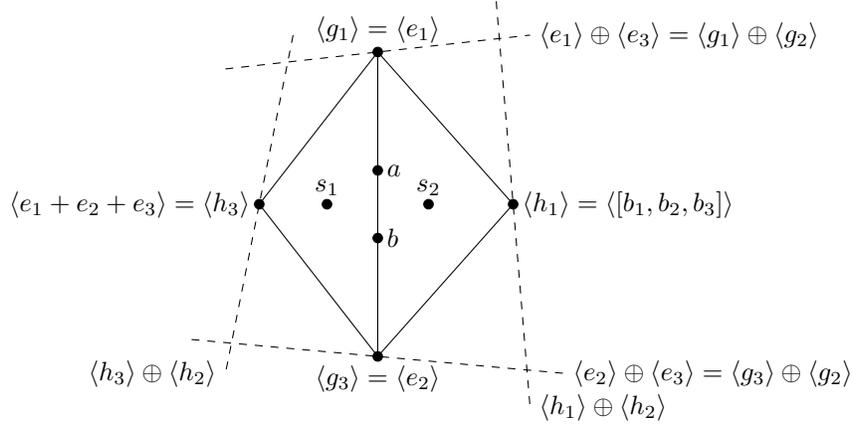
\end{proof}

\begin{restatable}{corollary}{pglthreeHHH} \label{pglthreeHHH}
Any subsequent hyperbolic projective transformation $M$ without repeated eigenvalues and with eigendirection flag pair $\beta, \beta'$, with $A, \beta, C, D$ in generic position and $A, \beta', C, D$ in generic position, is also simultaneously conjugate into $\pgl{3, \R}$ with $G$ and $H$ if and only if
\begin{enumerate}
    \item the four cross ratios, two coming from $[A, \beta, C, D]$ and two coming from $[A, \beta', C, D]$, are real
    \item the two new triple ratios $r_3 (A, \beta, C)$ and $r_3 (A, \beta', C)$ are real
\end{enumerate}
\end{restatable}
\begin{proof}
For other hyperbolic elements in $\pgl{3, \C}$, the calculations are equivalent, with the same assumptions on $A, C,$ and $D$. For each subsequent hyperbolic projective transformation $M$, construct two pictures, with the flag at $B$ being replaced by the two flags $\beta$ and $\beta' $ coming from the eigendirections of $M$ as described above. The four cross ratios will determine the one dimensional parts of $\beta$ and $\beta'$. Real cross ratios give rise to projectively real one dimensional parts of $\beta$ and $\beta'$, or equivalently, two eigendirections of $M$ which live in the projective $\R$ form preserved by $G$ and $H$. 

The triple ratios involving $A, C, D$ have already produced a triple ratio when considering the first two projective transformations, so produce no new information. We take the cross ratios from each picture and the triple ratio on the right side. Following the same calculations as for $G$ and $H$ with the one dimensional parts of $\beta$ and $\beta'$ being projectively real, the triple ratios $r_3 (A, \beta, C)$ and $r_3 (A, \beta', C)$ give us a system of two equations of three unknowns. However, in projective space we can normalize one of these unknowns, and the two equations give the ratio of the remaining unknowns to that which is normalized. Thus, if the collection of cross ratios and triple ratios is real then the remaining eigendirection of $M$ also lives in the same projective $\R$ form. Hence, $M$ is simultaneously conjugate into $\pgl{3, \R}$ if and only if all six of these coordinates (four new cross ratios and two new triple ratios) are real. 

\begin{figure}[H]
\begin{tikzpicture}[scale=.8]
    %Dashed lines
    \draw[dashed] (-2, 2) -- (2.5, 2.5);
    \draw[dashed] (-1, 2.5) -- (-2, -2.5);
    \draw[dashed] (-2.5, -2) -- (3, -2.5);
    \draw[dashed] (2, 3) -- (2.5, -3);
    %Solid lines
    \draw (0.25, 2.25) -- (2.25, 0) -- (0.25, -2.25) -- (-1.5, 0) -- (0.25, 2.25);
    \draw (0.25, 2.25) -- (0.25, -2.25);
    %Vertices
    \draw[fill=black] (0.25, 2.25) circle (2pt);
    \draw[fill=black] (0.25, -2.25) circle (2pt);
    \draw[fill=black] (2.25, 0) circle (2pt);
    \draw[fill=black] (-1.5, 0) circle (2pt);
    \node[above] at (0.25, 2.25) {$\spn{g_1} = \spn{e_1}$};
    \node[below] at (0.25, -2.25) {$\spn{g_3} = \spn{e_2}$};
    \node[left] at (-1.5, 0) {$\spn{[1, 1, 1]} = \spn{h_3}$};
    \node[right] at (2.25, 0) {$\spn{m_1}$};
    %Dashed line equations
    \node[right] at (2.5, 2.5) {$\spn{e_1} \oplus \spn{e_3}$};
    \node[right] at (3, -2.5) {$\spn{e_2} \oplus \spn{e_3}$};
    \node[right] at (2.5, -3.1) {$\spn{m_1} \oplus \spn{m_2}$};
    %Coordinates
    \draw[fill=black] (0.25, 0.5) circle (2pt);
    \draw[fill=black] (0.25, -0.5) circle (2pt);
    \node[right] at (0.25, 0.5) {$a$};
    \node[right] at (0.25, -0.5) {$b$};
    \draw[fill=black] (1, 0) circle (2pt);
    \node[right] at (1, 0) {$s$};
\end{tikzpicture}
\begin{tikzpicture}[scale=0.8]
    %Dashed lines
    \draw[dashed] (-2, 2) -- (2.5, 2.5);
    \draw[dashed] (-1, 2.5) -- (-2, -2.5);
    \draw[dashed] (-2.5, -2) -- (3, -2.5);
    \draw[dashed] (2, 3) -- (2.5, -3);
    %Solid lines
    \draw (0.25, 2.25) -- (2.25, 0) -- (0.25, -2.25) -- (-1.5, 0) -- (0.25, 2.25);
    \draw (0.25, 2.25) -- (0.25, -2.25);
    %Vertices
    \draw[fill=black] (0.25, 2.25) circle (2pt);
    \draw[fill=black] (0.25, -2.25) circle (2pt);
    \draw[fill=black] (2.25, 0) circle (2pt);
    \draw[fill=black] (-1.5, 0) circle (2pt);
    \node[above] at (0.25, 2.25) {$\spn{g_1} = \spn{e_1}$};
    \node[below] at (0.25, -2.25) {$\spn{g_3} = \spn{e_2}$};
    \node[left] at (-1.5, 0) {$\spn{[1, 1, 1]} = \spn{h_3}$};
    \node[right] at (2.25, 0) {$\spn{m_3}$};
    %Dashed line equations
    \node[right] at (2.5, 2.5) {$\spn{e_1} \oplus \spn{e_3}$};
    \node[right] at (3, -2.5) {$\spn{e_2} \oplus \spn{e_3}$};
    \node[right] at (2.5, -3.1) {$\spn{m_3} \oplus \spn{m_2}$};
    %Coordinates
    \draw[fill=black] (0.25, 0.5) circle (2pt);
    \draw[fill=black] (0.25, -0.5) circle (2pt);
    \node[right] at (0.25, 0.5) {$a'$};
    \node[right] at (0.25, -0.5) {$b'$};
    \draw[fill=black] (1, 0) circle (2pt);
    \node[right] at (1, 0) {$s'$};
\end{tikzpicture}
\caption{Further hyperbolic elements conjugate into $\pgl{3, \R}$}
\end{figure}
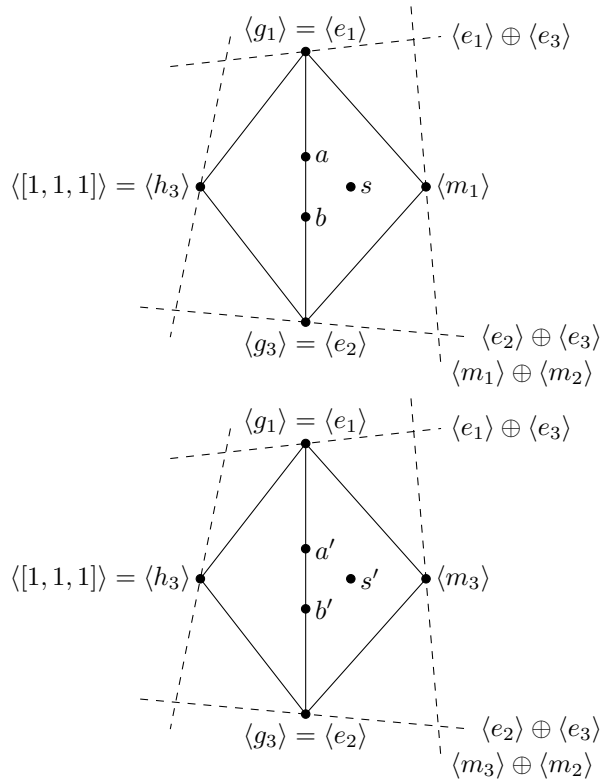
\end{proof}

\subsection{\texorpdfstring{$m \geq 2$}{m >= 2} Hyperbolic and \texorpdfstring{$n \geq 1$}{n >= 1} Elliptic}

We now consider adding some elliptic transformations to a set of hyperbolic transformations. In the $\pgl{3, \C}$ case we refer to projective transformations with two elliptic eigenvalues and one hyperbolic eigenvalue as elliptic, as it is not possible to have strictly elliptic elements in $\pgl{3, \C}$. Though a similar obstruction always arises in odd dimensions, we do not refer to elements with a hyperbolic eigendirection as elliptic in any higher dimension; this is only done in the $\pgl{3, \C}$ case. 

Again choose two hyperbolic projective transformations $G$ and $H$ and normalize so that the eigendirections of $G$ are $e_1, e_3$, and $e_2$ respectively, with an eigendirection of $H$ being $e_1 + e_2 + e_3$. Denote the remaining eigendirections of $H$ by $h_1 = [b_1, b_2, 1], h_2 = [b_1', b_2', 1]$. Then we have the four flags
\begin{align*}
    A &= [1, 0, 0] \subseteq [1, 0, 0] \oplus [0, 0, 1] \\
    B &= [b_1, b_2, 1] \subseteq [b_1, b_2, 1] \oplus [b_1', b_2', 1] \\
    C &= [0, 1, 0] \subseteq [0, 1, 0] \oplus [0, 0, 1] \\
    D &= [1, 1, 1] \subseteq [1, 1, 1] \oplus [b_1', b_2', 1]
\end{align*}
As in the $\pgl{2, \C}$ case, we first check all of the hyperbolic transformations using these first two. We do this using the flags coming from the eigendirections as described in the all hyperbolic case above. 

For the elliptics, since the eigendirections of a real, elliptic projective transformation must be componentwise complex conjugates, this means that the cross ratios with quotients as above must be conjugate. For an elliptic transformation $M$ with eigendirections $m_1, m_2, m_3$, where $m_1$ and $m_3$ correspond to eigenvalues whose ratio is in $S^1$, make flags
    $$\beta = m_1 \subseteq m_1 \oplus m_2 \subseteq m_1 \oplus m_2 \oplus m_3$$
and
    $$\beta = m_3 \subseteq m_3 \oplus m_2 \subseteq m_3 \oplus m_2 \oplus m_1$$

\begin{restatable}{lemma}{pglthreeHHE} \label{pglthreeHHE}
If $G$ and $H$ are two hyperbolic transformations without repeat eigenvalues which are simultaneously conjugate into $\pgl{3, \R}$ and have eigendirection flag pairs $A, C$ and $B, D$, then an elliptic transformation $M$ without repeat eigenvalues with flag pairs $\beta, \beta'$, where $A, \beta, C, D$ are in generic position and $A, \beta', C, D$ are in generic position, is also simultaneously conjugate into $\pgl{3, \R}$ if only if
\begin{enumerate}
    \item $[A, \beta, C, D] = \conj{[A, \beta', C, D]}$ after both quotients
    \item $r_3 (A, \beta, C) = \conj{r_3 (A, \beta', C)}$
\end{enumerate}
\end{restatable} 
\begin{proof}
The cross ratios being complex conjugate if and only if $m_1$ and $m_3$ are compatible with $M$ being also simultaneously conjugate into $\pgl{3, \R}$ follows directly from the computations of cross ratio above, as the cross ratios are in complex conjugate pairs if and only if the components of $m_1$ and $m_3$ are complex conjugates. 

The triple ratios of elliptics is more interesting. Note that the cross ratios give us more information, so we can essentially normalize more than we would be able to otherwise. Normalize so that the eigendirections of the hyperbolic element $G$ are projectively real and denote the line and plane parts of the associated flags by the direction vectors and linear functionals 
\begin{align*}
    v_A &= \spn{[r_1, s_1, t_2]} & f_A &= [x_1, y_1, z_1]^* \\
    v_C &= \spn{[r_2, s_2, t_2]} & f_C &= [x_2, y_2, z_2]^* 
\end{align*}

Normalize so that the elliptic projective transformation $M$ has eigendirections $m_1 = i \cdot e_1 + e_2, m_3 = -i \cdot e_1 + e_2$. Denote the remaining eigendirection of $M$ by $[a, b, 1]$, where the third component is necessarily nonzero to be a valid triple of eigendirections for $M$. Then we have flags $\beta$ and $\beta'$ with
\begin{align*}
    v_\beta &= \spn{[i, 1, 0]} & f_\beta &= [1, -i, -a + bi]^* \\
    v_{\beta'} &= \spn{[-i, 1, 0]} & f_{\beta'} &= [1, i, -a - bi]^*
\end{align*}

\begin{figure}[H]
\begin{tikzpicture}[scale=0.75]
    %Dashed lines
    \draw[dashed] (-2, 2) -- (2.5, 2.5);
    \draw[dashed] (-1, 2.5) -- (-2, -2.5);
    \draw[dashed] (-2.5, -2) -- (3, -2.5);
    \draw[dashed] (2, 3) -- (2.5, -3);
    %Solid lines
    \draw (0.25, 2.25) -- (2.25, 0) -- (0.25, -2.25) -- (-1.5, 0) -- (0.25, 2.25);
    \draw (0.25, 2.25) -- (0.25, -2.25);
    %Vertices
    \draw[fill=black] (0.25, 2.25) circle (2pt);
    \draw[fill=black] (0.25, -2.25) circle (2pt);
    \draw[fill=black] (2.25, 0) circle (2pt);
    \node[above] at (0.25, 2.25) {$A_1$};
    \node[below] at (0.25, -2.25) {$C_1$};
    \node[right] at (2.25, 0) {$\color{red} \spn{m_1} = \spn{i \cdot e_1 + e_2}$};
    %Dashed line equations
    \node[right] at (2.5, 2.5) {$A_2$};
    \node[right] at (3, -2.5) {$C_2$};
    \node[below right] at (2.5, -3) {$\spn{m_1} \oplus \spn{m_2}$};
    %Coordinates
    \draw[fill=black] (1, 0) circle (2pt);
    \node[right] at (1, 0) {$s$};
\end{tikzpicture}
\begin{tikzpicture}[scale=0.75]
    %Dashed lines
    \draw[dashed] (-2, 2) -- (2.5, 2.5);
    \draw[dashed] (-1, 2.5) -- (-2, -2.5);
    \draw[dashed] (-2.5, -2) -- (3, -2.5);
    \draw[dashed] (2, 3) -- (2.5, -3);
    %Solid lines
    \draw (0.25, 2.25) -- (2.25, 0) -- (0.25, -2.25) -- (-1.5, 0) -- (0.25, 2.25);
    \draw (0.25, 2.25) -- (0.25, -2.25);
    %Vertices
    \draw[fill=black] (0.25, 2.25) circle (2pt);
    \draw[fill=black] (0.25, -2.25) circle (2pt);
    \draw[fill=black] (2.25, 0) circle (2pt);
    \node[above] at (0.25, 2.25) {$A_1$};
    \node[below] at (0.25, -2.25) {$C_1$};
    \node[right] at (2.25, 0) {$\color{red} \spn{m_3} = \spn{-i \cdot e_1 + e_2}$};
    %Dashed line equations
    \node[right] at (2.5, 2.5) {$A_2$};
    \node[right] at (3, -2.5) {$C_2$};
    \node[below right] at (2.5, -3) {$\spn{m_3} \oplus \spn{m_2}$};
    %Coordinates
    \draw[fill=black] (1, 0) circle (2pt);
    \node[right] at (1, 0) {$s$};
\end{tikzpicture}
\caption{$\pgl{3, \R}$ triple ratios for an elliptic with $n \geq 2$ hyperbolics}
\end{figure}
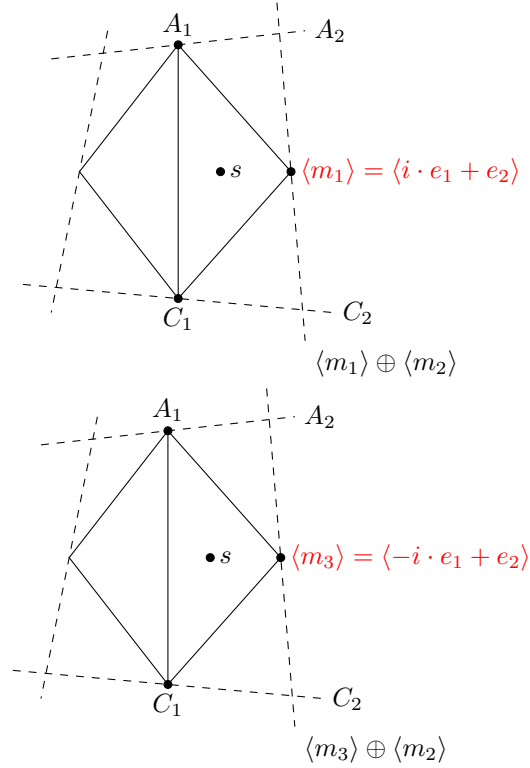

Then the triple ratios are
\begin{align*}
    r_3(A, \beta, C) &= \frac{f_A (v_\beta) \; f_\beta (v_C) \; f_C (v_A)}{f_A (v_C) \; f_\beta (v_A) \; f_C (v_\beta)} \\
    &= \frac{(x_1 i + y_1) \; f_\beta (e_2) \; (r_1 x_2 + s_1 y_2 + t_1 z_2)}{(x_1 r_2 + y_1 s_2 + z_1 t_2) \; f_\beta (e_1) \; (x_2 i + y_2)} \\
    &= \frac{(x_1 i + y_1) \; (r_1 x_2 + s_1 y_2 + t_1 z_2)}{(x_1 r_2 + y_1 s_2 + z_1 t_2) \; (x_2 i + y_2)} \cdot \frac{(r_2 - i s_2 + t_2 (-a + bi))}{(r_1 - i s_1 + t_1 (-a + b i))}
\end{align*}
and
\begin{align*}
    r_3(A, \beta', C) &= \frac{f_A (v_{\beta'}) \; f_{\beta'} (v_C) \; f_C (v_A)}{f_A (v_C) \; f_{\beta'} (v_A) \; f_C (v_{\beta'})} \\
    &= \frac{(-x_1 i + y_1) \; f_{\beta'} (e_2) \; (r_1 x_2 + s_1 y_2 + t_1 z_2)}{(x_1 r_2 + y_1 s_2 + z_1 t_2) \; f_{\beta'} (e_1) \; (-x_2 i + y_2)} \\
    &= \frac{(-x_1 i + y_1) \; (r_1 x_2 + s_1 y_2 + t_1 z_2)}{(x_1 r_2 + y_1 s_2 + z_1 t_2) \; (-x_2 i + y_2)} \cdot \frac{(r_2 + i s_2 + t_2 (-a - bi))}{(r_1 + i s_1 + t_1 (-a - b i))}
\end{align*}

If these triple ratios are complex conjugates of each other, then since all of $r_i, s_i, t_i$ and $x_i, y_i, z_i$ are all real, we have
\begin{align*}
    \frac{(r_2 - i s_2 + t_2 (-a + bi))}{(r_1 - i s_1 + t_1 (-a + b i))} 
    &= \conj{\frac{(r_2 + i s_2 + t_2 (-a - bi))}{(r_1 + i s_1 + t_1 (-a - b i))}} \\
    &= \frac{(r_2 - i s_2 + t_2 (-\conj{a} + \conj{b} i))}{(r_1 - i s_1 + t_1 (-\conj{a} + \conj{b} i))} 
\end{align*}
which simplifies to $(a - \conj{a}) + (\conj{b} - b)i = 0$ when the flags are in generic position. Since $a - \conj{a} = 2i \; \im(a)$ and $\conj{b} - b = -2i \; \im(b)$, and in particular $a - \conj{a}$ is strictly imaginary and $(\conj{b} - b) i$ is strictly real, we have $a, b \in \R$. Hence, the last eigendirection of $M$ is projectively real, and so $M$ is simultaneously conjugate into $\pgl{3, \R}$. 

\end{proof}

We summarize this section by combining the previous lemmas into a single theorem. 

\begin{restatable}{theorem}{pglthreeOne}
If $\set{H_1, \cdots, H_m, E_1, \cdots, E_n}$ is a collection of $m \geq 2$ hyperbolic and $n \geq 0$ elliptic transformations in $\pgl{3, \C}$, each without repeat eigenvalues, where $H_1$ and $H_2$ have eigendirection flag pairs $A, C$ and $B, D$ such that each remaining element $M$ in the collection has an eigendirection flag pair $\beta, \beta'$ with $A, \beta, C, D$ in generic position and $A, \beta', C, D$ in generic position, then the collection is simultaneously conjugate into $\pgl{3, \R}$ if and only if
\end{restatable}
\begin{enumerate}
    \item the two cross ratios coming from $[A, B, C, D]$ are real
    \item the two triple ratios $r_3(A, B, C)$ and $r_3(A, C, D)$ are real
\end{enumerate}
for each subsequent hyperbolic element $M$ with eigendirection flag pair $\beta, \beta'$
\begin{enumerate}[resume]
    \item the two cross ratios coming from $[A, \beta, C, D]$ are real
    \item the two cross ratios coming from $[A, \beta', C, D]$ are real
    \item the two triple ratios $r_3(A, \beta, C)$ and $r_3(A, \beta', C)$ are real
\end{enumerate}
and for each elliptic element $M$ with eigendirection flag pair $\beta, \beta'$
\begin{enumerate}[resume]
    \item the pair of cross ratios coming from $[A, \beta, C, D]$ are complex conjugates of the pair of cross ratios coming from $[A, \beta', C, D]$
    \item $r_3(A, \beta, C) = \conj{r_3(A, \beta', C)}$
\end{enumerate}

\subsection{\texorpdfstring{$n$}{n} Elliptic Transformations}

We would like to use the same method as with the hyperbolics, since the quotients and resulting cross ratios are convenient. Unfortunately, though $k+1$ hyperbolic eigendirections determines a projective $\R$ form which might be commonly preserved, the same is not true for elliptic eigendirections. This can be seen by considering the $\pgl{2, \C}$ case and placing an eigendirection pair at $0, \infty$ and a third eigendirection at $1$. If these were hyperbolic eigendirections, the corresponding hyperbolic transformations could only commonly preserve the projective $\R$ form coming from $\R^2 \subseteq \C^2$. On the other hand, if these are elliptic eigendirections, then any circle centered at the origin could be commonly preserved. Fortunately, we can pretend that we have two such hyperbolics and still achieve rigorous results. 

Choose two elliptic transformations $A$ and $B$. Normalize so that the elliptic eigendirections of $A$ are $[i, 1, 0]$ and $[-i, 1, 0]$, the hyperbolic eigendirection of $A$ is $[0, 0, 1]$, and the hyperbolic eigendirection of $B$ is $[1, 1, 1$]. Because of genericity of flags, we can make this normalization, but in higher dimensions we will need to be more clever. 

In particular, the collection $\set{[i, 1, 0], [-i, 1, 0], [0, 0, 1], [1, 1, 1]}$ forms a projective frame, so there is a unique corresponding conjugation action and therefore a unique projective $\R$ form preserved. Since inversion about the standard projective $\R$ form acts appropriately on the projective frame, it is therefore the unique projective $\R$ form which is preserved. 

We can thus take computations as if we have two hyperbolic transformations: $G$ with eigendirections $[1, 0, 0], [0, 1, 0], [0, 0, 1]$ and $H$ with an eigendirection $[1, 1, 1]$. Then the cross ratios and triple ratios are precisely as in the case where there are $m \geq 2$ hyperbolic transformation. Note that we do not need to take any coordinates with $G$ and that with $H$ we need only two cross ratios and two triple ratios to determine the remaining eigendirections (the same number of coordinates as when we started with two hyperbolic transformations). We will explain in more detail in the next section how to do these computations. 

\section{Fock-Goncharov Coordinates for Conjugation into \texorpdfstring{$\pgl{k, \C}$}{PGL(k, C)}}

In $\C P^2$, we were able to quotient by a codimension $1$ subspace to produce a copy of $\C P^1$, which we could then take cross ratios in. We will do the same in higher dimensions and, similarly, we are able to quotient by a codimension $2$ subspace in $\C P^{k-1}$ to produce a copy of $\C P^2$ which we can take triple ratios in. 

Our goal is to find a projective $\R$ form that is compatible with (and possibly determined by) $k+1$ of the eigendirections from a collection of projective transformations. We will first describe how to find that projective $\R$ form and then describe how to produce coordinates from the collection of projective transformations. We note, however, that finding flags in generic position in higher dimensions becomes increasing less likely. For flags $A, B, C$ in $\C^2$ to be in generic position, we simply require $A_1 \cap B_1 = \set{\vec{0}}$, $A_1 \cap C_1 = \set{\vec{0}}$, and $B_1 \cap C_1 = \set{\vec{0}}$. In higher dimensions, there become more possible obstructions to being in generic position because there are more components of flags that can intersect non trivially. Said another way, starting with two flags $A, B$ in $\C^2$, there are only two potential flags which are not in generic position with both $A$ and $B$, namely $A$ and $B$ themselves. In $\C^3$, given flags $A, B$ in generic position, there is a $\C$ worth of flags $C$ in $\C^2$ which are not in generic position with $A, B$. As a result, most collections of projective transformations one might encounter would have flags which are not in generic position. In the $\pgl{k, \C}$ case, we will outline some results similar to the lower dimensions utilizing flags of eigendirections in generic position, but then will present an alternative style of genericity which allows further use of cross ratios. We note that this alternative genericity condition we will use does not imply the standard generic position condition nor vice versa.  

For each computation, we will begin by normalizing so that we have flags
\begin{align*}
    A &= \spn{e_1} \subseteq \spn{e_1} \oplus \spn{e_2} \subseteq \cdots \subseteq \spn{e_1} \oplus \spn{e_2} \oplus \cdots \oplus \spn{e_k} \\
    C &= \spn{e_k} \subseteq \spn{e_k} \oplus \spn{e_{k-1}} \subseteq \cdots \subseteq \spn{e_k} \oplus \spn{e_{k-1}} \oplus \cdots \oplus \spn{e_1}
\end{align*}
and will only use the one dimensional part of the flag $D$, so normalize so that
    $$D_1 = \spn{e_1 + e_2 + \cdots + e_k}$$
These are based on normalizations of the eigenbasis of a transformation and one other eigendirection in generic position, regardless of if the eigendirection are elliptic or hyperbolic. If there exists any transformation with eigendirection flag pair $F, F'$ and any other eigendirection from the collection which can be used as the direction vector for some flag $G_1$ such that $F, F', G_1$ are in generic position, then we describe in the next section how to use those vectors to normalize such that we have the desired flags $A, C, D_1$ above. 

If such flags do not exist, then as in section 3.2, a basis associated to a projective $\R$ form which is preserved by the collection of transformations is some set $\set{v_1, \cdots, v_k}$ such that for every hyperbolic eigendirection $p$ we have
    $$p = \sum\limits_{j=1}^k \lambda_j v_j \text{ for } \lambda_j \in \R \text{ not all 0}$$
and for elliptic eigendirection pairs $p, p'$ we have
    $$p = \sum\limits_{j=1}^k \lambda_j v_j \text{ for } \lambda_j \in \C \text{ not all 0}$$
with
    $$\conj{p} = \sum\limits_{i=1}^k \conj{\lambda_j v_j} = \sum\limits_{i=1}^k \conj{\lambda_j} v_j' = p'$$
where
    $$v_j' = 
    \begin{cases}
        v_j & \text{if hyperbolic} \\
        \text{its eigendirection pair} & \text{if elliptic}
    \end{cases}$$
In that case, we can perform the linear algebra to again find flags $A, C, D_1$ as above, now created from nonexistent eigendirections which describe the candidate preserved projective $\R$ form. 

\subsection{Computing Cross Ratios}
To take cross ratios of four flags $A, \beta, C, D$ in generic position, we project the flags to $\C^k / \left( A_i \oplus C_j \right)$ for $i + j = k-2$ with $i, j \geq 0$ to get points in $\C P^1$. Because the flags are in generic position, we can modify them as
\begin{align*}
    A &\rightarrow A_{i+1} \oplus C_j \\
    \beta &\rightarrow \left( A_i \oplus C_j \right) \oplus \beta_1 \\
    C &\rightarrow A_i \oplus C_{j+1} \\
    D &\rightarrow \left( A_i \oplus C_j \right) \oplus D_1
\end{align*}
to get copies of $\C^k$ and then quotient by $A_i \oplus C_j$. As above, normalize so that
\begin{align*}
    A &= \spn{e_1} \subseteq \spn{e_1} \oplus \spn{e_2} \subseteq \cdots \subseteq \spn{e_1} \oplus \spn{e_2} \oplus \cdots \oplus \spn{e_k} \\
    C &= \spn{e_k} \subseteq \spn{e_k} \oplus \spn{e_{k-1}} \subseteq \cdots \subseteq \spn{e_k} \oplus \spn{e_{k-1}} \oplus \cdots \oplus \spn{e_1} \\
    D_1 &= \spn{e_1 + e_2 + \cdots + e_k}
\end{align*}
When we quotient $\beta$ by $A_i \oplus C_j$, we are left with the $i+1$ and $i+2$ components of the vector portion of $\beta$. In other words, if the vector portion of the flag $\beta$ is the eigendirection
    $$m_1 = [a_1, a_2, \cdots, a_k]$$
then when we quotient by $A_i \oplus C_j$ we have
    $$m_1 \mapsto [a_{i+1}, a_{i+2}]$$
Similarly, the flags $A, C,$ and $D$ will become
\begin{align*}
    A &\mapsto [1, 0] \\
    C &\mapsto [0, 1] \\
    D &\mapsto [1, 1]
\end{align*}

The cross ratios will therefore be
    $$\left[ \infty, \frac{a_{i+1}}{a_{i+2}}, 0, 1 \right] = \frac{(\infty - 1)(0 - \frac{a_{i+1}}{a_{i+2}})}{(\infty - \frac{a_{i+1}}{a_{i+2}})(0 - 1)} = \frac{a_{i+1}}{a_{i+2}}$$

This means that each of the $k-1$ cross ratios gives the ratio of two components of a single eigendirection. If the flags are in generic position, then the collection of $k-1$ cross ratios will fully determine (projectively) the one dimensional part of the flag $\beta$. Then we have the following lemmas.

\begin{restatable}{lemma}{pglKHH}
Let $G$ and $H$ be hyperbolic transformations in $\pgl{k, \C}$ without repeated eigenvalues, with flag pairs $A, C$ and $B, D$ respectively, and where $A, B, C, D$ are in generic position. Then $B_1$ and $D_1$ are compatible with a projective $\R$ form preserved by $G$ if and only if $[A, B, C, D] \in \R$ for each of the $k-1$ projections. 
\end{restatable}
\begin{proof}
By genericity of the flags, the eigendirections of $G$ together with $D_1$ determine a unique projective $\R$ form which might be preserved. If we normalize so that the eigenbasis of $G$ is the standard basis and $D_1$ is the sum of vectors in the standard basis, this projective $\R$ form is that coming from $\R^k \subseteq \C^k$. Thus, $B_1$ is also in that projective $\R$ form if and only if the ratios of every pair of components is real, which occurs precisely when each of the cross ratios is real. Note that none of the coordinates may be 0, because the flags are assumed to be in generic position. 

\end{proof}

\begin{restatable}{lemma}{pglKHHM}
Given two hyperbolic transformations without repeated eigenvalues which are simultaneously conjugate into $\pgl{k, \R}$ and have flag pairs $A, C$ and $B, D$, then for a transformation $M$ with flags $\beta, \beta'$ coming from eigendirections $m_1, \cdots, m_k$ such that $A, \beta, C, D$ are in generic position and $A, \beta', C, D$ are in generic position:
\begin{enumerate}
    \item $m_1$ and $m_k$ are an elliptic eigendirection pair compatible with the projective $\R$ form preserved by $G$ and $H$ if and only if $[A, \beta, C, D] = \conj{[A, \beta', C, D]}$ for each of the $k-1$ projections
    \item $m_1$ is a hyperbolic eigendirection compatible with the projective $\R$ form preserved by $G$ and $H$ if and only if the cross ratios $[A, \beta, C, D]$ are real for each projection. Similarly, $m_k$ is a hyperbolic eigendirection compatible with the projective $\R$ form preserved by $G$ and $H$ if and only if the cross ratios coming from $[A, \beta', C, D]$ are real. 
\end{enumerate}
\end{restatable}
\begin{proof}
Normalizing so that $A, C$ is the flag pair coming from the standard basis and $D_1$ is the span of the sum of the standard basis vectors, the elliptic eigendirection condition follows directly from the ratio of successive terms in componentwise complex conjugate eigendirections being complex conjugate. Let $[A, \beta, C, D] = \frac{a_{i+1}}{a_{i+2}}$ when quotienting by $A_i \oplus C_j$ and $[A, \beta', C, D] = \frac{b_{i+1}}{b_{i+2}}$. Normalize each eigendirection so that its last component, which is nonzero by genericity, is 1. Then the sequence of cross ratios being complex conjugate implies the sequence of products of cross ratios are complex conjugate. This sequence of products of cross ratios reads successive terms, so since the eigendirections are compatible with a real projective transformation if and only if there exists representatives such that their components are complex conjugate, we have the condition $[A, \beta', C, D] = \frac{\conj{a_{i+1}}}{\conj{a_{i+2}}}$.

The hyperbolic eigendirection condition is equivalent to the ratio of every pair of successive components of the eigendirection being real. As the flags are assumed to be in generic position, none of the components may be 0, so this means that the ratio of every pair of components is real which occurs if and only if the eigendirection is projectively real and thus is compatible with being conjugate into $\pgl{k, \R}$. 

\end{proof}

\begin{restatable}{lemma}{pglKEE}
Let $G$ and $H$ be elliptic transformations in $\pgl{k, \C}$ with eigendirection flag pairs $A, C$ and $B, D$, where $A, B, C, D$ are in generic position and $A$ (similarly $C$) is constructed in a non-standard fashion as $\spn{v_1} \subseteq \spn{v_1} \oplus \spn{v_2} \subseteq \cdots$ with each $v_{2n+1}$ and $v_{2n+2}$ being an elliptic eigendirection pair. Let $B_1$ and $D_1$ be an elliptic eigendirection pair, as is our standard for flags. Then $B_1$ and $D_1$ are compatible with a projective $\R$ form preserved by $G$ if and only if 
    $$\frac{[A^{2n+1,k-3-2n}, B^{2n+1,k-3-2n}, C^{2n+1,k-3-2n}, D^{2n+1,k-3-2n}]}
    {[A^{2n+2,k-4-2n}, B^{2n+2,k-4-2n}, C^{2n+2,k-4-2n}, D^{2n+2,k-4-2n}]} \in \R^+$$
for every $0 \leq 2n < k$.
\end{restatable}
Recall that given flags $A, B, C, D$, we let
    \begin{align*}
        A^{i,j} &= \left(A_{i+1} \oplus C_j\right) / (A_i \oplus C_j) \\
        B^{i,j} &= \left(B_{1} \oplus (A_i \oplus C_j)\right) / (A_i \oplus C_j) \\
        C^{i,j} &= \left(C_{j+1} \oplus A_i\right) / (A_i \oplus C_j) \\
        D^{i,j} &= \left(D_1 \oplus (A_i \oplus C_j)\right) / (A_i \oplus C_j)
    \end{align*}
and let $[A, B, C, D]$ denote the set of $k-1$ cross ratios $[A^{i,j}, B^{i,j}, C^{i,j}, D^{i,j}]$ for each $i+j = k-2$ with $i, j \geq 0$. 
\begin{proof}
Normalize so that the eigendirections of $G$ are $e_1, e_2, \cdots, e_k$, where (unlike our typical flag construction) $e_{2n+1}$ and $e_{2n+2}$ are an elliptic pair of eigendirections. Further, normalize so the vector portion of the flag $D$ is $e_1 + e_2 + \cdots + e_k$. Denote the vector portion of the flag $B$ by $[a_1, a_2, \cdots, a_k]$, which is (like our typical normalization) the elliptic pair eigendirection of $D_1$.

By being in generic position, the eigendirections of $G$ and the eigendirection corresponding to $D_1$ together form a projective frame. With our normalization, this projective frame is $\set{e_1, e_2, \cdots, e_k, e_1 + e_2 + \cdots + e_k}$ with associated basis $\set{e_1, e_2, \cdots, e_k}$. Similarly, the eigendirections of $G$ and the eigendirection of corresponding to $B_1$ together form a projective frame, which in our normalization is $\set{e_1, e_2, \cdots, e_k, [a_1, a_2, \cdots, a_k]}$ with associated basis $\set{a_1 e_1, a_2 e_2, \cdots a_k e_k}$. 

Points in $\C P^{k-1}$ are then represented as 
    $$\set{p = \sum\limits_{i=1}^k \lambda_i v_i \mid \lambda_i \in \C \text{ not all 0}}$$
with $v_i$ being the basis vectors in $\set{e_1, e_2, \cdots, e_k}$. A preserved projective $\R$ form is the subset such that $p = \conj{p}$ where $\conj{p} = \sum\limits_{i=1}^k \conj{\lambda_i} v_i'$ with $v_i'$ being the images of the original basis vectors $v_i$ in an associated basis, that is, basis vectors in $\set{a_1 e_1, a_2 e_2, \cdots a_k e_k}$. Then $p = \conj{p}$ implies
    $$[\lambda_1, \lambda_2, \cdots, \lambda_{k-1}, \lambda_k] = [a_1 \conj{\lambda_2}, a_2 \conj{\lambda_1}, \cdots, a_{k-1} \conj{\lambda_{k}}, a_k \conj{\lambda_{k-1}}]$$
Hence, for each $2n$ if $p = \conj{p}$ we have
    $$\frac{\lambda_{2n+1}}{\lambda_k} = \frac{a_{2n+1} \conj{\lambda_{2n+2}}}{a_k \conj{\lambda_{k-1}}}$$
and
    $$\frac{\lambda_{2n+2}}{\lambda_k} = \frac{a_{2n+2} \conj{\lambda_{2n+1}}}{a_k \conj{\lambda_{k-1}}}$$
Dividing the first equation by the second, we have that this implies
    $$\frac{\lambda_{2n+1}}{\lambda_{2n+2}} = \frac{a_{2n+1}}{a_{2n+2}} \cdot \frac{\conj{\lambda_{2n+2}}}{\conj{\lambda_{2n+1}}}$$
and thus
    $$\left( \frac{\lambda_{2n+1}}{\lambda_{2n+2}} \right) \conj{\left( \frac{\lambda_{2n+1}}{\lambda_{2n+2}} \right)} = \frac{a_{2n+1}}{a_{2n+2}}$$
so that $\frac{a_{2n+1}}{a_{2n+2}} \in \R^+$. Hence, the space of preserved points is non-empty and there is a preserved projective $\R$ form if and only if $\frac{a_{2n+1}}{a_{2n+2}} \in \R^+$ for every $n$. The result then follows from our construction of the quotients with this normalization. 

\end{proof}

Note that the previous result is analogous to the same result for elliptic eigendirections in $\C P^1$, where we saw normalized a pair to be at $\infty$ and $0$ and saw that another pair simultaneously preserved a projective $\R$ form if and only if they lived on a single ray from the origin. 

The previous lemmas allow us to normalize such that only the projective $\R$ form coming from $\R^k \subseteq \C^k$ is preserved and create a hyperbolic projective transformation with the standard basis vectors as eigendirections and another hyperbolic transformation which has the sum of the standard basis vectors as one eigendirection. We can then do computations for the remaining transformations, and the remaining eigendirections of the elliptic transformation, using the constructed hyperbolic transformations. We do not perform triple ratios using flags consisting of elliptic eigendirections, because the asymmetric quotients would not result in enlightening coordinates. 

\subsection{Computing Triple Ratios}
\subsubsection{Fock-Goncharov Method for Producing Copies of \texorpdfstring{$\C^3$}{C3}}
Using the Fock-Goncharov method to take triple ratios of the triple of flags $A, \beta, C$ in generic position, we quotient $\C^k$ by $A_i \oplus C_j \oplus \beta_\ell$ for $i + j + \ell = k-3$ with $i, j, \ell \geq 0$ to get points in $\C P^2$. We can't directly quotient the flags at any particular dimension, since by genericity they are not contained in a 3 complex-dimensional subspace of $\C^k$. We use the flags to produce:
\begin{align*}
    A &\rightarrow A_{i+1} \oplus C_j \oplus \beta_\ell \subseteq A_{i+2} \oplus C_j \oplus \beta_\ell \\
    \beta &\rightarrow  A_i \oplus C_j \oplus \beta_{\ell+1} \subseteq A_i \oplus C_j \oplus \beta_{\ell+2} \\
    C &\rightarrow A_i \oplus C_{j+1} \oplus \beta_\ell \subseteq A_i \oplus C_{j+2} \oplus \beta_\ell
\end{align*}
Let 
    $$\beta = \spn{m_1} \subseteq \spn{m_1} \oplus \spn{m_2} \subseteq \cdots \subseteq \spn{m_1} \oplus \spn{m_2} \oplus \cdots \oplus m_k$$ 
and let 
    $$m_i = [m_{(i)(1)}, m_{(i)(2)}, \cdots, m_{(i)(k)}]$$

For each $i + j + \ell = k - 3$, when we normalize so that 
\begin{align*}
    A_{i+1} &= \spn{e_1} \subseteq \spn{e_1} \oplus \spn{e_2} \subseteq \cdots \subseteq \spn{e_1} \oplus \cdots \oplus \spn{e_{i+1}} \\
    B_{\ell} &= \spn{e_{i+3}} \subseteq \spn{e_{i+3}} \oplus \spn{e_{i+4}} \subseteq \cdots \subseteq \spn{e_{i + 3}} \oplus \spn{e_{i + 4}} \oplus \cdots \oplus \spn{e_{i + \ell + 2}} \\ 
    C_{j+1} &= \spn{e_k} \subseteq \spn{e_k} \oplus \spn{e_{k-1}} \subseteq \cdots \subseteq \spn{e_k} \oplus \cdots \oplus \spn{e_{k-j}} \\
    &= \spn{e_k} \subseteq \spn{e_k} \oplus \spn{e_{k-1}} \subseteq \cdots \subseteq \spn{e_k} \oplus \cdots \oplus \spn{e_{i+\ell + 3}}
\end{align*}
and the intersection $A_{i+2} \cap C_{j+2} = \spn{e_{i+2}}$, which we can do by genericity of flags $A, B, C$. If we quotient by $A_i \oplus C_j \oplus \beta_\ell$ we have
\begin{align*}
    A &\mapsto \spn{[1, 0, 0]} \subseteq \spn{[1, 0, 0]} \oplus \spn{[0, 1, 0]} \\
    C &\mapsto \spn{[0, 0, 1]} \subseteq \spn{[0, 0, 1]} \oplus \spn{[0, 1, 0]}
\end{align*}
in the basis $e_{i+1}, e_{i + 2}, e_{i + \ell + 3}$ for $\C^3$. Further, we have

\begin{align*}
    \beta &\mapsto 
    (\beta_{\ell + 1} \oplus A_i \oplus C_j) / (A_i \oplus C_j \oplus B_{\ell}) \subseteq (\beta_{\ell + 2} \oplus A_i \oplus C_j) / (A_i \oplus C_j \oplus B_{\ell}) \\
    %\left( \left( \beta_{\ell + 1} \subseteq \beta_{\ell + 1} \oplus \beta_{\ell + 2} \right) \oplus \left( A_i \oplus C_j \oplus \beta_{\ell} \right) \right)/ \left( A_i \oplus C_j \oplus \beta_{\ell} \right) \\
    &= \spn{[m_{(\ell + 1)(i + 1)}, m_{(\ell + 1)(i + 2)}, m_{(\ell + 1)(i + \ell + 3)}]} \subseteq \\
    & \hspace{1cm} \spn{[m_{(\ell + 1)(i + 1)}, m_{(\ell + 1)(i + 2)}, m_{(\ell + 1)(i + \ell + 3)}]} \oplus \\
    & \hspace{1cm} \spn{[m_{(\ell + 2)(i + 1)}, m_{(\ell + 2)(i + 2)}, m_{(\ell + 2)(i + \ell + 3)}]}
\end{align*}

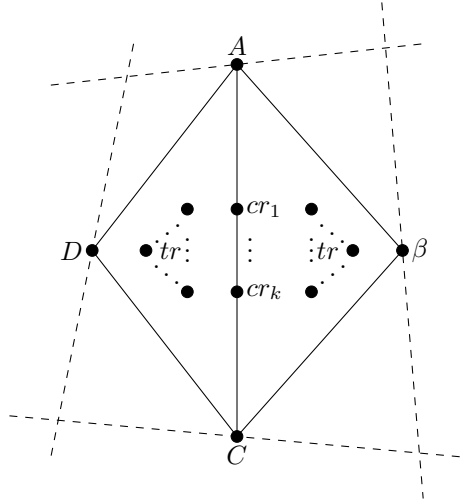
\begin{figure}[H]
\begin{tikzpicture}[scale=1.1]
    %Dashed lines
    \draw[dashed] (-2, 2) -- (2.5, 2.5);
    \draw[dashed] (-1, 2.5) -- (-2, -2.5);
    \draw[dashed] (-2.5, -2) -- (3, -2.5);
    \draw[dashed] (2, 3) -- (2.5, -3);
    %Solid lines
    \draw (0.25, 2.25) -- (2.25, 0) -- (0.25, -2.25) -- (-1.5, 0) -- (0.25, 2.25);
    \draw (0.25, 2.25) -- (0.25, -2.25);
    %Vertices
    \draw[fill=black] (0.25, 2.25) circle (2pt);
    \draw[fill=black] (0.25, -2.25) circle (2pt);
    \draw[fill=black] (2.25, 0) circle (2pt);
    \draw[fill=black] (-1.5, 0) circle (2pt);
    \node[above] at (0.25, 2.25) {$A$};
    \node[below] at (0.25, -2.25) {$C$};
    \node[left] at (-1.5, 0) {$D$};
    \node[right] at (2.25, 0) {$\beta$};
    %Coordinates
    \draw[fill=black] (0.25, 0.5) circle (2pt);
    \draw[fill=black] (0.25, -0.5) circle (2pt);
    \node[right] at (0.25, 0.5) {$cr_1$};
    \node[right] at (0.25, 0.1) {$\vdots$};
    \node[right] at (0.25, -0.5) {$cr_k$};
    \draw[fill=black] (-0.85, 0) circle (2pt);
    \draw[fill=black] (-0.35, 0.5) circle (2pt);
    \draw[fill=black] (-0.35, -0.5) circle (2pt);
    \draw[fill=black] (1.65, 0) circle (2pt);
    \draw[fill=black] (1.15, 0.5) circle (2pt);
    \draw[fill=black] (1.15, -0.5) circle (2pt);
    \node at (-0.55, 0) {$tr$};
    \node at (-0.6, 0.3) {$\iddots$};
    \node at (-0.6, -0.2) {$\ddots$};
    \node at (-0.35, 0.1) {$\vdots$};
    \node at (1.35, 0) {$tr$};
    \node at (1.4, 0.3) {$\ddots$};
    \node at (1.4, -0.2) {$\iddots$};
    \node at (1.15, 0.1) {$\vdots$};
\end{tikzpicture}
\caption{First two elements in $\pgl{k, \C}$}
\end{figure}

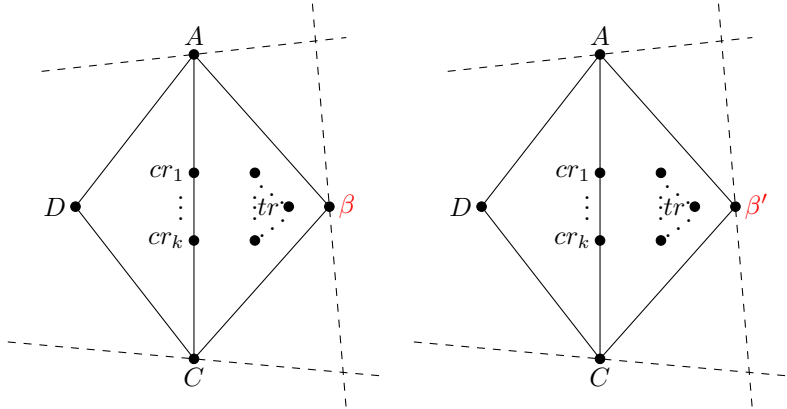
\begin{figure}[H]
\begin{tikzpicture}[scale=0.9]
    %Dashed lines
    \draw[dashed] (-2, 2) -- (2.5, 2.5);
    \draw[dashed] (-2.5, -2) -- (3, -2.5);
    \draw[dashed] (2, 3) -- (2.5, -3);
    %Solid lines
    \draw (0.25, 2.25) -- (2.25, 0) -- (0.25, -2.25) -- (-1.5, 0) -- (0.25, 2.25);
    \draw (0.25, 2.25) -- (0.25, -2.25);
    %Vertices
    \draw[fill=black] (0.25, 2.25) circle (2pt);
    \draw[fill=black] (0.25, -2.25) circle (2pt);
    \draw[fill=black] (2.25, 0) circle (2pt);
    \draw[fill=black] (-1.5, 0) circle (2pt);
    \node[above] at (0.25, 2.25) {$A$};
    \node[below] at (0.25, -2.25) {$C$};
    \node[left] at (-1.5, 0) {$D$};
    \node[right] at (2.25, 0) {$\color{red} \beta$};
    %Coordinates
    \draw[fill=black] (0.25, 0.5) circle (2pt);
    \draw[fill=black] (0.25, -0.5) circle (2pt);
    \node[left] at (0.25, 0.5) {$cr_1$};
    \node[left] at (0.25, 0.1) {$\vdots$};
    \node[left] at (0.25, -0.5) {$cr_k$};
    \draw[fill=black] (1.65, 0) circle (2pt);
    \draw[fill=black] (1.15, 0.5) circle (2pt);
    \draw[fill=black] (1.15, -0.5) circle (2pt);
    \node at (1.35, 0) {$tr$};
    \node at (1.4, 0.3) {$\ddots$};
    \node at (1.4, -0.2) {$\iddots$};
    \node at (1.15, 0.1) {$\vdots$};
\end{tikzpicture}
\hspace{0.2cm}
\begin{tikzpicture}[scale=0.9]
    %Dashed lines
    \draw[dashed] (-2, 2) -- (2.5, 2.5);
    \draw[dashed] (-2.5, -2) -- (3, -2.5);
    \draw[dashed] (2, 3) -- (2.5, -3);
    %Solid lines
    \draw (0.25, 2.25) -- (2.25, 0) -- (0.25, -2.25) -- (-1.5, 0) -- (0.25, 2.25);
    \draw (0.25, 2.25) -- (0.25, -2.25);
    %Vertices
    \draw[fill=black] (0.25, 2.25) circle (2pt);
    \draw[fill=black] (0.25, -2.25) circle (2pt);
    \draw[fill=black] (2.25, 0) circle (2pt);
    \draw[fill=black] (-1.5, 0) circle (2pt);
    \node[above] at (0.25, 2.25) {$A$};
    \node[below] at (0.25, -2.25) {$C$};
    \node[left] at (-1.5, 0) {$D$};
    \node[right] at (2.25, 0) {$\color{red} \beta'$};
    %Coordinates
    \draw[fill=black] (0.25, 0.5) circle (2pt);
    \draw[fill=black] (0.25, -0.5) circle (2pt);
    \node[left] at (0.25, 0.5) {$cr_1$};
    \node[left] at (0.25, 0.1) {$\vdots$};
    \node[left] at (0.25, -0.5) {$cr_k$};
    \draw[fill=black] (1.65, 0) circle (2pt);
    \draw[fill=black] (1.15, 0.5) circle (2pt);
    \draw[fill=black] (1.15, -0.5) circle (2pt);
    \node at (1.35, 0) {$tr$};
    \node at (1.4, 0.3) {$\ddots$};
    \node at (1.4, -0.2) {$\iddots$};
    \node at (1.15, 0.1) {$\vdots$};
\end{tikzpicture}
\caption{Further elements in $\pgl{k, \C}$}
\end{figure}

%It is clear that if $B_{\ell + 1}$ and $B_{\ell + 2}$ are both hyperbolic eigendirections, then $B_{\ell + 1}$ and $B_{\ell + 2}$ are compatible with the normalized copy of $\R P^{k-1}$ if and only if the resulting triple ratio is 

%Similarly, an elliptic eigendirection is compatible with the normalized copy of $\R P^{k-1}$ if and only if the resulting triple ratio is

This method produces $\frac{(k-2)(k-1)}{2}$ triple ratios with the flag $B$ and another $\frac{(k-2)(k-1)}{2}$ triple ratios with the flag $B'$. 

\begin{restatable}{lemma}{pglKhh}
Two hyperbolic elements $G$ and $H$ in $\pgl{k, \C}$ without repeated eigenvalues and with eigendirection flag pairs $A, C$ and $B, D$, with $A, B, C, D$ in generic position, are simultaneously conjugate into $\pgl{k, \R}$ if and only if
\begin{enumerate}
    \item $[A, B, C, D] \in \R$ for each of the $k-1$ projections
    \item $r_3 (A, B, C) \in \R$ for each of the $\frac{(k-2)(k-1)}{2}$ projections
    \item $r_3 (A, C, D) \in \R$ for each of the $\frac{(k-2)(k-1)}{2}$ projections
\end{enumerate}
\end{restatable}
\begin{proof}
This coincides exactly with when the eigendirection components and coefficients in the systems of linear equations are real with respect to a basis corresponding to the preserved projective $\R$ form. 

First we consider the $k-1$ cross ratios. Normalize so that
    $$A = \spn{e_1} \subseteq \spn{e_1} \oplus \spn{e_2} \subseteq \cdots \subseteq \spn{e_1} \oplus \cdots \oplus \spn{e_k}$$ 
and $C$ is the reverse flag. Further, normalize so that $D_1 = \spn{e_1 + e_2 + \cdots + e_k}$.

Denote the eigendirections of $H$ by $m_i$ so that 
    $$B = \spn{m_1} \subseteq \spn{m_1} \oplus \spn{m_2} \subseteq \cdots \subseteq \spn{m_1} \oplus \spn{m_2} \oplus \cdots \oplus \spn{m_k}$$ 
and $D$ is the reverse flag. If we quotient by $A_i \oplus C_j$ we have points in the quotient space consisting of points in $\C^k$ whose first $i$ components and last $j$ components are all $0$. As this is a copy of $\C^3$, denote the points simply by their $i+1, i+2,$ and $i+3$ coordinates. 
\begin{align*}
    A &\mapsto A_{i+1} / (A_i \oplus C_j) = \spn{[1, 0]} \\
    C &\mapsto C_{j+1} / (A_i \oplus C_j) = \spn{[0, 1]}
\end{align*}
and
\begin{align*}
    D &\mapsto \spn{m_1} \subseteq \spn{m_1} / (A_i \oplus C_j) \oplus \spn{m_2}) / (A_i \oplus C_j) \\
    &= \spn{[m_{(1)(i+1)}, m_{(1)(i+2)}]} = \spn{[1, 1]} \\
    B &\mapsto \spn{m_k} / (A_i \oplus C_j) = \spn{[m_{(k)(i+1)}, m_{(k)(i+2)}]}
\end{align*}

$B_1$ is in the preserved projective $\R$ form determined by $A, C,$ and $D_1$ in this basis if and only if it is in the projection of $\R^k \subseteq \C^k$, that is, if and only if the ratio of every pair of components is real. The cross ratio is giving the ratio of successive components of $B_1$, all of which are nonzero by genericity of the flags, so $B_1$ is compatible in the projective $\R$ form determined by $A, C, D_1$ if and only if all $k-1$ cross ratios are real. 

Similarly, the triple the triple ratios will determine the ratios of terms in each remaining dimension of $B$, thereby determining that the eigendirections all live in a single projective $\R$ form precisely when the collection of cross ratios and triple ratios are all real. The computations follow exactly the $\pgl{3, \C}$ case after normalizing as above. 

By projective invariance of cross ratios and triple ratios, the result follows. 

\end{proof}

\begin{restatable}{corollary}{pglKmoreH}
Given hyperbolic elements $G$ and $H$ in $\pgl{k, \C}$ which are simultaneously conjugate into $\pgl{k, \R}$, with eigendirection flag pairs $A, C$ and $B, D$ all in generic position, a hyperbolic element $M$ in $\pgl{k, \C}$, with eigendirection flags $\beta, \beta'$ in generic position with $A, C, D$, is simultaneously conjugate into $\pgl{k, \R}$ with $G$ and $H$ if and only if
\begin{enumerate}
    \item $[A, \beta, C, D] \in \R$ for each of the $k-1$ projections
    \item $[A, \beta', C, D] \in \R$ for each of the $k-1$ projections
    \item $r_3 (A, \beta, C) \in \R$ for each of the $\frac{(k-2)(k-1)}{2}$ projections
    \item $r_3 (A, \beta', C) \in \R$ for each of the $\frac{(k-2)(k-1)}{2}$ projections
\end{enumerate}
\end{restatable}
\begin{proof}
With the normalizations as in the previous lemma, this follows from lemma \ref{pglthreeHH} and corollary \ref{pglthreeHHH}. 

\end{proof}

\begin{restatable}{theorem}{pglKhhe}
Given hyperbolic elements $G$ and $H$ in $\pgl{k, \C}$ and simultaneously conjugate into $\pgl{k, \R}$, with eigendirection flag pairs $A, C$ and $B, D$ all in generic position, an elliptic or mixed type projective transformation $M$ with at most one hyperbolic eigendirection and flags $\beta, \beta'$ in generic position with $A, C, D$ will be simultaneously conjugate into $\pgl{k, \R}$ if and only if 
\begin{enumerate}
    \item $[A, \beta, C, D] = \conj{[A, \beta', C, D]}$ for each of the $k-1$ projections
    \item $r_3 (A, \beta, C) = \conj{r_3 (A, \beta', C)}$ for each of the $\frac{(k-1)(k-2)}{2}$ projections
\end{enumerate}
\end{restatable}
\begin{proof}
With the normalizations as above, this follows directly from lemma \ref{pglthreeHHE}.

\end{proof}

For projective transformations without flags of hyperbolic eigendirections in generic position, we are still able to take a collection of Fock-Goncharov coordinates and determine if a collection of corresponding elements are simultaneously conjugate into $\pgl{k, \R}$, albeit in a less direct way. The systems of linear equations that the triple ratios give will have unique solutions since the flags are assumed to be in generic position. Those solutions determine the eigendirections of each element and after determining each eigendirection we can check if they are compatible with the normalized projective $\R$ form. Then it is clear that a projective transformation is simultaneously conjugate into $\pgl{k, \R}$ if and only if all of its eigendirections are compatible with that projective $\R$ form. If we took Fock-Goncharov coordiantes directly from the eigendirection flags, we would have to solve a system of equations to determine if the collection is simultaneously conjugate into $\pgl{k, \R}$, because the quotients by asymmetrical spans of an eigendirection pairs would be disruptive. 

Instead, we produce a single, more concise theorem which gives motivation for the following section. The existence of some preserved projective $\R$ form implies the existence of some hyperbolic elements in $\pgl{k, \R}$ which preserve that projective $\R$ form. With respect to some basis for that projective $\R$ form, the hyperbolic eigendirections will be strictly real and the elliptic eigendirection pairs will be component-wise complex conjugates. The implication of this theorem is not reversible because of the generic position condition. 

\begin{restatable}{theorem}{pglK}
Given a collection of transformations in $\pgl{k, \C}$, each with $\pgl{k, \R}$ compatible non repeated eigenvalues, if there exists some flag pair $A, C$ in $\C^k$ and some eigendirection $d \in \C^k$ such that for every eigendirection $v$ from the collection of transformations we have $A, \spn{v}, C, \spn{d}$ is in generic position, then the collection is simultaneously conjugate into $\pgl{k, \R}$ if
\begin{enumerate}
    \item $[A, \spn{h}, C, \spn{d}] \in \R$ for each of the $k-1$ projections, for each hyperbolic eigendirection $h$
    \item $[A, \spn{v_1}, C, \spn{d}] = \conj{[A, \spn{v_2}, C, \spn{d}]}$ for each of the $k-1$ projections, for each elliptic eigendirection pair $v_1, v_2$
\end{enumerate}
\end{restatable}

\subsection{Using Only Cross Ratios}
In this section we expand on the idea of the previous theorem and continue to explore using only cross ratios, since their quotients are, in some sense, more directly enlightening. 

\begin{restatable}{lemma}{pglKehe}
Let $A, C$ an elliptic eigendirection flag pair of eigendirections from some transformations in $\pgl{k, \C}$ for even $k$, constructed atypically so that if $v, \conj{v}$ denotes an elliptic eigendirection pair then 
    $$A = \spn{v_1} \subseteq \spn{v_1} \oplus \spn{\conj{v_1}} \subseteq \spn{v_1} \oplus \spn{\conj{v_1}} \oplus \spn{v_2} \subseteq \spn{v_1} \oplus \spn{\conj{v_1}} \oplus \spn{v_2} \oplus \spn{\conj{v_2}} \subseteq \cdots$$
Let $v$ be any hyperbolic eigendirection from some transformation in $\pgl{k, \C}$ with the property that if $D_1 = \spn{v}$ then $A, C, D_1$ are in generic position. An elliptic eigendirection pair $\beta, \beta'$ where $A, \spn{\beta}, C, D_1$ are in generic position and $A, \spn{\beta'}, C, D_1$ are in generic position is compatible with the unique projective $\R$ form preserved by the eigendirections used to construct $A, C,$ and $D_1$ if and only if
\begin{enumerate}
    \item For every even $j$, 
\begin{align*}
    & [A^{j, k - 2 - j}, \spn{\beta}^{j, k - 2 - j}, C^{j, k - 2 - j}, D^{j, k - 2 - j}] \; \\
    & \cdot \conj{[A^{j, k - 2 - j}, \spn{\beta'}^{j, k - 2 - j}, C^{j, k - 2 - j}, D^{j, k - 2 - j}]} = 1
\end{align*} 
    \item For every odd $j$, the cross ratio 
        $$[A^{j, k - 2 - j}, \spn{\beta}^{j, k - 2 - j}, C^{j, k - 2 - j}, D^{j, k - 2 - j}]$$ 
    is equal to the complex conjugate of the product of the three cross ratios 
    \begin{align*} 
        & [A^{j-1, k - 1 - j}, \spn{\beta'}^{j-1, k - 1 - j}, C^{j-1, k - 1 - j}, D^{j-1, k - 1 - j}] \\
        & \cdot [A^{j, k - 2 - j}, \spn{\beta'}^{j, k - 2 - j}, C^{j, k - 2 - j}, D^{j, k - 2 - j}] \\
        & \cdot [A^{j+1, k - 3 - j}, \spn{\beta'}^{j+1, k - 3 - j}, C^{j+1, k - 3 - j}, D^{j+1, k - 3 - j}]
    \end{align*}
\end{enumerate}
\end{restatable}
\begin{proof}
Normalize so that 
    $$A = \spn{e_1} \subseteq \spn{e_1} \oplus \spn{e_2} \subseteq \spn{e_1} \oplus \spn{e_2} \oplus \spn{e_3} \subseteq \spn{e_1} \oplus \spn{e_2} \oplus \spn{e_3} \oplus \spn{e_4} \subseteq \cdots$$
and
    $$D_1 = \spn{e_1 + e_2 + \cdots + e_k}$$
By simple computation, the only preserved projective $\R$ form is 
    $$\set{[z_1, \conj{z_1}, z_2, \conj{z_2}, \cdots, z_{k/2}, \conj{z_{k/2}} \mid z_j \in \C \text{ not all } 0} / \sim$$
Conjugation of an elliptic eigendirection $[z_1, z_2, \cdots, z_{k-1}, z_k]$ with the conjugation coming from the preserved projective $\R$ form yields $[\conj{z_2}, \conj{z_1}, \cdots, \conj{z_k}, \conj{z_{k-1}}]$. This implies the elliptic eigendirection pair $z, w$ is compatible with the preserved projective $\R$ form if and only if there exists representatives of each eigendirection $[z_1, \cdots, z_k]$ and $[w_1, \cdots, w_k]$ such that
    $$z_1 = \conj{w_2}, z_2 = \conj{w_1}, z_3 = \conj{w_4}, z_4 = \conj{w_3}, \cdots, z_{k-1} = \conj{w_k}, z_k = \conj{w_{k-1}}$$
When quotienting by $A_j \oplus C_{k - 2 - j}$ in this normalization we are left with the $j+1$ and $j+2$ components of the space, so that the cross ratios give the ratios of successive terms. For each $A_j \oplus C_{k - 2 - j}$ with $j$ even, the cross ratio from $[A, \spn{\beta}, C, D]$ times the conjugate of that from $[A, \spn{\beta'}, C, D]$ will be
    $$[A, \spn{\beta}, C, D] \cdot \conj{[A, \spn{\beta'}, C, D]} = \left( \frac{z_{j+1}}{z_{j+2}} \right) \cdot \conj{\left( \frac{w_{j+1}}{w_{j+2}} \right)}$$
which simplifies to
    $$\left( \frac{z_{j+1}}{z_{j+2}} \right) \cdot \conj{\left(\frac{\conj{z_{j+2}}}{\conj{z_{j+1}}} \right)} = \left( \frac{z_{j+1}}{z_{j+2}} \right) \cdot \left(\frac{z_{j+2}}{z_{j+1}} \right) = 1$$
when the eigendirections are compatible with the preserved projective $\R$ form. Thus, $[A, \spn{\beta}, C, D] \cdot \conj{[A, \spn{\beta'}, C, D]}$ must be $1$ for the conjugation to be consistent with $e_1 \leftrightarrow e_2, e_3 \leftrightarrow e_4, \cdots, e_{k-1} \leftrightarrow e_k$. This is analogous to the $\pgl{2, \C}$ case where circle inversion in $\hat{\C}$ about a circle centered at the origin will preserve rays from the origin, so an elliptic eigendirection pair must live on a single ray from the origin. In particular, this is analagous to the case of circle inversion about the circle of radius 1 centered at the origin, since there is a hyperbolic eigendirection at $1 \in \hat{\C}$. On the other hand, if $[A, \spn{\beta}, C, D] \cdot \conj{[A, \spn{\beta'}, C, D]} = 1$, then we have $\left( \frac{z_{j+1}}{z_{j+2}} \right) \cdot \conj{\left( \frac{w_{j+1}}{w_{j+2}} \right)} = 1$ and thus $z_{2j+1} = \conj{w_{2j+2}}, z_{2j+2} = \conj{w_{2j+1}}$ with respect to some representatives of $z$ and $w$. Note that this is for an individual pair, but does not on its own guarantee that a representative exists where every pair of components $2j+1, 2j+2$ is simultaneously of this form. 

The quotient of $z$ by $A_j \oplus C_{k - 2 - j}$ gives $[z_{j+1}, z_{j+2}]$ and so the cross ratio $[A, \spn{z}, C, D]$ coming from $A_j \oplus C_{k - 2 - j}$ is $\frac{z_{j+1}}{z_{j+2}}$. The quotients by $A_{j-1} \oplus \C_{k - 1 - j}, A_j \oplus C_{k - 2 - j}$, and $A_{j+1} \oplus C_{k - 3 - j}$ give cross ratios $[A, \spn{w}, C, D]$ of $\frac{w_{j}}{w_{j+1}}, \frac{w_{j+1}}{w_{j+2}}$, and $\frac{w_{j+2}}{w_{j+3}}$ respectively, so the product of these cross ratios is $\frac{w_{j}}{w_{j+3}}$. It then follows that the elliptic eigendirection pair admits representatives such that
    $$z_1 = \conj{w_2}, z_2 = \conj{w_1}, z_3 = \conj{w_3}, z_4 = \conj{w_4}, \cdots, z_{k-1} = \conj{w_k}, z_k = \conj{w_{k-1}}$$
if and only if the two given conditions hold and therefore that the elliptic eigendirection pair is compatible with the preserved projective $\R$ form if and only if the two given conditions hold. 

\end{proof}

\begin{example}
\textit{Note:} A potential obstruction arises if two terms of $z$ are rotated from the corresponding conjugate terms in $w$, that is, if the second condition is not included. For example, consider the eigendirection pairs in $\C P^3$
\begin{align*}
    [1+i, 1-i, 3+i, 3+3i] &, [1+i, 1-i, 3-3i, 3-i] \\
    [2, 2i, 3i, -3] &, [\frac{\sqrt{2}}{2} (1+i), \frac{\sqrt{2}}{2} (1-i), -3, -3i]
\end{align*}
The first pair does respect the conjugation from our chosen normalization, but the second does not. This is despite the fact that the first condition holds, because the first two components of the last eigendirection were rotated by $\pi/4$ while the last two components were not. The second condition prevents this. 
\end{example}

\begin{restatable}{lemma}{pglKehh}
Let $A, C$ an elliptic eigendirection flag pair of eigendirections from some transformations in $\pgl{k, \C}$ for even $k$, constructed atypically so that if $v, \conj{v}$ denotes an elliptic eigendirection pair then 
    $$A = \spn{v_1} \subseteq \spn{v_1} \oplus \spn{\conj{v_1}} \subseteq \spn{v_1} \oplus \spn{\conj{v_1}} \oplus \spn{v_2} \subseteq \spn{v_1} \oplus \spn{\conj{v_1}} \oplus \spn{v_2} \oplus \spn{\conj{v_2}} \subseteq \cdots$$
Let $v$ be any hyperbolic eigendirection from some transformation in $\pgl{k, \C}$ with the property that if $D_1 = \spn{v}$ then $A, C, D_1$ are in generic position. A hyperbolic eigendirection $\beta$ where $A, \spn{\beta}, C, D_1$ are in generic position is compatible with the unique projective $\R$ form preserved by the eigendirections used to construct $A, C,$ and $D_1$ if and only if
\begin{enumerate}
    \item For each even $j$, $[A^{j, k-2-j}, \spn{\beta}^{j, k-2-j}, C^{j, k-2-j}, D^{j, k-2-j}] \in S^1$.
    \item For every even $j$
\begin{align*}
    & \arg [A^{j, k-2-j}, \spn{\beta}^{j, k-2-j}, C^{j, k-2-j}, D^{j, k-2-j}] \; + \\
    & 2 \arg[A^{j+1, k-3-j}, \spn{\beta}^{j+1, k-3-j}, C^{j+1, k-3-j}, D^{j+1, k-3-j}] \; + \\
    & \arg [A^{j+2, k-4-j}, \spn{\beta}^{j+2, k-4-j}, C^{j+2, k-4-j}, D^{j+2, k-4-j}]
\end{align*}
is an integral multiple of $2 \pi$. 
\end{enumerate}
\end{restatable}
\begin{proof}
Normalize $A, C, D_1$ as in the previous lemma so that again the only potential preserved projective $\R$ form is
    $$\set{[z_1, \conj{z_1}, z_2, \conj{z_2}, \cdots, z_{k/2}, \conj{z_{k/2}} \mid z_j \in \C \text{ not all } 0} / \sim$$
As before, the preserved projective $\R$ forms of a hyperbolic transformation are precisely those through the eigendirections, so it is clear that the hyperbolic eigendirections are compatible with the preserved projective $\R$ form if and only if they each admit representatives of the form
    $$\set{[z_1, \conj{z_1}, z_2, \conj{z_2}, \cdots, z_{k/2}, \conj{z_{k/2}} \mid z_j \in \C \text{ not all } 0}$$
If such a representative exists, then since $z /\conj{z} \in S^1$, we have that the first condition is necessary. As in the previous lemma, the second condition prevents rotation in an individual pair of components separate from the remaining components and is also necessary for the same reason. 

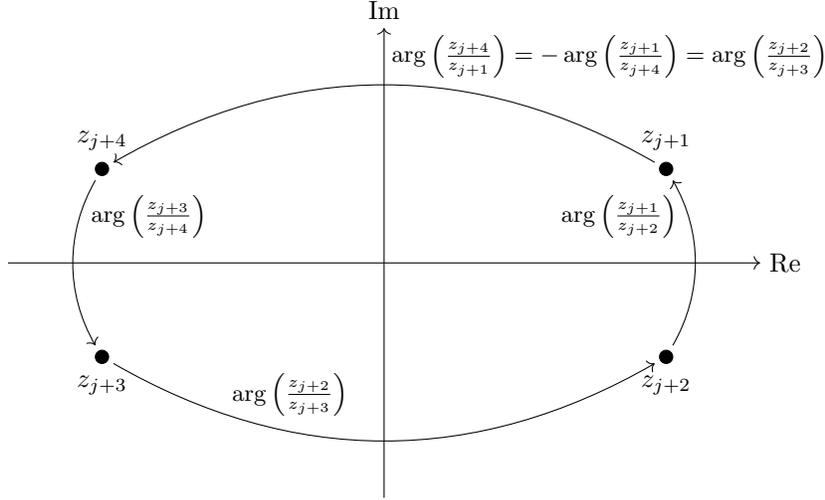
\begin{figure}[H]
    \begin{tikzpicture}[scale=1.25]
        \draw[->] (-4, 0) -- (4, 0);
        \node[right] at (4, 0) {Re};
        \draw[->] (0, -2.5) -- (0, 2.5);
        \node[above] at (0, 2.5) {Im};
        \draw[fill=black] (3, 1) circle (2pt);
        \node[above] at (3, 1.1) {$z_{j+1}$};
        \draw[fill=black] (3, -1) circle (2pt);
        \node[below] at (3, -1.1) {$z_{j+2}$};
        \draw[bend right=30, ->, shorten <=5pt, shorten >= 5pt] (3, -1) to (3, 1);
        \draw[fill=black] (-3, 1) circle (2pt);
        \node[above] at (-3, 1.1) {$z_{j+4}$};
        \draw[fill=black] (-3, -1) circle (2pt);
        \node[below] at (-3, -1.1) {$z_{j+3}$};
        \draw[bend right=30, ->, shorten <=5pt, shorten >= 5pt] (-3, 1) to (-3, -1);

        % Additional arrows (horizontal)
        \draw[bend right=30, ->, shorten <=5pt, shorten >= 5pt] (3, 1) to (-3, 1);
        \draw[bend right=30, ->, shorten <=5pt, shorten >= 5pt] (-3, -1) to (3, -1);

        % Top and bottom labels
        \node at (-1, -1.4) {\small$\arg{\left(\frac{z_{j+2}}{z_{j+3}} \right)}$};
        \node at (2.4, 2.2) {\small$\arg{\left(\frac{z_{j+4}}{z_{j+1}} \right)} = -\arg{\left(\frac{z_{j+1}}{z_{j+4}} \right)} = \arg{\left(\frac{z_{j+2}}{z_{j+3}} \right)}$};

        % Side labels
        \node at (-2.5, 0.5) {\small$\arg{\left(\frac{z_{j+3}}{z_{j+4}} \right)}$};
        \node at (2.5, 0.5) {\small$\arg{\left(\frac{z_{j+1}}{z_{j+2}} \right)}$};
    \end{tikzpicture}
    \caption{Consecutive pairs of terms simultaneously rotating to be complex conjugate pairs}
\end{figure}

When the first condition is satisfied, the argument of the cross ratio coming from the quotient by $A_{j} \oplus \C_{k - 2 - j}$ gives the angle from $z_{j+1}$ to $z_{j+2}$. Similarly, the argument of the cross ratio coming from the quotient by $A_{j+2} \oplus \C_{k - 4 - j}$ gives the angle from $z_{j+3}$ to $z_{j+4}$. As their ratio is in $S^1$, individually the pairs $z_{j+1}, z_{j+2}$ and $z_{j+3}, z_{j+4}$ can be rotated so that they represent a complex conjugate pair. The pair can be simultaneously rotated to be a complex conjugate pair, meaning there exists a representative where both pairs are a complex conjugate pair, if and only if the angle from $z_{j+2}$ to $z_{j+3}$ is equal to the negative of the angle from $z_{j+1}$ to $z_{j+4}$, or equivalently, equal to the angle from $z_{j+4}$ to $z_{j+1}$. Since the argument of a complex fraction is unchanged when scaling either term by a positive real number, this occurs precisely when the sum of arguments as described results in an integer number of complete rotations. This continues throughout the components of the eigendirection. 

Therefore, the two conditions are satisfied if and only if the hyperbolic eigendirection is compatible with the preserved projective $\R$ form. 

\end{proof}

\begin{restatable}{theorem}{pglKmh}
Let $A, C$ be any eigendirection flag pair in $\C P^{k-1}$ constructed so that if $v_j, \conj{v_j}$ denotes elliptic eigendirection pairs and $h_j$ denotes hyperbolic eigendirections then 
\begin{align*}
    A &= \spn{v_1} \subseteq \spn{v_1} \oplus \spn{\conj{v_1}} \\
    & \subseteq \spn{v_1} \oplus \spn{\conj{v_1}} \oplus \spn{v_2} \subseteq \spn{v_1} \oplus \spn{\conj{v_1}} \oplus \spn{v_2} \oplus \spn{\conj{v_2}} \subseteq \cdots \\
    & \subseteq \spn{v_1} \oplus \spn{\conj{v_1}} \oplus \spn{v_2} \oplus \spn{\conj{v_2}} \oplus \cdots \oplus \spn{v_m} \oplus \spn{\conj{v_m}} \oplus \spn{h_1} \subseteq \cdots \\
    & \subseteq \spn{v_1} \oplus \spn{\conj{v_1}} \oplus \spn{v_2} \oplus \spn{\conj{v_2}} \oplus \cdots \oplus \spn{v_m} \oplus \spn{\conj{v_m}} \oplus \spn{h_1} \oplus \cdots \spn{h_n}
\end{align*}
and let $h$ be any hyperbolic eigendirection in $\C P^{k-1}$ with the property that if $D_1 = \spn{h}$ then $A, C, D_1$ are in generic position. A hyperbolic eigendirection $\beta$ where $A, \spn{\beta}, C, D_1$ are in generic position is compatible with the unique projective $\R$ form preserved by the eigendirections used to construct $A, C, D_1$ if and only if
\begin{enumerate}
    \item For each even $j$ less than the number of elliptic eigendirections in flag $A$, $[A^{j, k-2-j}, \spn{\beta}^{j, k-2-j}, C^{j, k-2-j}, D^{j, k-2-j}] \in S^1$.
    \item For every even $j$ less than the number of elliptic eigendirections in $A$
\begin{align*}
    & \arg [A^{j, k-2-j}, \spn{\beta}^{j, k-2-j}, C^{j, k-2-j}, D^{j, k-2-j}] \; + \\
    & 2 \arg[A^{j+1, k-3-j}, \spn{\beta}^{j+1, k-3-j}, C^{j+1, k-3-j}, D^{j+1, k-3-j}] \; + \\
    & \arg [A^{j+2, k-4-j}, \spn{\beta}^{j+2, k-4-j}, C^{j+2, k-4-j}, D^{j+2, k-4-j}]
\end{align*}
is an integral multiple of $2 \pi$. 
    \item For each $j$ greater than or equal to the number elliptic eigendirections in $A$, $[A^{j, k-2-j}, \spn{\beta}^{j, k-2-j}, C^{j, k-2-j}, D^{j, k-2-j}] \in \R$.
    \item When $j$ is two less than the number of elliptic eigendirections in $A$ and $k-2-j$ is equal to the number of hyperbolic eigendirections in $C$, the argument of the cross ratio 
\begin{align*}
    & \arg [A^{j, k - 2 - j}, \spn{\beta}^{j, k - 2 - j}, C^{j, k - 2 - j}, D^{j, k - 2 - j}] \; + \\
    & 2 \arg [A^{j + 1, k - 3 - j}, \spn{\beta}^{j + 1, k - 3 - j}, C^{j + 1, k - 3 - j}, D^{j + 1, k - 3 - j}]
\end{align*}
    is an integer multiple of $2 \pi$.  
\end{enumerate}
An elliptic eigendirection pair $\beta, \beta'$ where $A, \spn{\beta}, C, D_1$ are in generic position and $A, \spn{\beta'}, C, D_1$ are in generic position is compatible with the unique projective $\R$ form preserved by the eigendirections used to construct $A, C, D_1$ if and only if
\begin{enumerate}
    \item For each even $j$ less than the number of elliptic eigendirections in the flag $A$, 
\begin{align*}
    & [A^{j, k - 2 - j}, \spn{\beta}^{j, k - 2 - j}, C^{j, k - 2 - j}, D^{j, k - 2 - j}] \; \\
    & \cdot \conj{[A^{j, k - 2 - j}, \spn{\beta'}^{j, k - 2 - j}, C^{j, k - 2 - j}, D^{j, k - 2 - j}]} = 1
\end{align*} 
    \item For every odd $j$ less than the number of elliptic eigendirections in the flag $A$, the cross ratio 
        $$[A^{j, k - 2 - j}, \spn{\beta}^{j, k - 2 - j}, C^{j, k - 2 - j}, D^{j, k - 2 - j}]$$ 
    is equal to the complex conjugate of the product of the three cross ratios 
    \begin{align*} 
        & [A^{j-1, k - 1 - j}, \spn{\beta'}^{j-1, k - 1 - j}, C^{j-1, k - 1 - j}, D^{j-1, k - 1 - j}] \\
        & \cdot [A^{j, k - 2 - j}, \spn{\beta'}^{j, k - 2 - j}, C^{j, k - 2 - j}, D^{j, k - 2 - j}] \\
        & \cdot [A^{j+1, k - 3 - j}, \spn{\beta'}^{j+1, k - 3 - j}, C^{j+1, k - 3 - j}, D^{j+1, k - 3 - j}]
    \end{align*}
    \item For each $j$ greater than or equal to the number of elliptic eigendirections in the flag $A$,
\begin{align*}
    & [A^{j, k - 2 - j}, \spn{\beta}^{j, k - 2 - j}, C^{j, k - 2 - j}, \spn{d}^{j, k - 2 - j}] \\
    &= \conj{[A^{j, k - 2 - j}, \spn{\beta'}^{j, k - 2 - j}, C^{j, k - 2 - j}, \spn{d}^{j, k - 2 - j}]}
\end{align*}
    \item The cross ratio 
        $$[A^{j, k - 2 - j}, \spn{\beta'}^{j, k - 2 - j}, C^{j, k - 2 - j}, D^{j, k - 2 - j}]$$
    equals the complex conjugate of the product of cross ratios
\begin{align*}
    & [A^{j-1, k - 1 - j}, \spn{\beta}^{j-1, k - 1 - j}, C^{j-1, k - 1 - j}, D^{j-1, k - 1 - j}] \\
    &\cdot [A^{j, k - 2 - j}, \spn{\beta}^{j, k - 2 - j}, C^{j, k - 2 - j}, D^{j, k - 2 - j}]
\end{align*}
    when $j$ is one less than the number of elliptic eigendirections in $A$ and $k - 2 - j$ is one less than the number of hyperbolic eigendirections in $C$.
\end{enumerate}
\end{restatable}
\begin{proof}
Normalize so that 
    $$A = \spn{e_1} \subseteq \spn{e_1} \oplus \spn{e_2} \subseteq \cdots \subseteq \spn{e_1} \oplus \spn{e_2} \oplus \cdots \oplus \spn{e_k}$$ 
and 
    $$D_1 = \spn{e_1 + e_2 + \cdots + e_k}$$ 
as usual. Then the unique potential preserved projective $\R$ form determined by $A, C, D_1$ is 
    $$\set{[z_1, \conj{z_1}, \cdots, z_m, \conj{z_m}, r_1, r_2, \cdots, r_n] \mid z_j \in \C, r_j \in \R, \text{ with } z_j, r_j \text{ not all 0}} / \sim$$
A hyperbolic eigendirection is compatible with this preserved projective $\R$ form if and only if it admits a representative in the projective $\R$ form. An elliptic eigendirection pair is compatible with this preserved projective $\R$ form if and only if it is swapped by the corresponding conjugation. 

The conditions for the cross ratios $[A, \spn{\beta}, C, D_1]$ coming from quotients by $A_j \oplus C_{k - 2 - j}$ for $j$ greater than or equal to the number of elliptic eigendirections in the flag $A$ are the same as the case where the entire eigendirection flag pair $A, C$ come from hyperbolic eigendirections and so follow from previous lemmas. Similarly, the conditions for the cross ratios $[A, \spn{\beta}, C, D_1]$ coming from quotients by $A_j \oplus C_{k - 2 - j}$ for $k - 2 - j$ greater than or equal to the number of hyperbolic eigendirections in the flag $C$ are the same as the case where the entire eigendirection flag pair $A, C$ come from elliptic eigendirections and so follow from previous lemmas. 

The overlap conditions, those corresponding to quotients into the codimension two subspace spanned by one hyperbolic and one elliptic eigendirection from the flag $A$, are alignment conditions for both the hyperbolic eigendirections and elliptic eigendirections. In the hyperbolic case, without the overlap condition we have that the components in the span of hyperbolic eigendirections in the flag $A$ can be collectively rotated separately from the components in the span of the elliptic eigendirections in the flag $A$. To be in the preserved projective $\R$ form there must be a representative where the portion in the span of the hyperbolic eigendirections is real and the components in the span of the elliptic eigendirections is of the form $[z_1, \conj{z_1}, z_2, \conj{z_2}, \cdots]$. This occurs precisely when the angle from the first component in the span of a hyperbolic eigendirection from the flag $A$ to the last component in the span of an elliptic eigendirection from the flag $A$ is equal to the angle from the last component in the span of an elliptic eigendirection from the flag $A$ to the first component in the span of a hyperbolic eigendirection from the flag $A$, or equivalently, when the angles as described above sum to an integer multiple of $2 \pi$. 

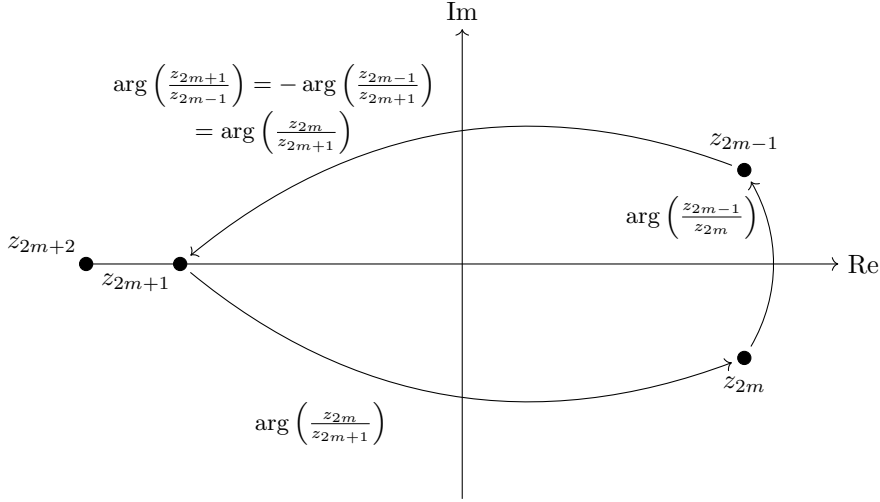
\begin{figure}[H]
    \begin{tikzpicture}[scale=1.25]
        \draw[->] (-4, 0) -- (4, 0);
        \node[right] at (4, 0) {Re};
        \draw[->] (0, -2.5) -- (0, 2.5);
        \node[above] at (0, 2.5) {Im};
        \draw[fill=black] (3, 1) circle (2pt);
        \node[above] at (3, 1.1) {$z_{2m-1}$};
        \draw[fill=black] (3, -1) circle (2pt);
        \node[below] at (3, -1.1) {$z_{2m}$};
        \draw[bend right=30, ->, shorten <=5pt, shorten >= 5pt] (3, -1) to (3, 1);
        \draw[fill=black] (-4, 0) circle (2pt);
        \node[above left] at (-4, 0) {$z_{2m+2}$};
        \draw[fill=black] (-3, 0) circle (2pt);
        \node[below left] at (-3, 0) {$z_{2m+1}$};

        % Additional arrows (horizontal)
        \draw[bend right=30, ->, shorten <=5pt, shorten >= 5pt] (3, 1) to (-3, 0);
        \draw[bend right=30, ->, shorten <=5pt, shorten >= 5pt] (-3, 0) to (3, -1);

        % Top and bottom labels
        \node at (-1.5, -1.7) {\small$\arg{\left(\frac{z_{2m}}{z_{2m+1}} \right)}$};
        \node at (-2.0, 1.9) {\small$\arg{\left(\frac{z_{2m+1}}{z_{2m-1}} \right)} = -\arg{\left(\frac{z_{2m-1}}{z_{2m+1}} \right)}$};
        \node at (-2.0, 1.4) {\small$= \arg{\left(\frac{z_{2m}}{z_{2m+1}} \right)}$};

        % Side labels
        \node at (2.45, 0.5) {\small$\arg{\left(\frac{z_{2m-1}}{z_{2m}} \right)}$};
    \end{tikzpicture}
    \caption{Overlap condition for hyperbolic eigendirection}
\end{figure}

The overlap condition for elliptics is most easily seen in a more algebraic manner. An elliptic eigendirection pair is compatible with the preserved projective $\R$ form if and only if it admits representatives 
    $$[z_1, z_2, \cdots, z_{2m}, z_{2m+1}, \cdots, z_k] \text{ and } [w_1, w_2, \cdots, w_{2m}, w_{2m+1}, \cdots, w_k]$$
such that
    $$z_1 = \conj{w_2}, z_2 = \conj{w_1}, z_3 = \conj{w_4}, z_4 = \conj{w_3}, \cdots, z_{2m-1} = \conj{w_{2m}}, z_{2m} = \conj{w_{2m-1}}$$
and
    $$z_{2m+1} = \conj{w_{2m+1}}, z_{2m+2} = \conj{w_{2m+2}}, \cdots, z_k = \conj{w_k}$$
As before, the conditions for cross ratios involving only the first $2m$ components follow from previous lemmas where the flag pair $A, C$ consists of all elliptic eigendirections and the conditions for cross ratios involving none of the first $2m$ components follow from previous lemmas where the flag pair $A, C$ consists of all hyperbolic eigendirections. 

Then in particular, assuming the conditions on the cross ratios from quotients by $A_j \oplus C_{k - 2 - j}$ for $j \neq 2m-1$, the elliptic eigendirection pair is compatible with the preserved projective $\R$ form if and only if
    $$z_{2m-1} = \conj{w_{2m}}, z_{2m} = \conj{w_{2m-1}}$$
and
    $$z_{2m+1} = \conj{w_{2m+1}}, z_{2m+2} = \conj{w_{2m+2}}$$
Since we can normalize one component of each of $z$ and $w$, this is equivalent to
    $$\frac{z_{2m-1}}{z_{2m}} \cdot \frac{z_{2m}}{z_{2m+1}} = \frac{z_{2m-1}}{z_{2m+1}} = \frac{\conj{w_{2m}}}{\conj{w_{2m+1}}}$$
which is equivalent to the cross ratio condition as stated. 

\end{proof}

If all of the eigendirections of some collection of transformations $\set{M_j}$ in $\pgl{k, \C}$, each with $\pgl{k, \R}$ compatible non-repeated eigenvalues, are compatible with the unique projective $\R$ form preserved by some subset of eigendirections from the transformations, then the collection is simultaneously conjugate into $\pgl{k, \R}$. If any of the eigendirections fail to be compatible with such a unique preserved projective $\R$ form, then the collection is not simultaneously conjugate into $\pgl{k, \R}$. In this way, we can use cross ratios and triple ratios to determine when a representation into $\pgl{k, \C}$ consisting of elements without repeat eigenvalues and with eigendirection flags satisfying certain genericity conditions is simultaneously conjugate into a representation into $\pgl{k, \R}$. 

We conclude the paper with short notes on the genericity conditions and going between Fock-Goncharov coordiantes and cross ratio only coordinates. 

\subsection{Note on Genericity Conditions}
Given a flag pair $A, C$ and the one dimensional part of a flag $D_1$ with $A, C, D_1$ in generic position, we again note that the genericity condition of individual eigendirections from a projective transformation $M$ being in generic position with $A, C, D_1$ does not imply there exists an eigendirection flag pair $\beta, \beta'$ for $M$ with $A, \beta, C, \beta'$ in generic position, nor vice versa. This genericity condition is distinct from the existence of flags in generic position. 

\begin{example}
Let $A, C$ be an eigendirection flag pair coming from $e_1, e_2, e_3$ as usual and let $D_1$ be the span of the eigendirection $e_1 + e_2 + e_3$. Let $M$ be another element in $\pgl{3, \C}$ with eigendirections $v_1 = [1, 2, 1], v_2 = [1, 2, -1], v_3 = [-1, 2, 1]$.

The eigendirections satisfy $A, \spn{v_j}, C, D_1$ in generic position for each $j$. This is easy to verify since 
\begin{align*}
    A_1 \oplus C_1 &= \spn{e_1} \oplus \spn{e_3} \\
    A_1 \oplus D_1 &= \spn{e_1} \oplus \spn{e_2 + e_3} \\
    C_1 \oplus D_1 &= \spn{e_3} \oplus \spn{e_1 + e_2} \\
    A_2 &= \spn{e_1} \oplus \spn{e_2} \\
    C_2 &= \spn{e_3} \oplus \spn{e_2}
\end{align*}
and since each eigendirection $v_j$ is nonzero in every component and has second component not equal to either the first or third components. 

On the other hand, since $\spn{v_1} \oplus \spn{v_2} \cap C_1 = C_1$ and $\spn{v_1} \oplus \spn{v_3} \cap A_1 = A_1$ it is not possible to construct an eigendirection flag pair $\beta, \beta'$ from $v_1, v_2, v_3$ such that $A, \beta, C, \beta'$ is in generic position. 

\end{example}

\begin{example}
Let $A, C$ be an eigendirection flag pair coming from $e_1, e_2, e_3$ as usual and let $D_1$ be the span of the eigendirection $e_1 + e_2 + e_3$. Let $M$ be another element in $\pgl{3, \C}$ with eigendirections $v_1 = [1, 1, -1], v_2 = [1, -1, 1], v_3 = [2, -1, 2]$.

The eigendirections do not satisfy $A, \spn{v_j}, C, D_1$ being in generic position because $\left( D_1 \oplus C_1 \right) \cap \spn{v_1} = \spn{v_1} \neq \set{\vec{0}}$. 

On the other hand, if $\beta = \spn{v_1} \subseteq \spn{v_1} \oplus \spn{v_2} \subseteq \spn{v_1} \oplus \spn{v_2} \oplus \spn{v_3}$ and $\beta'$ is its flag pair then the flags $A, \beta, C, \beta'$ are in generic position. 
\end{example}

\subsection{Dictionary between Fock-Goncharov Coordinates and New Coordinates}
Going between Fock-Goncharov coordinates and our new coordinates is straightforward, because we have a dictionary between Fock-Goncharov coordiantes and the eigendirections of the projective transformations (with respect to some normalization) and a dictionary between those eigendirections and our new coordinates. Given a collection of either coordinates, we construct the eigendirections in some chosen basis and and then compute the coordinates of the other type. This works in both directions and does not depend on our choice of basis since every coordinate used is invariant under action by $\pgl{,\C}$.

%Puts every item in myrefs.bib in the bibliography, regardless of if cited
\nocite{*}
%Generate the bibliography using data in 'myrefs.bib' and citations in document
\bibliography{main}

\begin{thebibliography}{1}

\bibitem{Conrad:Complexification}
Keith Conrad.
\newblock Complexification.
\newblock online notes at: \url{https://kconrad.math.uconn.edu/blurbs/linmultialg/complexification.pdf}.

\bibitem{Fock:2006}
Vladimir Fock and Alexander Goncharov.
\newblock Moduli spaces of local systems and higher {T}eichm\"{u}ller theory.
\newblock {\em Publ. Math. Inst. Hautes \'{E}tudes Sci.}, (103):1--211, 2006.

\bibitem{Hitchin:1992}
N.~J. Hitchin.
\newblock Lie groups and {T}eichm\"{u}ller space.
\newblock {\em Topology}, 31(3):449--473, 1992.

\bibitem{Labourie:2006}
Fran\c{c}ois Labourie.
\newblock Anosov flows, surface groups and curves in projective space.
\newblock {\em Invent. Math.}, 165(1):51--114, 2006.

\bibitem{Palesi:2013}
Frederic Palesi.
\newblock Introduction to positive representations and fock-goncharov coordinates.
\newblock {\em HAL}, 2013.
\newblock hal-01218570.

\bibitem{Wienhard:2018}
Anna Wienhard.
\newblock An invitation to higher {T}eichm\"{u}ller theory.
\newblock In {\em Proceedings of the {I}nternational {C}ongress of {M}athematicians---{R}io de {J}aneiro 2018. {V}ol. {II}. {I}nvited lectures}, pages 1013--1039. World Sci. Publ., Hackensack, NJ, 2018.

\end{thebibliography}

\end{document}